\pdfoutput=1
\documentclass[11pt]{amsart}
\usepackage{amsmath}
\usepackage{amssymb}
\usepackage{mathtools}

\usepackage{hyperref} 
\hypersetup{colorlinks=true, linkcolor=blue, citecolor = blue, urlcolor=blue, filecolor=magenta}

\usepackage{enumerate}
\usepackage{enumitem}
\usepackage{slashed}
\usepackage{caption}
\usepackage{subcaption}
\usepackage{newlfont}
\usepackage{amsrefs}
\usepackage{comment}
\usepackage[normalem]{ulem}
\usepackage{accents}
\usepackage{leftidx}

\usepackage{harpoon}
\usepackage[pdftex]{graphicx}
\usepackage{esint}
\usepackage{relsize}
\usepackage[latin1]{inputenc}
\usepackage{xcolor}

\definecolor{brass}{rgb}{0.71, 0.65, 0.36}

\usepackage{tikz-cd}
\theoremstyle{plain}
\newtheorem{theorem}{Theorem}[section]
\newtheorem{lemma}[theorem]{Lemma}
\newtheorem{prop}[theorem]{Proposition}

\newtheorem{proposition}[theorem]{Proposition}

\theoremstyle{definition}
\newtheorem{remark}[theorem]{Remark}
\newtheorem{definition}[theorem]{Definition}

\numberwithin{equation}{section}

\DeclareMathOperator{\sgn}{sgn}

\theoremstyle{plain}

\numberwithin{equation}{section}

\allowdisplaybreaks

\newcommand{\td}{\mathrm{d}}

\begin{document}

\title[A Comparison Theorem For the Mass of ALE and ALF Toric 4-Manifolds]
{A Comparison Theorem For the Mass of ALE and ALF Toric 4-Manifolds}

\author[Alaee]{Aghil Alaee}
\address{
Department of Mathematics, Clark University, Worcester, MA 01610, USA}
\email{aalaeekhangha@clarku.edu}

\author[Khuri]{Marcus Khuri}
\address{Department of Mathematics\\
Stony Brook University\\
Stony Brook, NY 11794, USA}
\email{marcus.khuri@stonybrook.edu}

\author[Kunduri]{Hari Kunduri}
\address{Department of Mathematics and Statistcs \& Physics and Astronomy, McMaster University, Hamilton, ON Canada}
\email{kundurih@mcmaster.ca}


\thanks{A. Alaee acknowledges the support of NSF Grant DMS-2316965. M. Khuri acknowledges the support of NSF Grants DMS-2104229 and DMS-2405045. H. Kunduri acknowledges the support of NSERC Grant RGPIN-2025-06027.}

\begin{abstract} 
We establish sharp lower bounds for the mass of asymptotically locally Euclidean (ALE) and asymptotically locally flat (ALF) toric 4-manifolds, in terms of equilibrium geometries consisting of gravitational instantons. More precisely, the mass of a complete ALE or ALF toric 4-manifold with nonnegative scalar curvature is bounded below by a sum comprised of the following quantities: the mass of the corresponding toric gravitational instanton having the same orbit space (rod) structure as the original ALE/ALF manifold, and an expression determined by the conical angle defects of totally geodesic 2-spheres within the instanton that serve as generators for its second homology. The inequality may be generalized to the situation in which the ALE/ALF manifold also possesses conical singularities as well as orbifold singularities, and it suggests a refined notion of `total mass' in which the result simply states that the total mass of the ALE/ALF manifold is not less than that of the corresponding gravitational instanton. Furthermore, we prove rigidity for these statements, namely the inequality is saturated only when the ALE/ALF manifold is Ricci flat and in fact agrees with the corresponding instanton. These results may be viewed in the context of positive mass theorems, providing an explanation of how positivity can fail in the ALE/ALF setting. Moreover, the main theorem may be interpreted as yielding a variational characterization of the relevant toric gravitational instantons.
\end{abstract}

\maketitle
\section{Introduction} 
\label{sec1} \setcounter{equation}{0}
\setcounter{section}{1}


The positive mass theorem is a central achievement in the study of scalar curvature and mathematical relativity, and was originally established for asymptotically Euclidean (AE) manifolds with nonnegative scalar curvature by Schoen-Yau \cite{schoen-yau1979} and Witten \cite{Witten}. Various incarnations of this theorem have been found in a variety of other settings. In the asymptotically hyperbolic case, important contributions were made by Andersson-Cai-Galloway \cite{ACG}, Chru\'{s}ciel-Herzlich \cite{ChruscielHerzlich}, Wang \cite{XWang}, and Zhang \cite{XZhang}. Extensions to the asymptotically locally hyperbolic setting were obtained by Alaee-Hung-Khuri \cite{AHK}, Brendle-Hung \cite{BrendleHung}, and Lee-Neves \cite{LeeNeves}, while the complex hyperbolic case was treated by Herzlich \cites{Herzlich1,Herzlich2}. In this article we will be concerned with ALE and ALF manifolds which, in particular, arise naturally in the study of \textit{gravitational instantons} --- complete, noncompact, Ricci flat 4-dimensional Riemannian manifolds with square-integrable curvature.

\begin{definition}\label{ALEdef}(ALE Manifold)  
A connected and complete Riemannian 4-manifold $(M,g)$ is said to be \textit{asymptotically locally Euclidean} (ALE) if there exists a compact set $K\subset M$, a finite subgroup $\mathcal{G}\subset O(4)$ acting freely on coordinate spheres,
and a diffeomorphism $\varphi:(\mathbb{R}^4 \setminus \overline{B}_1)/\mathcal{G} \rightarrow M\setminus K$ such that
\begin{equation}
|\mathring{\nabla}^l(\varphi^* g -b)|_b=O(r^{-1-\kappa-l}),\quad\quad l=0,1,2,
\end{equation}
for some $\kappa>0$ where $b$ is the flat cone metric on $(\mathbb{R}^4 \setminus \overline{B}_1)/\mathcal{G}$ with radial distance function $r$, and $\mathring{\nabla}$ denotes covariant differentiation with respect to $b$. Moreover, the scalar curvature of $g$ is required to be integrable, $R_g \in L^{1}(M)$. If $\mathcal{G}$ is trivial, then the manifold is called \textit{asymptotically Euclidean} (AE).
\end{definition}

\begin{remark}
The definition of an ALE manifold, and that of an ALF manifold given below, often allow for multiple ends. Although the results of this paper may be generalized
to include more than one end, for simplicity of presentation this will not be pursued here. Note also that the model metric may be expressed in polar form
$b=dr^2 +r^2 b_s$, where $b_s$ is the metric of constant $+1$ curvature on the radial cross-section $\mathcal{S}=S^3 /\mathcal{G}$. Moreover, regularity of the metric is left unspecified here and in the definition of an ALF manifold below, since mild singularities will eventually be included.
\end{remark}

Motivated by questions in quantum gravity, the validity of the positive mass theorem was also conjectured for the ALE setting.  However, the Eguchi-Hanson manifold \cite{EH} may be  observed to violate the rigidity statement, and LeBrun \cite{LeBrun} generalized their construction to find an infinite family of explicit counterexamples to the inequality. 
Nevertheless, positivity of mass in the ALE case has been established under additional hypotheses, notably for certain K\"{a}hler manifolds by Hein-LeBrun \cite{HeinLeBrun} and under suitable spin-structure matching conditions by Dahl \cite{Dahl} and Deruelle-Ozuch \cite{DO}. 

\begin{definition}[ALF Manifold \text{\cite[Definition 1]{BGL}}] \label{def:ALF} 
A connected and complete Riemannian 4-manifold $(M,g)$ will be called \textit{asymptotically locally flat} (ALF) if the following conditions are satisfied.
\begin{enumerate}[label=(\roman*)]
\item There is a compact subset $K\subset M$ and a diffeomorphism $\varphi: \mathbb{R}_+ \times \mathcal{S} \rightarrow M\setminus K$, where $\mathcal{S}$ is a closed 3-manifold finitely covered by $S^1 \times S^2$ or $S^3$. If $\mathcal{S}=S^1 \times S^2$ then the manifold is called \textit{asymptotically flat} (AF)\footnote{It should be noted that the terminology of an asymptotically flat manifold has taken on two inequivalent meanings within the context of the positive mass theorem, one arising from mathematical relativity and another from the study of gravitational instantons. In this article we will only use the latter notion.}.

\item On $\mathcal{S}$ there is a 1-form $\tau$ and a vector field $T$ such that $\iota_{T}\tau=1$ and $\mathcal{L}_{T}\tau=0$, where $\iota$ denotes interior product and $\mathcal{L}$ indicates Lie differentiation.

\item On $\mathcal{S}$ there is also a positive-semi-definite symmetric 2-tensor $\gamma$ such that $\mathcal{L}_{T}\gamma=0$, $\mathrm{ker}\!\text{ }\gamma=\mathrm{span}\!\text{ }T$, and which locally defines a metric of Gauss curvature $+1$ on the space of leaves of the foliation tangent to $T$.

\item $\mathbb{R}_+ \times\mathcal{S}$ is equipped with a model metric
\begin{equation}\label{def-model}
b =dr^2 +r^2 \gamma + \ell^2 \tau^2,
\end{equation}
where $r$ parameterizes $\mathbb{R}_+$ and $\ell>0$ is a constant.  

\item After pulling back via diffeomorphism $\varphi$, the metric $g$ asymptotes to the model with decay
\begin{equation}\label{ALFdecay}
|\mathring{\nabla}^l (\varphi^* g-b)|_{b}=O(r^{-\frac{1}{2}-\kappa-l}), \quad\quad l=0,1,2,
\end{equation}
where $\kappa>0$ and $\mathring{\nabla}$ denotes covariant differentiation with respect to $b$.

\item The scalar curvature of $g$ is integrable, $R_g \in L^1(M)$.
\end{enumerate}
\end{definition}


\begin{remark}
An important special case of an ALF manifold occurs when $\mathcal{S}$ is an $S^1$-bundle over $S^2$, as exemplified by the Taub-NUT gravitational instanton in which the relevant bundle is the Hopf fibration of $S^3$. Such an ALF manifold, in which the bundle has Euler number $e$, is referred to as ALF-$A_k$ where $k =-e - 1$. A notable case of AF manifolds occurs when the model geometry is flat and arises as the quotient $(\mathbb{R}^3 \times \mathbb{R})/\mathbb{Z}$, where the generator of the action is given by a rotation of angle $2\pi\beta \ell$ in $\mathbb{R}^3$ and a translation by distance $2\pi \ell$ in $\mathbb{R}$, for some real number $\beta$; such manifolds are referred to as AF$_{\beta \ell}$.
When $\beta \ell$ is irrational the vector field $T$ does not have closed orbits, which happens for generic members of the Kerr and Chen-Teo families of AF instantons. 
\end{remark}

Associated with each ALE or ALF manifold $(M,g)$ is a well-defined notion of mass. In analogy with the classical AE setting, its expression is derived from the flux of the linearized scalar curvature operator. In particular, the mass is defined by
\begin{equation}\label{mass.def}
\mathrm{mass}_{b} (M,g):=\lim_{r\to\infty}\frac{1}{4\pi}\int_{\mathcal{S}_r}\left(\mathrm{div}_{b}\text{ }\! \mathbf{e}-d\mathrm{Tr}_{b}\text{ }\! \mathbf{e}\right)\!(\partial_r) d\mathcal{V},\qquad \mathbf{e}=\varphi^*g-b,
\end{equation}
where $d\mathcal{V}$ is the volume form of the $r$-level set $\mathcal{S}_r$ in the model geometry. The fact that the mass is a geometric invariant in the ALE setting follows from the proof of Bartnik \cite{bartnik1986} and Chru\'{s}ciel \cite{Chrusciel} in the AE regime. Geometric invariance with respect to the choice of ALF structure
is established in \cite[Proposition 2.8]{KhuriWang} in the ALF-$A_k$ case. 
An earlier statement of this type for AF$_0$ asymptotics was made by Minerbe \cite[pg. 952]{minerbe}; in \cite[Proposition 6]{minerbe} an analogous result was proven for the so called `Gauss-Bonnet mass' (tailored to Ricci curvature) in the ALF-$A_k$ context. Note that in the AF$_0$ setting, the mass \eqref{mass.def} agrees with, up to normalization, those of \cites{BCH,CLSZ,Dai,LSZ} as well as \cite[Theorem 2]{minerbe}. Furthermore, we point out that in the AF$_{\beta\ell}$ case, the property of geometric invariance of the mass does not appear to have been addressed in the literature. Accordingly, in this case Theorem \ref{main.theorem} should be understood with respect to the prescribed AF$_{\beta\ell}$ coordinate structure.

The mass may be viewed as a geometric invariant that connects scalar curvature with the global geometry and topology of the manifold. The Euclidean Reissner-Nordstr\"{o}m metrics on $\mathbb{R}^2 \times S^2$ are AF$_0$, complete, and scalar flat, but they can have negative mass for certain choices of parameters. Similarly, the charged Taub-Bolt Einstein-Maxwell instanton is an example of a complete ALF-$A_0$ manifold with zero scalar curvature, that admits negative mass for certain ranges of parameters. These examples, as well as others, are discussed in detail in Section \ref{sec8}. Thus, as in the ALE setting, the positive mass theorem dramatically fails for ALF manifolds. On the other hand, like Dahl's result \cite{Dahl} in the ALE spin case, Minerbe \cite[Theorem 2]{minerbe} considered AF$_0$ manifolds with nonnegative scalar curvature and a matching condition for the spin structure at infinity, to establish a positive mass theorem. In a different direction, Liu-Shi-Zhu \cite[Theorem 1.2]{LSZ} (see also Chen-Liu-Shi-Zhu \cite[Theorem 1.8]{CLSZ}) show that for AF$_0$ manifolds of dimensions less than 8 the positive mass theorem holds if the circle at infinity is homotopically nontrivial. Moreover, Khuri-Wang \cite[Theorem 1.2]{KhuriWang} obtain the same conclusion in the AF$_0$ setting under the hypothesis that a codimension-two coordinate sphere in the asymptotic end is trivial within the homology of $M$. 
Related results were additionally found by Dai \cite{Dai}, Dai-Sun \cite{DaiSun}, and Barzegar-Chru\'{s}ciel-H\"{o}rzinger \cite{BCH}.
In the case of ALF-$A_k$ manifolds, Minerbe \cite[Theorem 1]{minerbe} proved positivity of mass under the assumption of nonnegative Ricci curvature; note that Minerbe's definition of mass in this result does not coincide with the standard one. While in the same setting, Khuri-Wang \cite[Theorem 1.7]{KhuriWang} establish a positive mass lower bound in terms of bundle degree with the hypotheses of nonnegative scalar curvature and an almost free $U(1)$ action.  Furthermore, an ALF-$D_2$ positive mass theorem was established by Kim-Ozuch \cite[Theorem 0.5]{KO}.

Despite this progress, positivity properties and more generally geometric inequalities involving the mass remain poorly understood in the ALE and ALF regimes. The purpose of the present paper is to investigate to what extent a positive mass style theorem can be achieved in the presence of such robust counterexamples, as described above. More precisely, we seek a result that `applies to' such counterexamples, rather than avoiding them with exclusionary hypotheses, in the hope of understanding the reason for the presence of negative mass.
In this regard, we shall restrict attention to toric ALE and ALF manifolds, and will show that a full and satisfactory answer to this question may be given in this setting.

\begin{definition}\label{toricdef}
A \textit{toric ALE or toric ALF manifold} is an ALE or ALF manifold that admits an effective isometric $T^2$ action,
which is compatible with the ALE or ALF structure in the following sense. Let $\eta_1$, $\eta_2$ denote Killing field generators of this action.
\begin{enumerate}[label=(\roman*)]
\item 
In the ALE case, it asymptotes to an effective isometric $T^2$ action on the model flat cone geometry of the end, which preserves radial cross-sections. Moreover, there exist generators $\xi_1$, $\xi_2$ of the limiting action such that
\begin{equation}
|\mathring{\nabla}^l \left((\varphi^{-1})_{*}\eta_i -\xi_i \right)|_{b}=O(r^{-\kappa-l}),
\end{equation}
for $l=0,1$, $i=1,2$.

\item In the ALF case, it asymptotes to an effective isometric $T^2$ action on the model geometry of the end, which preserves the radial cross-sections and leaves $\tau$ as well as $\gamma$ invariant. Moreover, there exist generators $\xi_1$, $\xi_2$ of the limiting action such that $T=a^i \xi_i $ where $a^i\in\mathbb{R}$ and
\begin{equation}
|\mathring{\nabla}^l \left((\varphi^{-1})_{*}\eta_i -\xi_i \right)|_{b}=O(r^{\frac{1}{2}-\kappa-l}),\quad\quad
|\mathring{\nabla}^l \left((\varphi^{-1})_{*}(a^i\eta_i) -T \right)|_{b}=O(r^{-\frac{1}{2}-\kappa-l}),
\end{equation}
for $l=0,1$, $i=1,2$.
\end{enumerate}
\end{definition}

\begin{remark}
The compatibility condition (i) for ALE manifolds is actually a consequence of the first statement in Definition \ref{toricdef} concerning the existence of a torus action \cite{JaraczKhuri}. Furthermore as explained in Section \ref{sec2} below, in both settings the toric condition implies that the radial cross-section $\mathcal{S}$ must be a lens space, or alternatively in the ALF case, $S^1 \times S^2$.
\end{remark}

Toric symmetries play an important role in the study of gravitational instantons \cites{BiquardGauduchon,MingyangLi}, and have been used for AE mass lower bounds \cites{AKK1,AKK2,AKK3,AlaeeYau}. In a dramatic recent development Li-Sun \cite{LiSun} have discovered toric AF instantons on infinitely many new diffeomorphism types of 4-manifolds, which are not locally Hermitian. It is then natural to consider positive mass theorems in the ALE and ALF settings which assume this symmetry, and to expect that toric instantons are featured in an essential way. To this point, we recall that among AE manifolds with zero scalar curvature the critical points of the mass are Ricci flat \cite[Proposition 7.1]{CBFM}, and note that the same conclusion holds in the ALE and ALF contexts. Therefore, the search for a global minimizer of the mass among such manifolds (or more generally those with nonnegative scalar curvature) leads inevitably to gravitational instantons. In fact, loosely speaking, an indication as to why the positive mass theorem fails in the ALE and ALF settings and on the other hand is valid in the AE setting, is that there are many Ricci flat manifolds with the same ALE/ALF structure (possibly with conical singularities) but there is only one Ricci flat AE manifold, namely Euclidean space. 

The variational approach suggests that a meaningful replacement for the positive mass theorem in the current setting should take the form of a rigid mass comparison result between ALE and ALF manifolds of nonnegative scalar curvature, and certain gravitational instantons. In order to realize such a concept, it is necessary to have a mechanism to produce a large variety of gravitational instantons with prescribed structure. Toric symmetry provides a robust solution to this problem by reducing the Ricci flat equations to an axisymmetric harmonic map into the hyperbolic plane with prescribed singularities. More precisely, it follows from \cites{KKRW,KunduriLucietti,LiSun, Kunduri:2026xvc} that given a rod data set 
--- an embellished orbit space boundary that characterizes the toric action --- 
and the ALE/ALF asymptotic structure, there exists a unique corresponding harmonic map giving rise to a toric gravitational instanton admitting these properties; the method is explained in Section \ref{secexistence}. We refer to this instanton as an \textit{equilibrium geometry}, and note that it typically will have conical singularities on the axes, see Section \ref{sec2}. For this reason, our main result below is naturally stated with conical angle defects. As described in more detail in the next section, a rod data set consists of a collection of intervals called axis rods whose union is the boundary of the orbit space $M/T^2$, and an associated collection of `weights' that detail the degeneration of the torus action; the intervals are parameterized by a coordinate $z$.

\begin{theorem}\label{main.theorem}
Let $(M,g)$ be a simply connected toric ALE, toric ALF-$A_k$, or toric AF$_{\beta\ell}$ manifold with nonnegative scalar curvature, possibly having conical singularities and corners\footnote{See Definition \ref{conedef}.} along finite axes. Consider the corresponding toric gravitational instanton $(M,g_o)$ sharing the same asymptotic ALE or ALF structure, and the same rod data set consisting of intervals 
$\{\Gamma_n\}_{n=1}^{N+1}$. Then
\begin{equation}\label{PMTinequality}
\mathrm{mass}_b(M,g)- \mathrm{mass}_b(M,{g}_o)\geq 2\pi\sum_{n=1}^{N+1}\int_{\Gamma_n}(\pmb{\vartheta}^n -\pmb{\vartheta}^n_o) dz,
\end{equation}
where $\pmb{\vartheta}^n$ and $\pmb{\vartheta}^n_o$ are the logarithmic angle defects on axis rod $\Gamma_n$ with respect to $g$ and $g_o$, respectively. Furthermore, equality holds if and only if $(M,g)$ is isometric to the Ricci flat equilibrium geometry $(M,g_o)$.
\end{theorem}

The precise meaning of conical singularities in the context of toric ALE/ALF manifolds, as well as their logarithmic angle defects, will be given in the next section. When conical singularities are not present, the inequality \eqref{PMTinequality} simply states that the mass of the given ALE/ALF manifold is bounded below by the mass of its corresponding instanton equilibrium geometry. Even in the general case when conical singularities are present, we are motivated to define a new mass 
\begin{equation}
\mathbf{mass}_b(M,g):=\mathrm{mass}_b(M,g)-2\pi\sum_{n=1}^{N+1}\int_{\Gamma_n}\pmb{\vartheta}^n dz,
\end{equation}
and again the result yields the simple statement that $\mathbf{mass}_b(M,g)\geq\mathbf{mass}_b(M,g_o)$. In this way, up to multiplication by $2\pi$, the total logarithmic angle defect along $\Gamma_n$ may be viewed as the mass of the rod, or rather the mass of the corresponding totally geodesic 2-sphere lying within $M$. As is shown in Section \ref{sec8}, these contributions are responsible for and explain the negative mass present in the Reissner-Nordstr\"{o}m manifolds, since the equilibrium geometry associated with these examples must possess conical defects.
It should also be noted that Theorem \ref{main.theorem} will continue to hold if conical singularities are present on the semi-infinite axes, as long as the difference $\pmb{\vartheta}^n -\pmb{\vartheta}^n_o$ is integrable on such rods. Moreover, the difference of masses on the left-hand side of \eqref{PMTinequality} is equivalent to $\mathrm{mass}_{g_o}(M,g)$, the mass of $g$ with respect to the Ricci flat background $g_o$, and thus the main inequality may then be viewed as giving a sharp lower bound for this interpretation of mass. In fact, the concept of using Ricci flat backgrounds to define the mass in ALF contexts has previously been put forward by Kim-Ozuch \cite[Introduction]{KO}. Finally we mention that in the ALE case, it follows from Bando-Kasue-Nakajima \cite{bando1989construction} that the instanton metric will fall-off at order 4, which implies that $\mathrm{mass}_b(M,{g}_o)=0$ and hence simplifies the inequality \eqref{PMTinequality}.

In order to illustrate a delicate aspect of this theorem, and the necessity of including some hypothesis beyond fixing the asymptotic structure, we may compare with the stability result of Dahl-Kr\"{o}ncke \cite[Theorem 1.8]{DahlKroncke}. Recall that an open Einstein manifold is called \textit{linearly unstable} if the linearized Ricci operator, restricted to compactly supported transverse-traceless tensors, has a negative bottom of the spectrum. When this occurs for a gravitational instanton, Dahl-Kr\"{o}ncke show that there exist compactly supported perturbations of the instanton metric which have nonnegative scalar curvature that is not identically zero. This perturbation may then be conformally changed back to zero scalar curvature, as in Schoen-Yau's \cite{schoen-yau1979} approach to the AE positive mass theorem, while preserving the asymptotics and resulting in a smaller mass than the original instanton. Thus, if applied to a Schwarzschild AF instanton $(\mathbb{R}^2 \times S^2,g_{sc})$ which is linearly unstable \cite[Example 1.13]{DahlKroncke}, while taking care to preserve toric symmetry in the deformation, we obtain a new toric AF manifold $(\mathbb{R}^2 \times S^2,\tilde{g}_{sc})$ of nonnegative scalar curvature with $\mathrm{mass}_b(\mathbb{R}^2 \times S^2,\tilde{g}_{sc})<\mathrm{mass}_b(\mathbb{R}^2 \times S^2,g_{sc})$. This appears to violate inequality \eqref{PMTinequality}, if we naively use the original Schwarzschild instanton as the equilibrium geometry. However, in order to apply Theorem \ref{main.theorem}, the equilibrium geometry must be chosen to have the same rod data set as the perturbation. Ultimately, the deformation disturbs the rod lengths, so that a new Schwarzschild instanton with different mass must be used for the comparison.

This paper is organized as follows. In Section \ref{sec2} we derive consequences of the toric action, and analyze the asymptotic model geometries. In Section \ref{sec3}, scalar curvature identities are exploited to obtain a relation between a reduced harmonic energy and certain flux integrals. Convexity properties of the reduced energy are studied in Section \ref{sec4}, and then used to produce a gap lower bound. Section \ref{sec5} is dedicated to showing that the flux integrals yield the desired difference of masses, while the main theorem is proved in Section \ref{sec6}. Asymptotics at infinity, the axes, and corners are derived in Section \ref{sec7:asymptotics}.
Finally, several examples are detailed in Section \ref{sec8} and an appendix is included to record miscellaneous calculations and formulae. 

\section{Background and Setup} 
\label{sec2} \setcounter{equation}{0}
\setcounter{section}{2}

Let $(M,g)$ be a simply connected toric ALE or toric ALF manifold. It follows from \cite{OrlikR} and the proof of \cite[Proposition 3]{HollandsY} that the orbit space $M/T^2$ is diffeomorphic to a half-plane $\{(\rho,z)\mid \rho\geq 0, z\in\mathbb{R}\}$, with certain `weights' embellishing the boundary. More precisely, the $z$-axis $\Gamma$ is decomposed into an exhaustive sequence of closed intervals referred to as \textit{axis rods} and denoted by
\begin{equation}
\Gamma_1 = (-\infty,z_1], \quad \Gamma_2 = [z_1,z_2], \text{ }\dots \text{ }\Gamma_{N} = [z_{N-1}, z_{N}], \quad \Gamma_{N+1} = [z_{N},\infty),
\end{equation}
where $z_n < z_{n+1}$, such that the interior of each $\Gamma_n$ corresponds to points in $M$ with 1-dimensional isotropy subgroup. The intersection point of two adjacent axis rods is called a \textit{corner} and represents a point in $M$ with 2-dimensional isotropy subgroup, whereas all interior points of the half-plane correspond to principal orbits. Note that in the gravitational instanton literature, corners are referred to as \textit{nuts} and the collection of orbits over an axis rod is referred to as a \textit{bolt}; these latter objects are totally geodesic 2-spheres in $M$ that represent the generators of its second homology.

Let $\eta_1$, $\eta_2$ denote Killing field generators of the $T^2$-action and consider the Gram matrix $G$ with components $G_{ij}=g(\eta_i, \eta_j)$. Associated with each $\Gamma_n$ is an element $\mathbf{v}_n=(v_n^1 , v_n^2)\in \mathbb{Z}^2$ called a \textit{rod structure}, whose components are relatively prime, and which generates the kernel of $G$ on this rod or equivalently the Killing field $v_n^1 \eta_1 +v_n^2 \eta_2$ vanishes on $\Gamma_n$. The collection $\mathcal{R}=\{(\mathbf{v}_n,\Gamma_n)\}_{n=1}^{N+1}$ of axis rods and their rod structures is referred to as the \textit{rod data set}, and completely encodes the topology of $M$, see \cite{KMWY} for further discussion. For instance, by examining the torus fibration over a semi-circle in the orbit space that connects the two semi-infinite rods $\Gamma_{1}$ and $\Gamma_{N+1}$, we find that the only allowable cross-sectional topologies \cite[Proposition 2]{HollandsY} for the asymptotic end of $M$ are $S^1 \times S^2$ and the lens spaces $L(p,q)$; these two cases occur for the pair of rod structures $\{\left((1,0),\Gamma_1\right),\left((1,0),\Gamma_{N+1}\right)\}$ and $\{\left((1,0),\Gamma_1\right),\left((q,p),\Gamma_{N+1}\right)\}$, respectively. Moreover, it will be assumed that any rod data set satisfies the \emph{admissibility condition} at corner points:
\begin{equation}
\det \begin{pmatrix} v_n^1 & v_n^2 \\ v_{n+1}^1 & v_{n+1}^2 \end{pmatrix} = \pm 1. 
\end{equation} 
This condition preserves the manifold structure in a neighborhood of corner points. Without it, such neighborhoods admit an orbifold structure, and although our results and proofs should continue to hold in that situation we will not pursue this direction here.

Conical singularities and corners arise naturally in this setting. In particular, let $\pi:M\rightarrow M/T^2$ be the quotient map, then conical singularities can occur along the axes or rather the 2-sphere bolts $\pi^{-1}(\Gamma_n)$. Consider a model cone metric on $D^2 \times S^1 \times (0,1)$ given by
\begin{equation}\label{conem}
g_{\mathrm{cone}}=ds^2 +c^2 s^2 (d\psi^1 +a d\psi^2)^2 +g_{\mathrm{cyl}}(y),
\end{equation}
where $c=c(t)>0$ and $a=a(t)$ are smooth functions, $(s,\psi^1)$ are polar coordinates on the open unit disk $D^2$, and $g_{\mathrm{cyl}}(y)$ is a metric on the cylinder $(0,1)\times S^1$ parameterized by coordinates $y=(t,\psi^2)$. Here $\psi^1,\psi^2$ are $2\pi$-periodic and their coordinate vector fields generate a $T^2$ action by isometries. Note that $s=0$ corresponds to the axis, and that $c=e^{-\pmb{\vartheta}}$ where $\pmb{\vartheta}$ is the \textit{logarithmic angle defect} of each cone along the axis. Consider also a flat model corner metric on a 4-dimensional ball $B_1$ in polar-Hopf coordinates $(r,\theta,\psi^1, \psi^2)$ given by
\begin{equation}\label{metriccornercone}
g_{\mathrm{corner}}=dr^2 +r^2 \left(d\theta^2 +c_1^2\sin^2 \theta (d\psi^1)^2 +c_2^2 \cos^2 \theta (d\psi^2)^2 \right),
\end{equation}
where $c_1 , c_2 \in (0,\infty)$ are constants, and $r\in[0,1)$, $\theta\in[0,\pi/2]$, while $\psi^1, \psi^2$ are again $2\pi$-periodic. The values $c_1$ and $c_2$ yield angle defects of the neighboring axes in the usual way.

\begin{definition}\label{conedef}
We say that a toric ALE or toric ALF manifold $(M,g)$ possesses \textit{conical singularities and corners} if the metric is globally $L^{\infty}$ and smooth away from the axes $\pi^{-1}(\Gamma)$, with the following two types of model asymptotics.
\begin{enumerate}[label=(\roman*)]
\item At each point of any bolt $\pi^{-1}(\mathrm{int }\text{ }\! \Gamma_n)$ there exists a neighborhood of the form $D^2 \times S^1 \times (0,1)$ with coordinates $(s,\psi^1,y)$ and an associated cone metric $g_{\mathrm{cone}}$ such that after pullback
\begin{equation}
|\mathring{\nabla}^l (g-g_{\mathrm{cone}})|_{g_{\mathrm{cone}}}=O(s^{1+\zeta-l}),\quad\quad l=0,1,
\end{equation}
for some $\zeta> 0$ where $\mathring{\nabla}$ denotes covariant differentiation with respect to $g_{\mathrm{cone}}$.

\item At each corner (nut) point $\pi^{-1}(\Gamma_{n}\cap\Gamma_{n+1})$ there exists a 4-ball neighborhood with coordinates $(r,\theta,\psi^1,\psi^2)$ and an associated corner metric $g_{\mathrm{corner}}$ such that after pullback
\begin{equation}
|\mathring{\nabla}^l (g-g_{\mathrm{corner}})|_{g_{\mathrm{corner}}}=O(r^{2-l}),\quad\quad l=0,1,
\end{equation}
where $\mathring{\nabla}$ denotes covariant differentiation with respect to $g_{\mathrm{corner}}$.
\end{enumerate}
In both cases, after a pushforward, the coordinate vector fields $\partial_{\psi^i}$, $i=1,2$ correspond to generators of
the $T^2$ action on $M$.
\end{definition}

\subsection{Asymptotic model geometries}\label{modelgeom}
The toric hypothesis places strong restrictions on the asymptotic model geometries of ALE and ALF manifolds. In fact we will show that the model metric $b$ can be assumed to take an explicit form, which falls into one of three types that we now describe.

\subsubsection{Asymptotically locally Euclidean (ALE)}
Consider the asymptotic end $(\mathbb{R}^4 \setminus \overline{B}_1)/\mathbb{Z}_p =(1,\infty)\times L(p,q)$ where $p,q\in\mathbb{Z}$ are relatively prime with $p\geq 1$, equipped with the flat metric
\begin{equation}\label{ALEmodel}
b_{\text{ALE}} = d r^2 + r^2 \left[d\theta^2 + p^{-2}\sin^2\theta (d\psi^1)^2 + \cos^2\theta \left( d\psi^2 + p^{-1}q d\psi^1\right)^2 \right], 
\end{equation} 
in which $r>1$ and $(\theta,\psi^1,\psi^2)$ are Hopf coordinates on the lens space with $\theta\in[0,\pi/2]$ and $\psi^1, \psi^2$ are $2\pi$-periodic. Here the toric symmetry is generated by the Killing fields $\partial_{\psi^1}$, $\partial_{\psi^2}$, and the semi-infinite rod structures are given by $\mathbf{v}_1 =(0,1)$ and $\mathbf{v}_{N+1}=(p,-q)$.

\subsubsection{Asymptotically locally flat (ALF-$A_{k-1}$)} \label{2.1.2}
Let $k\in\mathbb{Z}_{+}$ and consider the asymptotic end $\mathbb{R}_+ \times L(k,1)$ equipped with the metric 
\begin{equation}\label{bALF}
b_{\text{ALF}} = \ell^2 \left(d \psi^2 + k\cos^2\left(\theta/2\right) d \psi^1 \right)^2 + d r^2 + r^2 \left(d \theta^2 + \sin^2\theta (d\psi^1)^2 \right), 
\end{equation} 
where $r>0$, $\theta \in [0,\pi]$, and $\psi^1, \psi^2$ are $2\pi$-periodic. The induced metric on radial level sets exhibits the lens space as an $S^1$-bundle over the 2-sphere with Euler number $e=-k$, and thus this model geometry is associated with type ALF-$A_{k-1}$. By multiplying the second and third terms of \eqref{bALF} by $h(r)=1 + \frac{k \ell}{2r}$ and multiplying the first term by $h(r)^{-1}$ as in \eqref{hmetrics} one obtains a new metric $\tilde{b}_{ALF}$ which is Ricci flat, and coincides in the case of $k=1$ with the Taub-NUT gravitational instanton on $\mathbb{R}^4$ after adding the origin point $r=0$.  The toric symmetry is again generated by the Killing fields $\partial_{\psi^1}$, $\partial_{\psi^2}$, and the semi-infinite rod structures are given by $\mathbf{v}_1 =(1,0)$ and $\mathbf{v}_{N+1}=(-1,k)$. Note that the metric \eqref{bALF} may be placed into the context of 
Definition \ref{def:ALF} (iv) by setting 
\begin{equation}
T =  \ell^{-1}\partial_{\psi^2}, \qquad \tau = \ell \left(d\psi^2 + k\cos^2 \left(\theta/2 \right) d\psi^1\right),
\qquad \gamma = d\theta^2 + \sin^2\theta (d\psi^1)^2.
\end{equation} 
We may also consider the case when $k=p/q$ for relatively prime positive integers $p,q$. In this situation the radial level sets have topology $L(p,q)$, however the metric $b_{ALF}$ admits conical singularities on the axis rod $\Gamma_{N+1}$ when $q\neq 1$.

\subsubsection{Asymptotically flat (AF$_{\beta\ell}$)} 
Consider the asymptotic end $\mathbb{R}_+ \times S^1 \times S^2$ equipped with the flat metric 
\begin{equation}\label{bAF}
b_{\text{AF}} = \ell^2(d \psi^2)^2 + d r^2 + r^2 \left(d \theta^2 +  \sin^2 \theta (d \psi^1 + \beta \ell d \psi^2)^2 \right),
\end{equation} 
where $\ell>0$, $\beta\geq 0$ are constants and $r>0$, $\theta\in[0,\pi]$, and $\psi^1, \psi^2$ are $2\pi$-periodic. The radial level sets are topologically $S^1 \times S^2$ as realized by the rod structures $\mathbf{v}_1 =(1,0)$ and $\mathbf{v}_{N+1}=(1,0)$
on the semi-infinite rods, associated with the toric symmetry generated by the Killing fields $\partial_{\psi^1}$, $\partial_{\psi^2}$. This model geometry is of the type AF$_{\beta \ell}$. Note that the metric \eqref{bAF} may be placed into the context of 
Definition \ref{def:ALF} (iv) by setting 
\begin{equation}\label{fonapoitng}
T =  \ell^{-1}\left(\partial_{\psi^2}-\beta\ell \partial_{\psi^1}\right), \qquad \tau =\ell d\psi^2,
\qquad \gamma = d\theta^2 + \sin^2\theta (d \psi^1 + \beta\ell d \psi^2)^2 .
\end{equation} 
Observe that $T$ has closed orbits if and only if $\beta \ell$ is rational.

\begin{remark}\label{faihjfpjaoh}
We have chosen to distinguish the AF and ALF cases since many explicit examples fall into one of these two classes described above. However, they can be treated together as members of a larger toric family of Ricci flat geometries, possibly with conical singularities. Namely, using the previous notation consider the following metric 
\begin{equation}\label{hmetrics}
\tilde{b} =h^{-1}\ell^2 \left(d \psi^2 + k \cos^2\left(\theta/2\right) (d\psi^1+\beta\ell d\psi^2) \right)^2 +  h\left(d r^2 + r^2 \left(d \theta^2 + \sin^2\theta (d\psi^1+\beta\ell d\psi^2)^2\right)\right), 
\end{equation} 
where $h(r)$ is the radial function from Section \ref{2.1.2}. If $k=0$ this reduces to the setting of $b_{AF}$, so assume that $k>0$. In order to have the structure of a manifold in the vicinity of rod $\Gamma_{N+1}$, we require that $k^{-1}(1+k\beta\ell)=\frac{q}{p}$ for some relatively prime integers $p\neq 0$ and $q$. In this case, the semi-infinite rod structures are $\mathbf{v}_1 = (1,0)$ and $\mathbf{v}_{N+1} = (-q, p)$, and the topology of the asymptotic end on which this metric is defined is given by $\mathbb{R}_+\times L(p,q)$. Moreover, conical singularities occur on $\Gamma_{N+1}$ unless $k=p$ and $1+k\beta\ell=q$. 
\end{remark}

\begin{remark}\label{faihjfpjaoh1}
When $h$ is replaced by $1$ in \eqref{hmetrics}, we will refer to the resulting metric as $\tilde{b}_{ALF}$.  Although the results of this paper continue to hold for the more general model metric $\tilde{b}_{ALF}$, for simplicity of exposition we will mostly restrict attention to the less embellished version $b_{ALF}$ from \eqref{bALF}.  
\end{remark}

We will now show that in the toric setting, the model geometries must take one of the above three forms up to negligible error.

\begin{proposition}
Let $(M,g)$ be a toric ALE or ALF manifold.
\begin{enumerate}[label=\normalfont(\roman*)]
\item In the ALE case, the asymptotic model geometry is of the form $(\mathbb{R}_+\times L(p,q),b)$ for some relatively prime integers $p\geq 1$, $q$, and there exists an explicit model metric \eqref{ALEmodel} such that
\begin{equation}
|\mathring{\nabla}^{l}(b-b_{ALE})|_{b_{ALE}}=O(r^{-2-l}),\quad\quad l=0,1,2,
\end{equation}
where $\mathring{\nabla}$ denotes covariant differentiation with respect to $b_{ALE}$. 

\item In the ALF/non-AF case, the asymptotic model geometry is of the form $(\mathbb{R}_+\times L(p,q),b)$ for some relatively prime positive integers $p$, $q$, the bounded Killing field $T$ has closed orbits, and there exists an explicit model metric from Remark \ref{faihjfpjaoh1} with $k=p$ and $1+p\beta\ell=q$ such that
\begin{equation}
|\mathring{\nabla}^{l}(b-\tilde{b}_{ALF})|_{\tilde{b}_{ALF}}=O(r^{-1-l}),\quad\quad l=0,1,2,
\end{equation}
where $\mathring{\nabla}$ denotes covariant differentiation with respect to $\tilde{b}_{ALF}$. 

\item In the AF case, the asymptotic model geometry is of the form $(\mathbb{R}_+\times S^1 \times S^2,b)$, the bounded Killing field $T$ may not have closed orbits, and there exists an explicit model metric \eqref{bAF} such that
\begin{equation}
|\mathring{\nabla}^{l}(b-b_{AF})|_{b_{AF}}=O(r^{-1-l}),\quad\quad l=0,1,2,
\end{equation}
where $\mathring{\nabla}$ denotes covariant differentiation with respect to $b_{AF}$.
\end{enumerate}
\end{proposition}

\begin{proof}
We will treat the ALF and AF cases here, and simply note that the ALE case may be proved similarly.
According to Definition \ref{toricdef} there is an effective isometric $T^2$ action on the model geometry of the end, which induces a toric symmetry on its radial cross-sections $\mathcal{S}$. By \cite[Section 2]{OrlikR} the orbit space $\mathcal{S}/T^2$ is a closed interval which we may parameterize by $\theta\in[0,\pi]$. Let $\psi^1, \psi^2$ be $2\pi$-periodic coordinates parameterizing the torus fibers, then by expressing the cross-section metrics in Riemannian submersion format we find that
\begin{equation}\label{afonaoirnpioqy}
b|_{\mathcal{S}}=r^2 \gamma +\ell^2 \tau^2 =A(r)d\theta^2 +B_{ij}(r,\theta)d\psi^i d\psi^j,
\end{equation}
for some coefficient functions $A$ and $B_{ij}$ which are independent of the torus coordinates since $\partial_{\psi^1}$, $\partial_{\psi^2}$ generate the toric symmetry. Note also that there are no cross-terms between $\theta$ and $\psi^i$ since the horizontal distribution is integrable. 

The bounded Killing field is a linear combination of the action generators, and thus by
rescaling $\ell$ if necessary, we may assume without loss of generality that $T=a\partial_{\psi^1}+\partial_{\psi^2}$ for some $a\in\mathbb{R}$. Let 
\begin{equation}
\omega^1=d\theta,\qquad \omega^2 =d\psi^1 -a d\psi^2,\qquad \omega^3 =d\psi^2,
\end{equation}
be a co-frame tailored to $T$ in the sense that $\omega^1(T)=\omega^2(T)=0$ and $\omega^3(T)=1$. Then we may write
\begin{equation}
\tau=\sum_{\mathrm{i}=1}^{3}\tau_{\mathrm{i}}(\theta)\omega^{\mathrm{i}},\quad\quad \gamma=\sum_{\mathrm{i,j}=1}^{3}\gamma_{\mathrm{ij}}(\theta)\omega^{\mathrm{i}}\omega^{\mathrm{j}}.
\end{equation}
Since $\tau(T)=1$ and $\gamma(T,\cdot)=0$ we find that $\tau_3 =1$ and $\gamma_{13}=\gamma_{23}=\gamma_{33}=0$.
Next, by inserting the resulting expressions into \eqref{afonaoirnpioqy} we find that $\gamma_{11}$ is constant, and using that there are no cross-terms between $\theta$ and $\psi^i$ on the right-hand side it follows that $\tau_1 =\gamma_{12}=0$.
Moreover, since $\gamma$ locally defines a metric of Gauss curvature $+1$ on the space of leaves of the foliation tangent to $T$,
we conclude that $\gamma_{11}=1$ and $\gamma_{22}=\sin^2 \theta$. By setting $\beta=-\ell^{-1}a$ the model metric may now be expressed as
\begin{equation}
b= dr^2 + r^2 \left(d \theta^2 + \sin^2\theta (d\psi^1+\beta\ell d\psi^2)^2\right)
+\ell^2 \left(d \psi^2 + \tau_2 (d\psi^1+\beta\ell d\psi^2) \right)^2.
\end{equation}

By initially choosing coordinates on the torus appropriately, it may be assumed that the rod structures for the asymptotic end
are given by $\mathbf{v}_1 =(1,0)$ and $\mathbf{v}_{N+1}=(-q,p)$, for two coprime nonnegative integers $p$, $q$. Note that the first rod structure implies $\tau_2(\pi)=0$. If $\mathbf{v}_{N+1}=(-1,0)$ then $\mathcal{S}=S^1 \times S^2$, and also $\tau_2(0)=0$ so that the portion of the metric involving $\tau_2$ may be treated as error to produce $|\mathring{\nabla}^l (b-b_{AF})|_{b_{AF}}=O(r^{-1-l})$ where
\begin{equation}
b_{AF}=dr^2 + r^2 \left(d \theta^2 + \sin^2\theta (d\psi^1+\beta\ell d\psi^2)^2\right)
+\ell^2 (d \psi^2)^2 ;
\end{equation}
this yields case (iii) of the proposition. If $\mathbf{v}_{N+1}\neq (-1,0)$, then as in Remark \ref{faihjfpjaoh} regularity
demands that $p+\tau_2(0)(p\beta\ell-q)=0$ and $1+p\beta\ell =q$. 
It follows that $\tau_2(0)=p$. Hence, treating terms involving $\tau_2 -p\cos^2(\theta/2)$ as error produces $|\mathring{\nabla}^l (b-\tilde{b}_{ALF})|_{\tilde{b}_{ALF}}=O(r^{-1-l})$ where
\begin{equation}
\tilde{b}_{ALF}=d r^2 + r^2 \left(d \theta^2 + \sin^2\theta (d\psi^1 +\beta\ell d\psi^2)^2 \right)
+ \ell^2 \left(d \psi^2 + p\cos^2\left(\theta/2\right) (d \psi^1 +\beta\ell d\psi^2) \right)^2 ;
\end{equation}
this yields case (ii).
\end{proof}

\subsection{Brill coordinates}\label{brill}
Let $(M,g)$ be a simply connected toric ALE or toric ALF manifold, possibly having conical singularities and corners along the axes, and
let $\phi^i$ be a pair of independent $2\pi$-periodic angles adapted to the Killing field generators of the toric action so that $\eta_i = \partial_{\phi^i}$. By \cite{JaraczKhuri} there exists a set of global (Brill) coordinates $(\rho,z,\phi^1,\phi^2)$ for $M$ in which the metric may be expressed in submersion format
\begin{equation}\label{m1}
g = e^{2\alpha}\left(d\rho^2+dz^2\right) + G_{ij} (d \phi^i + A^i_a dx^a)(d \phi^j + A^j_a dx^a),
\end{equation} 
where $(x^1,x^2) = (\rho,z)$. The first portion of \eqref{m1} involving $e^{2\alpha}$ represents the metric on the orbit space $M/T^2$ which is parameterized by the half-plane $\{(\rho,z)\mid \rho\geq 0, z\in\mathbb{R}\}$, while $G=(G_{ij})$ yields the torus fiber metric, and the coefficients $A_a^i$ measure the obstruction to local integrability of the distribution orthogonal to the fibers. All coefficients $\alpha$, $G_{ij}$, and $A_a^i$ are functions of $(\rho,z)$ alone and satisfy the asymptotics as layed out in Section \ref{sec7:asymptotics} for neighborhoods of corner points, axis points, and at infinity. In particular, there exists a model metric from Section \ref{modelgeom} expressed in (radial) Brill coordinates such that
\begin{equation}\label{decay12}
     |\mathring{\nabla}^l(g-b_{ALE})|_{b_{ALE}}=O(r^{-1-\kappa-l}),\quad\quad
     |\mathring{\nabla}^l(g-\tilde{b}_{ALF})|_{\tilde{b}_{ALF}}=O(r^{-\frac{1}{2}-\kappa-l}),\quad l=0,1,2,
\end{equation}
for some $\kappa>0$. Here, the relation between radial and cylindrical Brill coordinates is given by the transformations \eqref{alerho} and \eqref{alfrho}, depending on the asymptotic type of $(M,g)$. Moreover, there exist model metrics $g_{\mathrm{cone}}$ as in Section \ref{sec7.axis}, derived from \eqref{conem} and expressed in Brill coordinates, such that upon approach to the interior of an axis rod
\begin{equation}\label{ccc1}
|\mathring{\nabla}^l (g-g_{\mathrm{cone}})|_{g_{\mathrm{cone}}}=O(\rho^{1+\zeta-l}),\quad\quad l=0,1,
\end{equation}
for some $\zeta> 0$. Similarly, in the neighborhood of a corner point, there exists a model metric of the form \eqref{metriccornercone} expressed in radial Brill coordinates such that 
\begin{equation}\label{cccc}
|\mathring{\nabla}^l (g-g_{\mathrm{corner}})|_{g_{\mathrm{corner}}}=O(r_n^{2-l}),\quad\quad l=0,1,
\end{equation}
where the relation between radial and cylindrical Brill coordinates is given by the transformation \eqref{cornerrho}, and $r_n$ denotes the $g_{\mathrm{corner}}$-distance to the corner point.

The Brill coordinate system gives rise to an advantageous expression for the scalar curvature, which makes contact with a certain harmonic map energy that is fundamental for mass comparison result. Lastly, we note that the logarithmic angle defect at interior points of an axis rod $\Gamma_n$ with rod structure $(v_n^1, v_n^2)$ may be expressed as
\begin{equation}\label{coneangleaxis}
e^{\pmb{\vartheta}}=\lim_{\rho \to 0} \frac{2\pi \cdot \mathrm{Radius}}{\mathrm{Circumference}} = \lim_{\rho \to 0} \sqrt{ \frac{\rho^2 e^{2\alpha}}{G_{ij} v_n^i v_n^j}}.
\end{equation} 
The existence of this limit is a consequence of the asymptotics detailed in Section \ref{sec7:asymptotics}.

\subsection{Toric harmonic maps}\label{secexistence}
In the setting of simply connected toric ALE/ALF manifolds, the Ricci flat equations reduce to solving for an axisymmetric harmonic map (\cite[Section 3]{KunduriLucietti}, \cite[Section 2]{LiSun}, \cite{Lott}) into the hyperbolic plane, $\Phi_o :\mathbb{R}^3 \setminus\Gamma \rightarrow \mathbb{H}^2$. In fact, one may prescribe the desired rod structure
and asymptotic type of the toric gravitational instanton, by solving for a harmonic map that is asymptotic to a given \textit{model map} that realizes this structure. By \textit{asymptotic}, we mean that the hyperbolic distance between the two maps stays bounded globally and converges to zero near infinity. The resulting instanton will most likely have concial singularities for generic rod data sets. The method to establish existence of such a harmonic map, asymptotic to a prescribed model map in this context, is based on an approach initiated by Weinstein
\cite{Weinstein} for 4-dimensional axisymmetric stationary vacuum black holes, and was later developed to incorporate rod structures by Khuri-Weinstein-Yamada \cite{KWY}. The adaptation to the (Riemannian) setting of toric gravitational instantons
was given by Kunduri-Lucietti \cite[Theorem 1.2]{KunduriLucietti} for the AF case, and this was recently expanded and generalized by Li-Sun \cite[Theorem 4.24]{LiSun}. The harmonic maps produced from this process are unique among those asymptotic to the given model. Although the two aforementioned results were carried out in the AF regime, the same technique holds for toric ALE/ALF instantons \cite[Theorem 1.1]{Kunduri:2026xvc}. The only requirement is the ability to construct an appropriate model map, and this may be achieved in the same manner as \cite[Theorem 1.2]{KunduriLucietti} except that at infinity we choose the model map to coincide with the harmonic map arising from the three Ricci flat model geometries of Section \ref{modelgeom}; note that in the ALF case this refers to the metric $\tilde{b}_{ALF}$ which is a modification of \eqref{bALF} using the function $h$. Thus, we obtain the following existence result.

\begin{theorem}\label{thmexistence}
Given a simply connected toric ALE or toric ALF manifold $(M,g)$ with rod data set $\mathcal{R}$, possibly having conical singularities and corners along finite axes, there exists a corresponding toric gravitational instanton $(M,g_o)$ potentially with conical singularities and corners sharing the same asymptotic ALE or ALF structure, and the same rod data set $\mathcal{R}$.
\end{theorem}

\begin{remark}\label{remreg}
Regularity of the harmonic maps associated with the instantons $(M,g_o)$ was investigated in the vicinity of axis rods and corners in \cite[Section 4.2]{LiSun}. The resulting asymptotics for the harmonic maps and Brill coordinate coefficients in these regions, as well as at infinity, are detailed in Section \ref{sec7:asymptotics}.
\end{remark}

\section{The Reduced Energy Functional} 
\label{sec3} \setcounter{equation}{0}
\setcounter{section}{3}

In the presence of a toric symmetry the scalar curvature naturally contains a harmonic map energy density
arising from the torus fiber portion of the metric. This density, however, exhibits blow-up behavior at
the axes and thus must be `renormalized' in order to serve a useful role in the context of mass comparison.
We begin with the basic expression for scalar curvature in this setting. This may be obtained from O'Neill's
formulas for Riemannian submersions \cite{MR200865} although here we give a direct derivation.


\begin{lemma}\label{lem:R(g)} 
Let $(M,g)$ be a toric Riemannian 4-manifold with metric expressed in Brill coordinates \eqref{m1}. On $\mathbb{R}^3 \setminus \Gamma$, define a function $Z$ and a $2\times 2$ symmetric matrix $\Phi$ with $\det\Phi =1$ by setting $G=\rho e^{Z}\Phi$, then the scalar curvature satisfies
\begin{equation}\label{scal3} 
\begin{aligned}
e^{2\alpha} R =&-2\Delta\alpha+2\nabla\alpha\cdot\nabla\log\rho-\frac{1}{4}\mathrm{Tr}\left(\Phi^{-1}\nabla \Phi\right)^2-\frac{1}{4}e^{-2\alpha} \delta_3^{ac} \delta_3^{bd} G_{ij} F^i_{ab} F^j_{cd}\\
&- 2\Delta Z-\frac{3}{2}|\nabla Z|^2-\nabla Z\cdot\nabla\log\rho+\frac{1}{2}|\nabla\log\rho|^2
\end{aligned}
\end{equation}
where $F^i_{ab}=\partial_aA^i_b-\partial_bA^i_a$ and $\delta_3=d\rho^2+dz^2+\rho^2d\varphi^2$ is the flat metric on $\mathbb{R}^3$ written in cylindrical coordinates, with $\nabla$, $\Delta$, and $\cdot$ denoting its covariant derivative, Laplacian, and inner product respetively. 
\end{lemma}

\begin{remark}
In the Ricci flat setting Proposition \ref{detG=rho} shows that $Z=0$ and $A^i_{a} =0$ for all $i,a=1,2$. 
\end{remark}

\begin{proof}
According to appendix equation \eqref{scalar2}, a computation shows that the scalar curvature takes the form
\begin{equation}\label{scal2} 
\begin{aligned}
e^{2\alpha} R=&-2\Delta\alpha+2\nabla\alpha\cdot\nabla\log\rho - \left(\Delta_2\log\det G+\frac{1}{4}\mathrm{Tr}\left(G^{-1}\nabla G\right)^2+\frac{1}{4}|\nabla\log\det G|^2\right)\\
&-\frac{1}{4}e^{-2\alpha} \delta_3^{ac} \delta_3^{bd} G_{ij} F^i_{ab} F^j_{cd}\\
\end{aligned}
\end{equation}
where $\Delta_2$ is the Laplacian with respect to the 2-dimensional flat metric $\delta_2=d\rho^2+dz^2$. 
Define the symmetric unimodular matrix
$\Phi:= (\det G)^{-\frac{1}{2}} G$, let $I_2$ denote the identity matrix, and observe that
\begin{equation}
    \begin{split}
    \left(G^{-1}\nabla G\right)^2&=\left(\frac{1}{2}\nabla\log\det G\, I_{2}+\Phi^{-1}\nabla \Phi\right)^2\\
    &=\frac{1}{4}|\nabla\log\det G|^2\, I_{2}+\nabla\log\det G\, \left(\Phi^{-1}\nabla \Phi\right)+\left(\Phi^{-1}\nabla \Phi\right)^2.
    \end{split}
\end{equation}  
Then combining this with $\mathrm{Tr}\left(\Phi^{-1}\nabla \Phi\right)=\nabla (\log\det\Phi)=0$ produces
\begin{equation}
\begin{split}
\mathrm{Tr}\left(G^{-1}\nabla G\right)^2&=\frac{1}{2}|\nabla\log\det G|^2+\mathrm{Tr}\left(\Phi^{-1}\nabla \Phi\right)^2.
\end{split}
\end{equation} 
Next, define the function
\begin{equation}\label{def:Z}
    Z:=\frac{1}{2}\log\det G-\log\rho 
\end{equation} 
and note that
\begin{equation}
  \begin{split}
        |\nabla\log\det G|^2&=4|\nabla Z|^2+8\nabla Z\cdot\nabla\log\rho+4|\nabla\log\rho|^2,
  \end{split}
\end{equation}
as well as
\begin{equation}
    \begin{split}
    \Delta_2\log\det G &=\Delta\log\det G-\nabla\log\det G\cdot\nabla\log\rho\\
    &=\Delta\left(\log\det G-2\log\rho\right)-2\nabla\left(Z+\log \rho\right)\cdot\nabla\log\rho\\
    &=2\Delta Z-2\nabla Z\cdot\nabla\log\rho-2|\nabla\log\rho|^2.
    \end{split}
\end{equation}
Inserting these expressions into \eqref{scal2} yields the desired result.
\end{proof}

We now seek an interpretation of the term involving $\Phi$ within the scalar curvature formula. 
To this end, define functions
\begin{equation}\label{Def:VW}
\begin{aligned}
V&:=\frac{1}{2}\log\left(\frac{\Phi_{11}}{\Phi_{22}}\right),\qquad W:=\sinh^{-1}\left(\Phi_{12}\right),\quad \text{ in the ALE, ALF, and AF$_0$ cases,} \\
V&:=\frac{1}{2}\log\left(\frac{\Phi_{11}}{\Phi_{22}-2\beta\ell \Phi_{12}+\beta^2\ell^2\Phi_{11}}\right),\quad W:=\sinh^{-1}\left(\Phi_{12}-\beta\ell \Phi_{11}\right),\quad \text{ in the AF$_{\beta\ell}$ case,}
\end{aligned}
\end{equation}
and observe that using $\det\Phi=1$ the inverse relations are given by
\begin{equation}
\Phi_{11} = e^V \cosh W, \qquad \Phi_{12} = \sinh W, \qquad    \Phi_{22} = e^{-V} \cosh W ,
\end{equation} 
and
\begin{equation}
\Phi_{11} \!=\! e^{V}\!\cosh W,\!\quad
\Phi_{12}\! = \!\sinh W \!+\! \beta \ell e^{V}\!\cosh W,\!\quad
\Phi_{22}\! = \!(e^{-V}\!\!+\beta^2 \ell^2 e^{V})\cosh W \!+ 2\beta \ell \sinh W ,
\end{equation}
respectively. It follows that
\begin{equation}\label{targetmet}
 \frac{1}{2}\mathrm{Tr}\left(\Phi^{-1}\nabla \Phi\right)^2=\cosh^2W|\nabla V|^2+|\nabla W|^2,
\end{equation} 
showing that this expression is a harmonic map energy density for the map $\Phi:\mathbb{R}^3 \setminus\Gamma\rightarrow\mathbb{H}^2$,
where the hyperbolic plane is parameterized by Fermi coordinates.

The asymptotics of $\Phi$ at the axes will typically produce an infinite energy, which motivates the following renormalization.
Let $(M,g)$ be a simply connected toric ALE or toric ALF manifold, and let $(M,g_o)$ be the corresponding toric gravitational instanton having the same rod data set given by Theorem \ref{thmexistence}. Consider the maps $\Psi=(V,W)$ and $\Psi_o =(V_o ,W_o)$ associated with $g$ and $g_o$, respectively. Using
\begin{align}
    \begin{split}
        |\nabla V|^2&=|\nabla(V-V_o)|^2+2\nabla(V-V_o)\cdot \nabla V_o+|\nabla V_o|^2,\\
       |\nabla W|^2&=|\nabla(W-W_o)|^2+2\nabla(W-W_o)\cdot \nabla W_o+|\nabla W_o|^2,
    \end{split}
\end{align}
and the harmonic property of $\log\rho$ with respect to $\delta_3$, produces the difference of the $g$ and $g_o$-scalar curvatures 
\begin{equation}\label{scal3bDiff} 
\begin{aligned}
e^{2\alpha}R =&\text{div}_{\delta_3}X-\frac{3}{2}|\nabla Z|^2-\frac{1}{4}e^{-2\alpha} \delta_3^{ac} \delta_3^{bd} G_{ij} F^i_{ab} F^j_{cd}+(V-V_o)\Delta V_o+(W-W_o)\Delta W_o\\
&-\frac{1}{2}\left(\sinh^2W|\nabla V|^2-\sinh^2W_o|\nabla V_o|^2+|\nabla (V-V_0)|^2+|\nabla (W-W_o)|^2\right),
\end{aligned}
\end{equation}
where
\begin{equation}\label{defX}
\begin{split}
X:=&-2\nabla(\alpha-\alpha_o+Z)+\left(2\alpha-2\alpha_o-Z\right)\nabla\log\rho-(V-V_o)\nabla V_o-(W-W_o)\nabla W_o .
\end{split}
\end{equation}

Let $\boldsymbol{\varsigma}=(\varsigma_1,\varsigma_2,\varsigma_3)$ be a collection of small positive parameters.
We may decompose the open ball $B_{2/\varsigma_{3}}\subset\mathbb{R}^3$ centered at the origin that includes part of the semi-infinite rods, into three types of pairwise disjoint regions 
$B_{2/\varsigma_3}=\Omega_{\boldsymbol{\varsigma}}\cup\mathcal{A}_{\boldsymbol{\varsigma}}\cup \left(\cup_{n=1}^N \overline{B_{\varsigma_2}(z_n)}\right)$ where
\begin{equation}
\Omega_{\boldsymbol{\varsigma}}\!=\!\{ r<2/\varsigma_3, \text{ }\rho>\varsigma_1, \text{ }\varsigma_2<r_n \text{ for } n=1,\dots\!, N
\},\quad \mathcal{A}_{\boldsymbol{\varsigma}}\!=\!B_{2/\varsigma_3}\setminus
\left(\Omega_{\boldsymbol{\varsigma}}\cup \left(\cup_{n=1}^N \overline{B_{\varsigma_2}(z_n)}\right)\right).
\end{equation}
Here $r$ is adapted to the 4-dimensional model geometries and is given by \eqref{alerho}, \eqref{alfrho} in the ALE, ALF cases respectively, whereas $r_n$ is a radial coordinate defined by \eqref{cornerrho} which is centered at the $n$th corner point on the $z$-axis located at height $z_n$, and $B_{\varsigma_2}(z_n)$ is the open ball centered at this point of radius $\varsigma_2$. Integrating \eqref{scal3bDiff} over $\Omega_{\boldsymbol\varsigma}$ and using the divergence theorem yields
\begin{equation}\label{integralscalarcurvature}
\mathcal{B}^{\boldsymbol\varsigma}_\text{axis}+\mathcal{B}^{\boldsymbol\varsigma}_\text{corner}+\mathcal{B}^{\boldsymbol\varsigma}_{\infty}
=\mathcal{I}_{\Omega_{\boldsymbol\varsigma}}(\Psi)+\int_{\Omega_{\boldsymbol\varsigma}}\left(e^{2\alpha}R+\frac{3}{2}|\nabla Z|^2+\frac{1}{4}e^{-2\alpha} \delta_3^{ac} \delta_3^{bd} G_{ij} F^i_{ab} F^j_{cd}\right)dx
\end{equation} 
where
\begin{align}\label{ReducedI}
\begin{split}
\mathcal{I}_{\Omega_{\boldsymbol\varsigma}}(\Psi)=\frac{1}{2}&\int_{\Omega_{\boldsymbol\varsigma}}\left(\sinh^2W|\nabla V|^2-\sinh^2W_o|\nabla V_o|^2+|\nabla (V-V_0)|^2+|\nabla (W-W_o)|^2\right)\,dx\\
-&\int_{\Omega_{\boldsymbol\varsigma}}\left((V-V_o)\Delta V_o+(W-W_o)\Delta W_o\right)dx,
\end{split}
\end{align}
and
\begin{equation}\label{Bon1}
\mathcal{B}^{\boldsymbol\varsigma}_\text{axis}=\int_{\partial \mathcal{A}_{\boldsymbol\varsigma}\cap\partial\Omega_{\boldsymbol\varsigma}}\!\!\!\!\!\!\!\!\!\! X(\nu) dA,\quad 
\mathcal{B}^{\boldsymbol\varsigma}_\text{corner}=\sum_{n=1}^{N}\int_{\partial \mathcal{B}_{\varsigma_2}(z_n)\cap\partial\Omega_{\boldsymbol\varsigma}}\!\!\!\!\!\!\!\!\!\! X(\nu) dA,\quad
\mathcal{B}^{\boldsymbol\varsigma}_\infty=\int_{\partial B_{2/\varsigma_3}\cap\partial\Omega_{\boldsymbol\varsigma}}\!\!\!\!\!\!\!\!\!\! X(\nu) dA,
\end{equation} 
with $\nu$ denoting the unit outer normal. We define the \emph{reduced energy} to be the following limit 
\begin{equation}
\mathcal{I}(\Psi):=\lim_{\varsigma_3\to 0}\lim_{\varsigma_2\to 0}\lim_{\varsigma_1\to 0}\mathcal{I}_{\Omega_{\boldsymbol\varsigma}}(\Psi).
\end{equation}
The corresponding limits for the boundary integrals \eqref{Bon1} exist and are finite by
Lemmas \ref{lem4.1}, \ref{lem4.2}, and \ref{LemmaIinfinity}, and the same will now be shown for the reduced energy.


\begin{proposition}\label{Prop:finiteness}
Let $\Psi=(V,W)$ and $\Psi_o=(V_o,W_o)$ be maps as described above. Then the reduced energy functional $\mathcal{I}(\Psi)$ is well-defined and finite.
\end{proposition}

\begin{proof}
Since the limits of boundary integrals in \eqref{integralscalarcurvature} exist and are finite, it suffices to show the same for the bulk integral expression in this equation.
To see this, observe that the asymptotics of Section \ref{sec7:asymptotics} imply
\begin{equation}
|\nabla Z|^2=O(r^{-6-2\kappa})\quad \text{for ALE},\qquad |\nabla Z|^2=O(r^{-3-2\kappa})\quad \text{for ALF/ AF$_{\beta\ell}$},
\end{equation}
\begin{equation}
|\nabla Z|^2=O(1)\quad \text{in $\mathcal{A}_{\boldsymbol\varsigma}$},\qquad |\nabla Z|^2=O(1)\quad \text{in $B_{\varsigma_2}(z_n)$},
\end{equation}
showing that the second integrand is integrable. Moreover,
since $R\in L^1(M)$ and $dx_g=e^{2\alpha}e^{Z}dx$, it follows that $e^{2\alpha}R\in L^1(\mathbb{R}^3)$. 
Furthermore, the asymptotics of Section \ref{sec7:asymptotics} also produce
\begin{equation}
\delta_3^{ac} \delta_3^{bd} G_{ij} F^i_{ab} F^j_{cd}=O(r^{-8-2\kappa})\quad \text{for ALE},\quad \delta_3^{ac} \delta_3^{bd} G_{ij} F^i_{ab} F^j_{cd}=O(r^{-3-2\kappa})\quad \text{for ALF/AF$_{\beta\ell}$},
\end{equation}
\begin{equation}
\delta_3^{ac} \delta_3^{bd} G_{ij} F^i_{ab} F^j_{cd}=O(\rho^{2\zeta})\quad\text{in $\mathcal{A}_{\boldsymbol\varsigma}$},\qquad \delta_3^{ac} \delta_3^{bd} G_{ij} F^i_{ab} F^j_{cd}=O(r_n^{-2})\quad \text{in $B_{\varsigma_2}(z_n)$}.
\end{equation}
Thus, the last integrand is integrable.
\end{proof}

\section{Convexity of Reduced Energy Functional} 
\label{sec4} \setcounter{equation}{0}
\setcounter{section}{4}

Consider the hyperbolic plane $\mathbb{H}^2$ with metric expressed in Fermi coordinate $(V,W)$ as follows
\begin{equation}\label{16}
g_{\mathbb{H}^2}=\cosh^2WdV^2+dW^2.
\end{equation}
Let $\Omega\subset\mathbb{R}^3 \setminus\Gamma$ be a domain, then the harmonic energy of a map ${\Psi}=(V,W):\Omega\rightarrow
\mathbb{H}^2$ is given by
\begin{align}\label{energy123}
\begin{split}
E_{\Omega}({\Psi})=&\frac{1}{4}\int_{\Omega}\mathrm{Tr}\left(\Phi^{-1}\nabla\Phi\right)^2\,dx=\frac{1}{2}\int_{\Omega}
\left(\cosh^2W|\nabla V|^2+|\nabla W|^2\right)\,dx.
\end{split}
\end{align}
Critical points $\Psi_o =(V_o,W_0)$ of this energy satisfy the harmonic map equations
\begin{equation}\label{ELeqn}
  \mathrm{div}\left(\cosh^2 W_o\nabla V_o\right)=0,\qquad 
\Delta W_o-\sinh  W_o\cosh W_o |\nabla V_o|^2=0\,.
\end{equation}
Moreover, the relation between the harmonic energy $E$ and reduced energy $\mathcal{I}$ takes the form
\begin{equation}\label{relationenergy}
\begin{split}
\mathcal{I}_{\Omega}(\Psi)&=E_{\Omega}({\Psi})
-E_\Omega(\Psi_o)-\int_{\partial\Omega}\left(\nu(V_o)\left(V-V_o\right)+\nu(W_o)\left(W-W_o\right)\right)dA,
\end{split}
\end{equation}
where $\nu$ is the unit outward normal on $\partial\Omega$. 
The main goal of this section is to establish a gap lower bound for the reduced energy. 

\begin{theorem}\label{thm.convexity}
Suppose that the map $\Psi=(V,W)$ and related harmonic map $\Psi_o=(V_o,W_o)$ are smooth on $\mathbb{R}^3 \setminus\Gamma$, and satisfy the asymptotics of Section \ref{sec7:asymptotics}. Then there exists a constant $C>0$ such that
\begin{equation}\label{RHEconvexity}
\begin{split}
\mathcal{I}(\Psi)
&\geq C\left(\int_{\mathbb{R}^3}
\operatorname{dist}_{\mathbb{H}^2}^{6}(\Psi,\Psi_{o})\,dx
\right)^{1/3}\,.
\end{split}
\end{equation}
\end{theorem}

Since the target space is negatively curved, the harmonic energy is convex on bounded regions that exclude the axis and corner singularities. The singular behavior of the maps $\Psi$ and $\Psi_o$ near the axis, however, prevents this convexity from extending directly to the reduced energy on the whole of $\mathbb{R}^3$. It is therefore necessary to analyze the boundary behavior of the reduced energy separately near the axis, at the corners, and at infinity. Proving that the boundary terms make no contribution to the convexity argument requires a cut-and-paste construction in which $\Psi$ is replaced by $\Psi_o$ near the axis. 

Let $\boldsymbol{\varepsilon}=(\varepsilon_1,\varepsilon_2,\varepsilon_3)$ be a collection of small positive parameters such that $\varsigma_i<\varepsilon_i <1$, where $\boldsymbol{\varsigma}=(\varsigma_1,\varsigma_2,\varsigma_3)$ is given in Section \ref{sec3}. Consider the following cut-off function
\begin{equation}\label{67}
\varphi_{\varepsilon_1}=\begin{cases}
0 & \text{ if $\rho\leq\varepsilon_1$} \\
\frac{\log(\rho/\varepsilon_1)}{\log(\sqrt{\varepsilon_1}/
\varepsilon_1)} &
\text{ if $\varepsilon_1<\rho<\sqrt{\varepsilon_1}$} \\
1 & \text{ if $\rho\geq\sqrt{\varepsilon_1}$} \\
\end{cases}\,.
\end{equation}
Recall the region $\mathcal{A}_{\boldsymbol\varepsilon}$ and define an additional annular cylindrical region about the $z$-axis by
\begin{equation}\label{91.1}
\begin{split}
\tilde{\mathcal{A}}_{\boldsymbol{\varepsilon}}&=\{\varepsilon_1\leq \rho\leq \sqrt{\varepsilon_1}\}\cap \{ r< 2/\varepsilon_3,\text{ }{\varepsilon_2}< r_n\,\text{ for }\, n=1,\dots,N\},\\
\mathcal{A}_{\boldsymbol{\varepsilon}}&=\{\rho\leq \varepsilon_1\}\cap \{r< 2/\varepsilon_3,\text{ }{\varepsilon_2}< r_n\,\text{ for }\text{}\, n=1,\dots,N\}.
\end{split}
\end{equation}
Furthermore, set $\Psi_{\boldsymbol\varepsilon}=(V_{\varepsilon_1},W_{\varepsilon_1}):=(V_o+\varphi_{\varepsilon_1}(V-V_o),W_o+\varphi_{\varepsilon_1}(W-W_o))$ so that
\begin{equation}
(V_{\varepsilon_1},W_{\varepsilon_1})=\begin{cases}
(V,W) & \text{in }\tilde{B}_{2/\varepsilon_3}\setminus \left(\tilde{\mathcal{A}}_{\boldsymbol{\varepsilon}}\cup \mathcal{A}_{\boldsymbol{\varepsilon}}\right)\\
(V_o+\varphi_{\varepsilon_1}(V-V_o),W_o+\varphi_{\varepsilon_1}(W-W_o)) & \text{in }\tilde{\mathcal{A}}_{\boldsymbol{\varepsilon}}\\
(V_o,W_o) & \text{in }\mathcal{A}_{\boldsymbol{\varepsilon}}
\end{cases},
\end{equation}
where $\tilde{B}_{2/\varepsilon_3}=B_{2/\varepsilon_3}\setminus \cup_{n=1}^N B_{\varepsilon_2}(z_n)$.

\begin{lemma}\label{lemma3.4}
For fixed $\varepsilon_2 , \varepsilon_3>0$ it holds that
\begin{equation}
\lim_{\varepsilon_1\rightarrow 0}\mathcal{I}_{\tilde{B}_{2/\varepsilon_3}}(\Psi_{\boldsymbol\varepsilon})=\mathcal{I}_{\tilde{B}_{2/\varepsilon_3}}(\Psi).
\end{equation}
\end{lemma}

\begin{proof}
Write
\begin{equation}\label{94}
\mathcal{I}_{\tilde{B}_{2/\varepsilon_3}}(\Psi_{\boldsymbol{\varepsilon}})
=\mathcal{I}_{\tilde{\mathcal{A}}_{\boldsymbol{\varepsilon}}}(\Psi_{\boldsymbol{\varepsilon}})
+\mathcal{I}_{\mathcal{A}_{\boldsymbol{\varepsilon}}}(\Psi_{\boldsymbol{\varepsilon}})
+\mathcal{I}_{\tilde{B}_{2/\varepsilon_3}\setminus(\tilde{\mathcal{A}}_{\boldsymbol{\varepsilon}}
\cup\mathcal{A}_{\boldsymbol{\varepsilon}})}(\Psi_{\boldsymbol{\varepsilon}}),
\end{equation}
and observe that 
\begin{equation}\label{95}
\mathcal{I}_{\tilde{B}_{2/\varepsilon_3}\setminus \left(\tilde{\mathcal{A}}_{\boldsymbol{\varepsilon}}\cup \mathcal{A}_{\boldsymbol{\varepsilon}}\right)}(\Psi_{\boldsymbol{\varepsilon}})=\mathcal{I}_{\tilde{B}_{2/\varepsilon_3}\setminus \left(\tilde{\mathcal{A}}_{\boldsymbol{\varepsilon}}\cup \mathcal{A}_{\boldsymbol{\varepsilon}}\right)}(\Psi).
\end{equation}
Moreover $\Psi_{\boldsymbol{\varepsilon}}=\Psi_o$ on $\mathcal{A}_{\boldsymbol{\varepsilon}}$, so the reduced energy vanishes when restricted to $\mathcal{A}_{\boldsymbol{\varepsilon}}$. On the remaining region we have 
\begin{align}
\begin{split}
&2\mathcal{I}_{\tilde{\mathcal{A}}_{\boldsymbol{\varepsilon}}}(\Psi_{\boldsymbol{\varepsilon}})
=\underbrace{\int_{\tilde{\mathcal{A}}_{\boldsymbol{\varepsilon}}}
|\nabla ({V}_{\varepsilon_1}-V_o)|^{2}}_{I_{1}}+\underbrace{\int_{\tilde{\mathcal{A}}_{\boldsymbol{\varepsilon}}}
|\nabla ({W}_{\varepsilon_1}-W_o)|^{2}}_{I_{2}}\\
&+\underbrace{\int_{\tilde{\mathcal{A}}_{\boldsymbol{\varepsilon}}}
\left(\sinh^2{W}_{\varepsilon_1}|\nabla {V}_{\varepsilon_1}|^{2}-\sinh^2 {W}_o|\nabla V_o|^2\right)-2({V}_{\varepsilon_1}-V_o)\Delta V_o-2
({W}_{\varepsilon_1}-W_o)\Delta W_o}_{I_{3}}.
\end{split}
\end{align}
To estimate these expressions, it is helpful to decompose the region into connected components $\tilde{\mathcal{A}}_{\boldsymbol{\varepsilon}}=\cup_{n=1}^{N+1}\tilde{\mathcal{A}}_{\boldsymbol{\varepsilon}}^n$, where the annular cylinder $\tilde{\mathcal{A}}_{\boldsymbol{\varepsilon}}^n$ is associated with the rod $\Gamma_n$. In what follows, we will analyze each integral according to the asymptotics of Section \ref{sec7.axis} for the three different types of rod structure on $\Gamma_n$, namely: (I) $\mathbf{v}_n=(1,0)$, (II) $\mathbf{v}_n=(0,1)$, and (III) $\mathbf{v}_n=(v^1_n,v^2_n)\in(\mathbb{Z}\backslash \{0\})^2$. 
In all cases it holds that
\begin{equation}
|I_{1}|\leq \int_{\tilde{\mathcal{A}}_{\boldsymbol{\varepsilon}}^n}
\left(\underbrace{|\nabla (V-V_o)|^2}_{O(1)}+\underbrace{|V-V_o|^2}_{O(1)}
\underbrace{|\nabla\varphi_{\varepsilon_1}|^2}_{ O\left((\rho\log\varepsilon_1)^{-2}\right)}\right)\rho d\rho dz\rightarrow 0,
\end{equation} 
and similarly
\begin{equation}
\begin{split}
|I_{2}|\leq \int_{\tilde{\mathcal{A}}_{\boldsymbol{\varepsilon}}^n}
&\left(\underbrace{|\nabla (W-W_o)|^2}_{O(1)}+\underbrace{|W-W_o|^2}_{O(1)}
\underbrace{|\nabla\varphi_{\varepsilon_1}|^2}_{ O\left((\rho\log\varepsilon_1)^{-2}\right)}\right)\rho d\rho dz\rightarrow 0.
\end{split}
\end{equation}

Next consider $I_3$. For this integral, we will further decompose case II into two subcases: II$_0$ in which $\beta=0$, and II$_\beta$ in which $\beta\neq 0$. From Section \ref{sec7.axis} it follows that
\begin{equation}
\begin{split}
\sinh^2 {W}_{\varepsilon_1}-\sinh^2 {W}_o&=\sinh\left({W}_{\varepsilon_1}+{W}_o\right)\sinh\left({W}_{\varepsilon_1}-{W}_o\right)=\begin{cases}
    O(\rho^{2}) & \text{case I and II$_{0}$}\\
    O(\rho^{-2}) & \text{case II$_{\beta}$ and III}\\
\end{cases}\, ,
\end{split}
\end{equation}
and with the help of the harmonic map equations \eqref{ELeqn} we find
\begin{equation}
\begin{split}
\Delta W_o&=\sinh W_o\cosh W_o|\nabla V_o|^2=\begin{cases}
    O(\rho^{-1}) & \text{case I and II$_0$}\\
    O(\rho^{2\zeta-2}) & \text{case II$_{\beta}$}\\
    O(1) & \text{case III}\\
\end{cases},\\
\Delta V_o&=-2\tanh W_o\nabla W_o\cdot \nabla V_o=\begin{cases}
    O(1) & \text{case I and II$_0$}\\
    O(\rho^{\zeta -1}) & \text{case II$_{\beta}$}\\
    O(1) & \text{case III}\\
\end{cases}\,,
\end{split}
\end{equation}
where $\zeta>0$. Therefore, in all cases
\begin{equation}
\begin{split}
|I_{3}|&\leq \int_{\tilde{\mathcal{A}}_{\boldsymbol{\varepsilon}}^n}
\underbrace{\left|\sinh^2{W}_{\varepsilon_1}-\sinh^2 {W}_o\right||\nabla V_o|^2}_{O(\rho^{2\zeta-2})}\rho d\rho dz\\
&+\int_{\tilde{\mathcal{A}}_{\boldsymbol{\varepsilon}}^n}\underbrace{\sinh^2 {W}_{\varepsilon_1}|\nabla(V_{\varepsilon_1}-V_o)||\nabla (V_{\varepsilon_1}+V_o)|}_{O(\rho^{2\zeta-2})}\rho d\rho dz\\
&+2\int_{\tilde{\mathcal{A}}_{\boldsymbol{\varepsilon}}^n}\left(\underbrace{|V-V_o||\Delta V_o|}_{ O(1)}+\underbrace{|W-W_o||\Delta W_o|}_{ O(\rho^{2\zeta-2})}\right)\rho d\rho dz.
\end{split}
\end{equation}
Hence, $\mathcal{I}_{\tilde{\mathcal{A}}_{\boldsymbol{\varepsilon}}}(\Psi_{\boldsymbol{\varepsilon}})\rightarrow 0$
and the desired is obtained.
\end{proof}

\begin{proof}[Proof of Theorem \ref{thm.convexity}]
Let $\Psi,\Psi_o:\mathbb{R}^3\setminus\Gamma\to\mathbb{H}^2$ be as in the statement of this theorem, and consider the cut-and-paste map $\Psi_{\boldsymbol{\varepsilon}}$ for $\boldsymbol{\varepsilon}>0$. Let $\Psi_{\boldsymbol{\varepsilon}}^t=(V^t_{\boldsymbol{\varepsilon}},W^t_{\boldsymbol{\varepsilon}})$ be the geodesic deformation from $\Psi_{o}$ to $\Psi_{\boldsymbol{\varepsilon}}$ in $\mathbb{H}^2$. In particular, for each $x\in\mathbb{R}^3\setminus\Gamma$ we have that $\Psi^{\bullet}_{\boldsymbol{\varepsilon}}(x):[0,1]\to\mathbb{H}^2$ is the geodesic with 
$\Psi^1_{\boldsymbol{\varepsilon}}(x)=\Psi_{\boldsymbol{\varepsilon}}(x)$ and $\Psi^0_{\boldsymbol{\varepsilon}}(x)=\Psi_o(x)$. Let $\boldsymbol{\varepsilon}>\boldsymbol{\varsigma}>0$
and observe that the second variation of energy \cite[(2.4)]{schoen2013convexity} yields
\begin{equation}
\frac{d^{2}}{dt^{2}}E_{\Omega_{\boldsymbol{\varsigma}}}(\Psi_{\boldsymbol{\varepsilon}}^{t})
\geq 2\int_{\Omega_{\boldsymbol{\varsigma}}}|\nabla\operatorname{dist}_{\mathbb{H}^2}(\Psi_{\boldsymbol{\varepsilon}},\Psi_{0})|^{2}dx.
\end{equation}
Combining this with \eqref{relationenergy} produces
\begin{equation}
\begin{split}
\frac{d^{2}}{dt^{2}}\mathcal{I}_{\Omega_{\boldsymbol{\varsigma}}}(\Psi^{t}_{\boldsymbol{\varepsilon}})
&=\frac{d^{2}}{dt^{2}}E_{\Omega_{\boldsymbol{\varsigma}}}
(\Psi^{t}_{\boldsymbol{\varepsilon}})-
\frac{d^{2}}{dt^{2}}\int_{\partial\Omega_{\boldsymbol{\varsigma}}}\left(\nu(V_o)\left(V^{t}_{\boldsymbol{\varepsilon}}-V_o\right)+\nu(W_o)\left(W^{t}_{\boldsymbol{\varepsilon}}-W_o\right)\right)dA\\
&\geq  2\int_{\Omega_{\boldsymbol{\varsigma}}}|\nabla\operatorname{dist}_{\mathbb{H}^2}(\Psi_{\boldsymbol{\varepsilon}},\Psi_{0})|^{2}dx-\int_{\partial\Omega_{\boldsymbol{\varsigma}}}\left(\nu(V_o)\ddot{V}^t_{\boldsymbol{\varepsilon}}+\nu(W_o)\ddot{W}^{t}_{\boldsymbol{\varepsilon}}\right)dA,
\end{split}
\end{equation}
where the `dot' derivatives are with respect to $t$. Now integrate from $0$ to $t$ to find
\begin{equation}\label{eq48}
\begin{split}
\frac{d}{dt}\mathcal{I}_{\Omega_{\boldsymbol{\varsigma}}}(\Psi^{t}_{\boldsymbol{\varepsilon}})-\frac{d}{dt}\mathcal{I}_{\Omega_{\boldsymbol{\varsigma}}}(\Psi^{t}_{\boldsymbol{\varepsilon}})|_{t=0}
&\geq 2t\int_{\Omega_{\boldsymbol{\varsigma}}}|\nabla\operatorname{dist}_{\mathbb{H}^2}(\Psi_{\boldsymbol{\varepsilon}},\Psi_{0})|^{2}dx\\
&\quad -\int_{\partial\Omega_{\boldsymbol{\varsigma}}}\left(\nu(V_o)(\dot{V}^t_{\boldsymbol{\varepsilon}}-\dot{V}^0_{\boldsymbol{\varepsilon}})+\nu(W_o)(\dot{W}^t_{\boldsymbol{\varepsilon}}-\dot{W}^0_{\boldsymbol{\varepsilon}})\right)dA.
\end{split}
\end{equation}
On the other hand, from the first variation of \eqref{relationenergy} we have
\begin{equation}
\begin{split}
\frac{d}{dt}\mathcal{I}_{\Omega_{\boldsymbol{\varsigma}}}(\Psi^{t}_{\boldsymbol{\varepsilon}})&=-\int_{\Omega_{\boldsymbol{\varsigma}}}
\dot{V}^t_{\boldsymbol{\varepsilon}}\mathrm{div}\left(\cosh^2W^{t}_{\boldsymbol{\varepsilon}}\nabla V^{t}_{\boldsymbol{\varepsilon}}\right)\,dx\\
&\quad-\int_{\Omega_{\boldsymbol{\varsigma}}}
\dot{W}^t_{\boldsymbol{\varepsilon}}\left(\Delta W^{t}_{\boldsymbol{\varepsilon}}-\sinh W^{t}_{\boldsymbol{\varepsilon}}\cosh W^{t}_{\boldsymbol{\varepsilon}}|\nabla V^{t}_{\boldsymbol{\varepsilon}}|^2\right)\,dx\\
&\quad+\int_{\partial \Omega_{\boldsymbol{\varsigma}}}\left(\dot{W}^t_{\boldsymbol{\varepsilon}}\left(\nu(W^{t}_{\boldsymbol{\varepsilon}})-\nu(W_o)\right)+\dot{V}^t_{\boldsymbol{\varepsilon}}\left(\nu(V^{t}_{\boldsymbol{\varepsilon}})\cosh^2 W^{t}_{\boldsymbol{\varepsilon}}-\nu(V_o)\right)\right)dA.
\end{split}
\end{equation}
Since $\Psi^t_{\boldsymbol{\varepsilon}}$ at $t=0$ is a harmonic map and satisfies \eqref{ELeqn}, it follows that
\begin{equation}\label{eq410}
\begin{split}
\frac{d}{dt}\mathcal{I}_{\Omega_{\boldsymbol{\varsigma}}}(\Psi^{t}_{\boldsymbol{\varepsilon}})|_{t=0}&=\int_{\partial \Omega_{\boldsymbol{\varsigma}}}\dot{V}^0_{\boldsymbol{\varepsilon}}\nu(V_o)\sinh^2 W_o dA.
\end{split}
\end{equation}
Putting this together with \eqref{eq48} gives rise to
\begin{equation}
\begin{split}
\frac{d}{dt}\mathcal{I}_{\Omega_{\boldsymbol{\varsigma}}}(\Psi^{t}_{\boldsymbol{\varepsilon}})
&\geq 2t\int_{\Omega_{\boldsymbol{\varsigma}}}|\nabla\operatorname{dist}_{\mathbb{H}^2}(\Psi_{\boldsymbol{\varepsilon}},\Psi_{0})|^{2}dx\\
&\qquad-\int_{\partial\Omega_{\boldsymbol{\varsigma}}}\left(\nu(V_o)(\dot{V}^t_{\boldsymbol{\varepsilon}}-\dot{V}^0_{\boldsymbol{\varepsilon}}-\dot{V}^0_{\boldsymbol{\varepsilon}}\sinh^2 W_o)+\nu(W_o)(\dot{W}^t_{\boldsymbol{\varepsilon}}-\dot{W}^0_{\boldsymbol{\varepsilon}})\right)dA.
\end{split}
\end{equation}
Now integrate again from $0$ to $1$ and use the fact that $\mathcal{I}_{\Omega_{\boldsymbol{\varsigma}}}(\Psi^{0}_{\boldsymbol{\varepsilon}})=0$ to obtain
\begin{equation}\label{boundaryenergy}
\begin{split}
\mathcal{I}_{\Omega_{\boldsymbol{\varsigma}}}(\Psi_{\boldsymbol{\varepsilon}})
&\geq \int_{\Omega_{\boldsymbol{\varsigma}}}|\nabla\operatorname{dist}_{\mathbb{H}^2}(\Psi_{\boldsymbol{\varepsilon}},\Psi_{0})|^{2}dx\\
&\quad\underbrace{-\int_{\partial\Omega_{\boldsymbol{\varsigma}}}\left(\nu(V_o)(V_{\boldsymbol{\varepsilon}}-V_o-\dot{V}^0_{\boldsymbol{\varepsilon}}\cosh^2W_o)+\nu(W_o)(W_{\boldsymbol{\varepsilon}}-W_o-\dot{W}^0_{\boldsymbol{\varepsilon}})\right)dA}_{:=I_{\partial\Omega_{\boldsymbol{\varsigma}}}}\,.
\end{split}
\end{equation}

We will use the geodesic equations and distance function in the hyperbolic plane to estimate each term within $I_{\partial\Omega_{\boldsymbol{\varsigma}}}$. Applying a Taylor expansion about $\Psi^0_{\boldsymbol{\varepsilon}}=\Psi_o$ yields
\begin{equation}\label{taylor1}
\Psi^t_{\boldsymbol{\varepsilon}}=\Psi_o+\dot{\Psi}^0_{\boldsymbol{\varepsilon}}+\frac{1}{2}\ddot{\Psi}^{t'}_{\boldsymbol{\varepsilon}},\qquad \text{for some } t'\in(0,1)\,,
\end{equation}
while the geodesic equations may be written explicitly as
\begin{equation}\label{gedesic}
\ddot{V}^{t}_{\boldsymbol{\varepsilon}}+2\tanh W^{t}_{\boldsymbol{\varepsilon}}\,\dot{V}^{t}_{\boldsymbol{\varepsilon}}\dot{W}^{t}_{\boldsymbol{\varepsilon}}=0,\qquad \ddot{W}^{t}_{\boldsymbol{\varepsilon}}-\sinh W^{t}_{\boldsymbol{\varepsilon}}\cosh W^{t}_{\boldsymbol{\varepsilon}}\,\left(\dot{V}^{t}_{\boldsymbol{\varepsilon}}\right)^2=0\,.
\end{equation}
Moreover, since for each $x\in\mathbb{R}^3\setminus\Gamma$ the geodesic $\Psi^t_{\boldsymbol{\varepsilon}}(x)$ has constant velocity and there is a unique minimizing geodesic between any two points of $\mathbb{H}^2$, it holds that
\begin{equation}\label{distancelenght}
\mathbf{d}^2:=\mathrm{dist}^2_{\mathbb{H}^2}\left(\Psi_{\boldsymbol{\varepsilon}}, \Psi_o\right)= \left(\dot{W}^t_{\boldsymbol{\varepsilon}}\right)^2+\cosh^2 W^t_{\boldsymbol{\varepsilon}}\left(\dot{V}^t_{\boldsymbol{\varepsilon}}\right)^2\,\qquad \text{for all } t\in[0,1]\,.
\end{equation}
Then the boundary integral of \eqref{boundaryenergy} may be rewritten with the Taylor expansion \eqref{taylor1} and geodesic equations \eqref{gedesic} by
\begin{equation}\label{4.28}
I_{\partial\Omega_{\boldsymbol{\varsigma}}}\!=\!\int_{\partial\Omega_{\boldsymbol{\varsigma}}}\!\!\!\left(\!\nu(V_o)\left(\dot{V}^0_{\boldsymbol{\varepsilon}}\sinh^2W_o+\tanh W_{\boldsymbol{\varepsilon}}^{t'}\dot{V}^{t'}_{\boldsymbol{\varepsilon}}\dot{W}^{t'}_{\boldsymbol{\varepsilon}}\right)
\!-\!\frac{1}{2}\nu(W_o)\sinh W_{\boldsymbol{\varepsilon}}^{t'}\cosh W_{\boldsymbol{\varepsilon}}^{t'}\left(\dot{V}^{t'}_{\boldsymbol{\varepsilon}}\right)^2\right)dA,
\end{equation}
and therefore
\begin{equation}\label{4.29}
|I_{\partial\Omega_{\boldsymbol{\varsigma}}}|\leq \int_{\partial\Omega_{\boldsymbol{\varsigma}}}
\left(|\nu(V_o)|\left(|\sinh W_o|\mathbf{d}
+\mathbf{d}^2\right)+\frac{1}{2}|\nu(W_0)|
\mathbf{d}^2\right)dA.
\end{equation}

Recall the distance function in the hyperbolic plane 
\begin{equation}
\cosh \mathbf{d}=\cosh (W_{\boldsymbol{\varepsilon}}-W_o)\cosh(V_{\boldsymbol{\varepsilon}}-V_o)+\sinh W_{\boldsymbol{\varepsilon}}\sinh W_o\left(\cosh(V_{\boldsymbol{\varepsilon}}-V_o)-1\right).
\end{equation}
Using the asymptotics of Section \ref{sec7:infinity} for the asymptotic end, we then obtain decay rates
for the distance function
\begin{equation}\label{gfobain0qhh}
\mathbf{d}=
\begin{cases}
    O(r^{-1-\kappa}) & \text{ALE}\\
    O(r^{-\frac{1}{2}-\kappa}) & \text{ALF}\\
    O(r^{-\frac{1}{2}-\kappa}) & \text{AF$_{\beta\ell}$}\\
\end{cases}\,.
\end{equation}
Similarly, using Section \ref{sec7:corner} near each corner $z_n \in \Gamma$ we find
\begin{equation}
\mathbf{d}= O(|\log\sin2\theta|),
\end{equation}
as $r_n \rightarrow 0$. Next, observe that $\partial\Omega_{\boldsymbol{\varsigma}}$ may be decomposed into disjoint portions contained in the three regions: $\partial\mathcal{A}_{\boldsymbol{\varsigma}}$, $\cup_{n=1}^{N}\partial B_{\varsigma_2}(z_n)$, and $\partial B_{2/\varsigma_3}$. Since $\varsigma_1 < \varepsilon_1$ we have that $\Psi_{\boldsymbol{\varepsilon}}=\Psi_o$ on $\partial\mathcal{A}_{\boldsymbol{\varsigma}}\cap \partial\Omega_{\boldsymbol{\varsigma}}$, and thus this portion of the integral \eqref{4.28} vanishes. Moreover, applying 
\eqref{4.29} and the distance function estimates, along with the asymptotics of Sections \ref{sec7:infinity} and \ref{sec7:corner}, we find that over the remaining two regions the integral tends to zero as $\varsigma_2, \varsigma_3 \rightarrow 0$. 
Therefore, with the aid of the Sobolev inequality it follows that
\begin{equation}\label{79}
\begin{split}
\mathcal{I}_{\Omega_{\boldsymbol{\varsigma}}}(\Psi_{\boldsymbol{\varepsilon}})
\geq C\left(\int_{\Omega_{\boldsymbol{\varsigma}}}\operatorname{dist}^6_{\mathbb{H}^2}(\Psi_{\boldsymbol{\varepsilon}},\Psi_{o})dx\right)^{1/3}+o(1),
\end{split}
\end{equation}
for some constant $C>0$ independent of $\boldsymbol{\varsigma}$ and $\boldsymbol{\varepsilon}$. It should be noted that error $o(1)$ is independent of $\varepsilon_1$ as well as $\varsigma_1$ and converges to zero when $\varsigma_2,\varsigma_3\to 0$. 

We will now take a series of limits to arrive at the desired conclusion. First note that the integrand on the right-hand side of \eqref{79} vanishes on $\mathcal{A}_{\boldsymbol{\varsigma}}$, so that this integral may be taken over $\tilde{B}_{2/\varsigma_3}$. Moreover, on the left-hand side, by Proposition \ref{Prop:finiteness} we may take the limit as $\varsigma_1\rightarrow 0$ to obtain the reduced energy over this same domain. Next take the liminf on both sides as $\varepsilon_1 \rightarrow 0$, and apply Lemma \ref{lemma3.4} as well as Fatou's lemma to find
\begin{equation}
\mathcal{I}_{\tilde{B}_{2/\varsigma_3}}(\Psi)\geq C\left(\int_{\tilde{B}_{2/\varsigma_3}}\operatorname{dist}^6_{\mathbb{H}^2}(\Psi,\Psi_{o})dx\right)^{1/3}+o(1),
\end{equation}
where we have used that the error terms are uniform in $\varepsilon_1$ and $\varsigma_1$. Now take the liminf on both sides, first as $\varsigma_2\rightarrow 0$ and then as $\varsigma_3\rightarrow 0$, utilizing Proposition \ref{Prop:finiteness} on the left-hand side and Fatou's lemma again on the right-hand side to obtain
\begin{equation}
\mathcal{I}(\Psi)\geq C\left(\int_{\mathbb{R}^3}\operatorname{dist}^6_{\mathbb{H}^2}(\Psi,\Psi_{o})dx\right)^{1/3}.
\end{equation}

\end{proof}

\section{The Boundary Integrals} 
\label{sec5} \setcounter{equation}{0}
\setcounter{section}{5}

The purpose of this section is to investigate the boundary integrals appearing in \eqref{Bon1}. The axis integral will be treated first. In particular, using the asymptotics of Section \ref{sec7.axis}, it is shown that the limit of this integral may be expressed in terms of logarithmic angle defects. Throughout, $(M,g)$ will denote a simply connected toric ALE or toric ALF manifold, possibly having conical singularities and corners along finite axes, and $(M,g_o)$ will denote the corresponding toric gravitational instanton sharing the same asymptotic ALE or ALF structure, and the same rod data set consisting of intervals 
$\{\Gamma_n\}_{n=1}^{N+1}$. 

\begin{lemma}\label{lem4.1}
Let $\pmb{\vartheta}^n$ and $\pmb{\vartheta}^n_o$, $n=1,\cdots,N+1$ be the logarithmic angle defects on axis rod $\Gamma_n$ for  $(M,g)$ and $(M,g_o)$, respectively, then 
\begin{equation}
\lim_{\varsigma_3\to 0}\lim_{\varsigma_2\to 0}\lim_{\varsigma_1\to 0}\mathcal{B}^{\boldsymbol{\varsigma}}_\text{axis}=-4\pi\sum_{n=1}^{N+1}\int_{\Gamma_n}(\pmb{\vartheta}^n -\pmb{\vartheta}^n_o) dz.
\end{equation}
\end{lemma}

\begin{proof}
Observe that \eqref{defX} yields  
\begin{equation}
X(\nu)=2\partial_\rho(\alpha-\alpha_o+Z)-(2\alpha-2\alpha_o-Z)\partial_\rho\log\rho+\left(V-V_o\right)\partial_\rho V_o+\left(W-W_o\right)\partial_\rho W_o,
\end{equation}
where we have used $\nu=-\partial_\rho$. Each term of $X(\nu)$ has potentially different asymptotics depending on the type of rod structure at the axis. In particular, there are three rod structure cases to consider on $\Gamma_n$: (I) $\mathbf{v}_n=(1,0)$, (II) $\mathbf{v}_n=(0,1)$, (III) $\mathbf{v}_n=(v^1_n,v^2_n)\in(\mathbb{Z}\backslash \{0\})^2$. The asymptotics of Section \ref{sec7.axis} imply  
\begin{align}
\begin{split}
\text{Case I:}&\qquad X(\nu)\rho =-2\alpha+2\alpha_o+Z+\left(V-V_o\right)+O(\rho^{\min\{1,\zeta\}}),\\
\text{Case II$_0$:}&\qquad X(\nu)\rho =-2\alpha+2\alpha_o+Z-\left(V-V_o\right)+O(\rho^{\min\{1,\zeta\}}),\\
\text{Case II$_{\beta \ell }$:}&\qquad X(\nu)\rho =-2\alpha+2\alpha_o+Z+\left(W-W_o\right)+O(\rho^{\min\{1,\zeta\}}),\\
\text{Case III:}&\qquad X(\nu)\rho =-2\alpha+2\alpha_o+Z-\left(W-W_o\right)+O(\rho),
\end{split}
\end{align}
where the leading terms have well-defined finite limits as $\rho\rightarrow 0$ away from corner points; the sum of leading terms will be denoted by $X_\text{axis}$. Moreover, the asymptotics of Sections \ref{sec7.axis} and \ref{sec7:corner} guarantee that each integral over a finite rod is well-defined and finite. The integral over semi-infinite axis rods $\Gamma_1$ and $\Gamma_{N+1}$ will be shown to vanish, due to the absence of conical singularities.
We may then write
\begin{equation}\label{int1}
\begin{split}
\lim_{\varsigma_3\to 0}\lim_{\varsigma_2\to 0}\lim_{\varsigma_1\to 0}\mathcal{B}^{\boldsymbol{\varsigma}}_\text{axis}&=\lim_{\varsigma_3\to 0}\lim_{\varsigma_2\to 0}\lim_{\varsigma_1\to 0}\int_{\partial \mathcal{A}_{\boldsymbol{\varsigma}}\cap\partial\Omega_{\boldsymbol{\varsigma}}}\!\!\!\!\!\!\!\!\!\! X(\nu) \, dA\\
&=2\pi\sum_{n=1}^{N+1}\int_{\Gamma_n}X_\text{axis}\, dz.
\end{split}
\end{equation}

Let $\mathbf{v}_n=(v^1_n,v^2_n)\in\mathbb{Z}^2$ be the rod structure for $\Gamma_n$.
Then according to the definition \eqref{coneangleaxis} of logarithmic angle defect, at interior points of $\Gamma_n$ we have 
\begin{equation}\label{re1}
\begin{split}
   e^{\pmb{\vartheta}^n}=&\lim_{\rho\to 0}\frac{\rho e^{\alpha}}{\sqrt{G_{ij}v_n^iv_n^j}}=\lim_{\rho\to 0}\frac{e^{\alpha}}{ \exp\left(\frac{1}{2}Z-\frac{1}{2}\log\rho+\frac{1}{2}\log(\Phi_{ij}v^i_nv^j_n)\right)},
\end{split}
\end{equation}
and therefore
\begin{equation}\label{regphi}
2\pmb{\vartheta}^n= \lim_{\rho\rightarrow 0}\left(  2\alpha-Z+\log\rho-\log(\Phi_{ij}v^i_nv^j_n)\right).
\end{equation}
This limit exists and is finite according to the asymptotics of Section \ref{sec7.axis}. Moreover,
this implies that in
\begin{align}
\begin{split}
\text{Case I:}&\qquad 2\pmb{\vartheta}^n=\lim_{\rho\rightarrow 0}\left(  2\alpha-Z+\log\rho-V\right),\\
\text{Case II$_0$:}&\qquad 2\pmb{\vartheta}^n= \lim_{\rho\rightarrow 0}\left( 2\alpha-Z+\log\rho+V\right),\\
\text{Case II$_{\beta \ell }$:}&\qquad 2\pmb{\vartheta}^n=\lim_{\rho\rightarrow 0}\left(  2\alpha-Z+\log\rho-W -\log(2\beta \ell)\right).
\end{split}
\end{align}
To obtain the expression in case III recall that $\textbf{v}_n\in\text{Ker } G$. Then with the help 
of asymptotics from Section \ref{sec7.axis}, upon approach to interior points of $\Gamma_n$ it holds that
\begin{equation}
v_n^1\Phi_{11}+v_n^2\Phi_{12}=O(\rho),\quad\quad \quad c(z)^{-1}\leq \rho\Phi_{11}\leq c(z),
\end{equation}
where $c(z)$ is a positive function. Next, observe that an algebraic manipulation using $\det\Phi=1$ yields
\begin{equation}
\Phi_{ij}v_n^i v_n^j=\frac{(v_n^2)^2}{\Phi_{11}}+\frac{(\Phi_{11}v_n^1 +\Phi_{12} v_n^2)^2}{\Phi_{11}}.
\end{equation}
Since $v_n^2 \neq 0$ we then have
\begin{equation}
\log\left(\Phi_{ij}v^i_nv^j_n\right)= -V-\log\cosh W +2\log |v_n^2|+O(\rho^2),
\end{equation}
and hence
\begin{equation}\label{reg3}
\text{Case III:}\qquad 2\pmb{\vartheta}^n=\lim_{\rho\rightarrow 0}\left(  2\alpha-Z+\log\rho+V+\log\cosh W -2\log |v_n^2|\right).
\end{equation}

Similar statements hold that relate the harmonic map $\Psi_o=(V_o,W_o)$ to the corresponding logarithmic angle defect $\pmb{\vartheta}^n_o$. Combining these formulas shows that
\begin{equation}
\begin{split}
X_\text{axis}&=-2\left(\pmb{\vartheta}^n-\pmb{\vartheta}^n_o\right).
\end{split}
\end{equation}
The desired result then follows from \eqref{int1}.
\end{proof}


\begin{lemma}\label{lem4.2}
The boundary integral about corners vanishes, that is 
\begin{equation}
\lim_{\varsigma_3\to 0}\lim_{\varsigma_2\to 0}\lim_{\varsigma_1\to 0}\mathcal{B}^{\boldsymbol{\varsigma}}_\text{corner}=0.
\end{equation}
\end{lemma}

\begin{proof}
Consider a corner point at the intersection of two rods $\Gamma_n \cap \Gamma_{n+1}$. Let us assume first that the rod structures on either side are $\mathbf{v}_{n}=(1,0)$ and $\mathbf{v}_{n+1}=(0,1)$.
Applying the asymptotics of Section \ref{sec7:corner} yields
\begin{equation}\label{aoijfnoainofi}
\begin{split}
r_nX(\nu)&=\underbrace{2\partial_{r_n}(\alpha-\alpha_o+Z)}_{O(r_n)}-\underbrace{(2\alpha-2\alpha_o-Z)}_{O(1)}\underbrace{\partial_{r_n}\log\rho}_{O(r_n^{-1})}\\
&\qquad+\underbrace{\left(V-V_o\right)}_{O(1)}\underbrace{\partial_{r_n} V_o}_{O(r_n)}+\underbrace{\left(W-W_o\right)}_{O(1)}\underbrace{\partial_{r_n} W_o}_{O(r_n)},
\end{split}
\end{equation}
where we have used that $\nu=-\frac{1}{r_n}\partial_{r_n}$. Since the area element induced from $\mathbb{R}^3$ on
$\partial \mathcal{B}_{\varsigma_2}(z_n)$ is given in the polar coordinates of \eqref{cornerrho} by
$\tfrac{1}{2}r_n^4\sin 2\theta d\theta d\phi$, the desired result follows. In the general case, when the rod structures are not in canonical form, an $SL(2,\mathbb{Z})$ congruence transformation may be applied to $\Phi$ and $\Phi_o$ giving a reduction back to the canonical rod structures. In this process $\alpha$, $\alpha_o$, and $Z$ remain unchanged, while the asymptotics of the remaining terms in \eqref{aoijfnoainofi} are the same. Thus, in the general case, the desired outcome is achieved.
\end{proof}

In order to treat the boundary integral at infinity, we will first provide three preliminary propositions to compute
the mass \eqref{mass.def} for metrics of the form \eqref{m1}.  

\begin{prop}\label{?}
Let $(M,g)$ be as in Theorem \ref{main.theorem}, and let $b$ denote the corresponding asymptotic model metric
as given in Section \ref{modelgeom}. If both metrics $g$ and $b$ are expressed in radial Brill coordinates, then the mass density takes the form
\begin{equation}\label{massdensity}
\begin{split}
\left(\mathrm{div}_b \mathbf{e}-d\mathrm{Tr}_b \mathbf{e}\right)(\partial_r)&= \frac{1}{2} b^{rr}b^{\theta\theta}g_{rr}\partial_rb_{\theta\theta}+\frac{1}{2}(b^{\theta\theta})^2 g_{\theta\theta}\partial_rb_{\theta\theta}-b^{\theta\theta}\partial_rg_{\theta\theta}\\
&\quad +b^{rr}g_{rr}\partial_r\log\rho -\mathrm{Tr}\left(G_b^{-1}\partial_rG\right)-\frac{1}{2}\mathrm{Tr}\left(G\partial_rG_b^{-1}\right)\\
&\quad +b^{\theta\theta}\partial_\theta g_{\theta r}+b^{\theta\theta}g_{\theta r}\partial_\theta\log\rho ,
\end{split}
\end{equation}
where $\mathbf{e}=g-b$ is the error tensor.
\end{prop}

\begin{proof}
Brill coordinates may be placed in radial form with the transformations \eqref{alerho} and \eqref{alfrho}, depending on the asymptotic structure. In these coordinates, the model metrics from Section \ref{modelgeom} may be written as
\begin{equation}
b=b_{rr}dr^2+b_{\theta\theta}d\theta^2+G_{bij}d\phi^id\phi^j,
\end{equation}
where $b_{\theta\theta}=r^2 b_{rr}$ and $\det G_b =\rho^2$.
The components of $g$ and its inverse in these coordinates are given by
\begin{equation}
\begin{split}
g_{ac}=e^{2\alpha}\delta_{2ac}+G_{ij}A^i_aA^i_c,\quad\quad g_{ij}=G_{ij}, \quad\quad g_{ia}=G_{ij}A^j_a,
\end{split}
\end{equation}
\begin{equation}
\begin{split}\label{asympt}
g^{ab}=e^{-2\alpha}\delta_2^{ac},\quad\quad g^{ij}=G^{ij}+e^{-2\alpha}\delta_2^{ac}A^i_aA^j_c,\quad\quad g^{ia}=-e^{-2\alpha}\delta_2^{ac}A^i_c
\end{split}
\end{equation}
where here and below $a,c,d=r,\theta$ and $i,j,k,l=1,2$ index the torus fiber, while
the 2-dimensional flat metric is 
\begin{equation}
\delta_2=e^{2\alpha}\frac{r^2}{p^2}\left(dr^2+r^2 d\theta^2\right)\quad \text{ALE},\quad\quad
\delta_2=e^{2\alpha}\ell^2\left(dr^2+r^2 d\theta^2\right)\quad \text{ALF and AF$_{\beta\ell}$}.
\end{equation}

We begin by computing the Christoffel symbols associated with the metric $b$:
\begin{equation}
\Gamma_{rr}^r=\frac{1}{2}b^{rr}\partial_rb_{rr},\qquad \Gamma^r_{\theta r}=\frac{1}{2}b^{rr}\partial_\theta b_{rr},\qquad \Gamma^r_{\theta\theta}=-\frac{1}{2}b^{rr}\partial_r b_{\theta\theta},\qquad \Gamma^r_{ij}=-\frac{1}{2}b^{rr}\partial_rG_{bij},
\end{equation}
\begin{equation}
\Gamma^\theta_{rr}=-\frac{1}{2}b^{\theta\theta}\partial_\theta b_{rr},\qquad \Gamma^{\theta}_{\theta\theta}=\frac{1}{2}b^{\theta\theta}\partial_\theta b_{\theta\theta},\qquad \Gamma^{\theta}_{\theta r}=\frac{1}{2}b^{\theta\theta}\partial_rb_{\theta\theta},\qquad \Gamma^{\theta}_{ij}=-\frac{1}{2}b^{\theta\theta}\partial_\theta G_{b ij},
\end{equation}
\begin{equation}
\Gamma^{r}_{ri}=\Gamma^{r}_{\theta i}=\Gamma^{\theta}_{\theta i}=\Gamma^{\theta}_{r i}=\Gamma^k_{rr}=\Gamma^{k}_{\theta\theta}=\Gamma^k_{r\theta}=\Gamma^k_{ij}=0,
\end{equation}
\begin{equation}
\Gamma^i_{rj}=\frac{1}{2}G_b^{ik}\partial_rG_{bkj},\qquad \Gamma^i_{\theta j}=\frac{1}{2}G_b^{ik}\partial_\theta G_{bkj}.
\end{equation}
Then the divergence term of the mass density becomes
\begin{align}\label{diver}
\begin{split}
\left(\text{div}_b \mathbf{e}\right)(\partial_r)&=b^{ac}\left(\partial_a g_{cr}-\Gamma^d_{ac} g_{dr}-\Gamma^d_{ar}g_{cd}\right)+b^{ij}\left(-\Gamma^c_{ij} g_{cr}-\Gamma^c_{ir}g_{jc}-\Gamma^k_{ir}g_{jk}\right)\\
&=b^{rr}\partial_rg_{rr}+b^{\theta\theta}\partial_\theta g_{\theta r}-\frac{b^{rr}}{2}\left(b^{rr}\partial_rb_{rr}-b^{\theta\theta}\partial_rb_{\theta\theta}-\mathrm{Tr}\left(G_b^{-1}\partial_rG_b\right)\right)g_{rr}\\
&\quad+\frac{1}{2}b^{\theta\theta}\mathrm{Tr}\left(G_b^{-1}\partial_\theta G_b\right) g_{\theta r}-\frac{1}{2}\left(b^{rr}\right)^2 g_{rr}\partial_rb_{rr}-\frac{1}{2}(b^{\theta\theta})^2g_{\theta\theta}\partial_rb_{\theta\theta}\\
&\quad+\frac{1}{2}\mathrm{Tr}\left(G\partial_rG_b^{-1}\right).
\end{split}
\end{align}
Similarly, for the trace term we have
\begin{align}\label{dtraceer}
\begin{split}
(d\mathrm{Tr}_b \mathbf{e})(\partial_r)&=b^{ac}\left(\partial_rg_{ac}-2\Gamma^d_{ra}g_{cd}-2\Gamma^k_{ra}g_{ck}\right)+b^{ij}\left(\partial_rg_{ij}-2\Gamma^d_{ri}g_{jd}-2\Gamma^k_{ri}g_{jk}\right)\\
&=b^{ac}\left(\partial_rg_{ac}-2\Gamma^d_{ra}g_{cd}\right)+b^{ij}\left(\partial_rg_{ij}-2\Gamma^k_{ri}g_{jk}\right)\\
&= b^{rr}\partial_rg_{rr}+b^{\theta\theta}\partial_rg_{\theta\theta}+\mathrm{Tr}\left(G^{-1}_b\partial_r G\right)
-(b^{rr})^2 g_{rr}\partial_rb_{rr}\\
&\quad - (b^{\theta\theta})^2g_{\theta\theta} \partial_r b_{\theta\theta}+\mathrm{Tr}\left(G\partial_rG_b^{-1} \right).
\end{split}
\end{align}
Subtracting this from \eqref{diver} yields the desired result.
\end{proof}

\begin{prop}
Assume the hypotheses and setting of Proposition \ref{?}. Write the torus fiber parts of the metrics $g$ and $b$ as $G=\rho e^Z\Phi$ and $G_b=\rho\Phi_b$, where $\Phi$ and $\Phi_b$ are unimodular. Let $(V,W)$ and $(V_b,W_b)$ be the parameterizations of $\Phi$ and $\Phi_b$ given by \eqref{Def:VW}, then
\begin{equation}\label{GbG1}
\begin{split}
\mathrm{Tr}\left(G_b^{-1}\partial_rG\right)&=e^Z\left(\partial_r\log\rho+\partial_r Z\right) \left(2\cosh (V-V_b)\cosh W_b\cosh W-2\sinh W_b\sinh W\right)\\
&\quad +e^Z\left(2\sinh (V-V_b)\cosh W_b\cosh W\partial_rV-2\cosh W\sinh W_b \partial_rW\right.\\&
\left. \quad+2\cosh (V-V_b)\sinh W\cosh W_b\partial_rW\right),
\end{split}
\end{equation} 
and 
\begin{equation}\label{GbG2}
\begin{split}
\mathrm{Tr}\left(G\partial_rG_b^{-1}\right)&=-e^Z(\partial_r\log\rho) \left(2\cosh (V-V_b)\cosh W_b\cosh W-2\sinh W_b\sinh W\right)\\
&\quad +e^Z\left(2\sinh (V_b-V)\cosh W_b\cosh W\partial_rV_b-2\cosh W_b\sinh W \partial_rW_b\right.\\&
\left. \quad+2\cosh (V-V_b)\sinh W_b\cosh W\partial_rW_b\right).
\end{split}
\end{equation}
\end{prop}

\begin{proof}
We will first consider the $\beta =0$ case, in which
\begin{equation}
\Phi=\begin{pmatrix}
        e^{V}\cosh W & \sinh W\\
        \sinh W & e^{-V}\cosh W
    \end{pmatrix},\qquad \Phi_b=\begin{pmatrix}
        e^{V_b}\cosh W_b & \sinh W_b\\
        \sinh W_b & e^{-V_b}\cosh W_b
    \end{pmatrix}    .
\end{equation} 
A computation shows that
\begin{align}
\begin{split}
\mathrm{Tr}\left(G_b^{-1}\partial_rG\right) &=e^Z\left(\partial_r\log\rho+\partial_rZ\right)\mathrm{Tr}\left(\Phi_b^{-1}\Phi\right)+e^Z\mathrm{Tr}\left(\Phi_b^{-1}\partial_r\Phi\right),\\
\mathrm{Tr}\left(G\partial_rG_b^{-1}\right) &=-e^Z\mathrm{Tr}\left(\Phi\Phi_b^{-1}\right)\partial_r\log\rho +e^Z\mathrm{Tr}\left(\Phi\partial_r\Phi_b^{-1}\right).
\end{split}
\end{align}
The desired result then follows from
\begin{equation}\label{GbG3}
\mathrm{Tr}\left(\Phi \Phi_b^{-1}\right)=\mathrm{Tr}\left(\Phi_b^{-1}\Phi\right)
=2\cosh (V-V_b)\cosh W_b\cosh W-2\sinh W\sinh W_b,
\end{equation}
\begin{align}\label{GbG4}
\begin{split}
\mathrm{Tr}\left(\Phi_b^{-1}\partial_r\Phi\right)
&=2\sinh (V-V_b)\cosh W_b\cosh W\partial_rV+2\cosh (V-V_b)\cosh W_b\sinh W\partial_rW\\
&\quad-2\cosh W\sinh W_b\partial_rW,
\end{split}
\end{align}
and
\begin{align}\label{GbG5}
\begin{split}
\mathrm{Tr}\left(\Phi\partial_r\Phi_b^{-1}\right)
&=-2\sinh(V-V_b)\cosh W\cosh W_b\partial_rV_b+2\cosh (V-V_b)\cosh W\sinh W_b\partial_rW_b\\
&\quad-2\cosh W_b\sinh W\partial_rW_b.
\end{split}
\end{align}

In the case $\beta\neq 0$, set \(c=\beta\ell\) and define
\begin{equation}
        P_c=
        \begin{pmatrix}
        1 & c\\
        0 & 1
        \end{pmatrix}.
\end{equation}
Then
\begin{equation}
        \Phi=P_c^T H(V,W)P_c,
        \qquad
        \Phi_b=P_c^T H(V_b,W_b)P_c,
\end{equation}
where
\begin{equation}
        H(V,W)=
        \begin{pmatrix}
        e^V\cosh W & \sinh W\\
        \sinh W & e^{-V}\cosh W
        \end{pmatrix}.
\end{equation}
Since \(P_c\) is constant
\begin{equation}
        \operatorname{Tr}(\Phi_b^{-1}\partial_r\Phi)
        =
        \operatorname{Tr}(H_b^{-1}\partial_rH),
\quad\quad\quad
        \operatorname{Tr}(\Phi\,\partial_r\Phi_b^{-1})
        =
        \operatorname{Tr}(H\,\partial_rH_b^{-1}).
\end{equation}
Therefore, the computation is identical to the \(\beta=0\) case. 
\end{proof}

\begin{prop}\label{massinfinity}
Assume the hypotheses and setting of Proposition \ref{?}. In radial Brill coordinates the mass takes the following
form in the ALE case
\begin{equation}\label{massAEALE}
\mathrm{mass}_b(M,g)=\lim_{r\to\infty}\frac{1}{4\pi}\int_{\mathcal{S}_r}\left(-2\partial_r\left(\alpha-\alpha_b+Z\right)+\left(2\alpha-2\alpha_b-Z\right)\partial_r\log\rho\right)d\mathcal{V}
\end{equation}
where $d\mathcal{V}=\frac{r^3}{2p}\sin2\theta d\theta d\phi^1 d\phi^2$, whereas in the ALF and AF$_{\beta\ell}$ cases it takes the form
\begin{equation}\label{massAFALF}
\mathrm{mass}_b(M,g)\!=\!\lim_{r\to\infty}\frac{1}{4\pi}\!\int_{\mathcal{S}_r}\!\!\left(-2\partial_r\left(\alpha-\alpha_b+\!Z\right)+\left(2\alpha-2\alpha_b-Z\right)\partial_r\log\rho-(V-V_b)\partial_r V_b\right)d\mathcal{V}
\end{equation}
where $d\mathcal{V}=\ell r^2\sin\theta d\theta d\phi^1 d\phi^2$. 
\end{prop}

\begin{proof}
Consider first the following term from \eqref{GbG1}, broken into smaller subexpressions
\begin{equation}\label{GbG1A}
\begin{split}
\mathrm{Tr}\left(G_b^{-1}\partial_rG\right)&=\underbrace{e^Z\left(\partial_r\log\rho+\partial_r Z\right)}_{I_0}\underbrace{\left(2\cosh (V-V_b)\cosh W_b\cosh W-2\sinh W_b\sinh W\right)}_{I_1}\\
&\quad +e^Z\left(\underbrace{2\sinh (V-V_b)\cosh W_b\cosh W\partial_rV}_{I_2}-\underbrace{2\cosh W\sinh W_b \partial_rW}_{I_3}\right.\\&
\left.\quad +\underbrace{2\cosh (V-V_b)\sinh W\cosh W_b\partial_rW}_{I_4}\right)
\end{split}
\end{equation}
We may use the asymptotics of Section \ref{sec7:infinity} to isolate the leading terms and estimate each of these subexpressions, namely
\begin{equation}
\begin{split}
I_0&=e^Z \left(\partial_r\log\rho+\partial_r Z\right)\\
&=\begin{cases}
(1+Z)\partial_r\log\rho+\partial_rZ+O(r^{-3-2\kappa}) & \text{for ALE}\\
(1+Z)\partial_r\log\rho+\partial_rZ+O(r^{-2-2\kappa}) & \text{for ALF and AF$_{\beta\ell}$}
\end{cases},
\end{split}
\end{equation}
while the hyperbolic distance estimates \eqref{gfobain0qhh} give
\begin{equation}
I_1=2\cosh \left(\mathrm{dist}_{\mathbb{H}^2}(\Psi,\Psi_b)\right)
=\begin{cases}
2+O(r^{-2-2\kappa}) & \text{for ALE}\\
2+O(r^{-1-2\kappa}) & \text{for ALF and AF$_{\beta\ell}$}
\end{cases}.
\end{equation} 
Furthermore
\begin{equation}
I_2=\begin{cases}
O(r^{-3-2\kappa}) & \text{for ALE}\\
2(V-V_b)\partial_rV+O(r^{-\frac{5}{2}-2\kappa}) & \text{for ALF and AF$_{\beta\ell}$}
\end{cases},
\end{equation}
and
\begin{equation}
I_4-I_3=\begin{cases}
O(r^{-3-2\kappa}) & \text{for ALE }\\
O(r^{-2-2\kappa}) & \text{for ALF and AF$_{\beta\ell}$}
\end{cases} .
\end{equation}
Therefore, we have 
\begin{equation}\label{G111}
\begin{split}
\mathrm{Tr}\left(G_b^{-1}\partial_rG\right)\!&=\!\begin{cases}
2(1\!+\!Z)\partial_r\log\rho+2\partial_rZ+O(r^{-3-2\kappa}) & \text{for ALE}\\
2(1\!+\!Z)\partial_r\log\rho+2\partial_rZ+(V\!-\!V_b)\partial_rV+O(r^{-2-2\kappa}) & \text{for ALF and AF$_{\beta\ell}$}
\end{cases}.
\end{split}
\end{equation}
Similar computations produce
\begin{equation}\label{G112}
\begin{split}
\mathrm{Tr}\left(G\partial_rG_b^{-1}\right)&=\begin{cases}
-2(1+Z)\partial_r\log\rho+O(r^{-3-2\kappa}) & \text{for ALE}\\
-2(1+Z)\partial_r\log\rho-2(V-V_b)\partial_rV_b+O(r^{-2-2\kappa}) & \text{for ALF and AF$_{\beta\ell}$}
\end{cases} .
\end{split}
\end{equation}
Also observe that
\begin{equation}
g_{\theta r}=G_{ij}A^j_rA^i_\theta=\begin{cases}
O(r^{-1-2\kappa}) & \text{for ALE}\\
O(r^{-2\kappa}) & \text{for ALF and AF$_{\beta\ell}$}
\end{cases},
\end{equation}
and thus
\begin{equation}\label{G113}
b^{\theta\theta}\partial_\theta g_{\theta r}+b^{\theta\theta}g_{\theta r}\partial_\theta\log\rho =\begin{cases}
O(r^{-3-2\kappa}) & \text{for ALE}\\
O(r^{-2-2\kappa}) & \text{for ALF and AF$_{\beta\ell}$}
\end{cases}.
\end{equation}

In the ALE case we combine \eqref{massdensity}, \eqref{G111}, \eqref{G112}, and \eqref{G113} together with
\begin{equation}
b_{rr}=r^{-2}b_{\theta\theta}=1, \quad\quad \quad
g_{rr}=r^{-2}g_{\theta\theta}=p^{-2}r^2e^{2\alpha}=e^{2\alpha-2\alpha_b},
\end{equation}
to obtain
\begin{equation}
\begin{split}
\left(\mathrm{div}_b \mathbf e-d\mathrm{Tr}_b\mathbf e\right)(\partial_r)&=-\partial_r\left(2\alpha-2\alpha_b+2Z\right)+\left(2\alpha-2\alpha_b-Z\right)\partial_r\log\rho+O_1(r^{-3-\kappa}),
\end{split}
\end{equation} 
whereas in the ALF and AF$_{\beta\ell}$ cases we use
\begin{equation}
b_{rr}=r^{-2}b_{\theta\theta}=1, \qquad\quad g_{rr}=r^{-2}g_{\theta\theta}=e^{2\alpha}\ell^2=e^{2\alpha-2\alpha_b},
\end{equation}
to find
\begin{equation}
\begin{split}
\left(\mathrm{div}_b \mathbf e-d\mathrm{Tr}_b\mathbf e\right)(\partial_r)&=-\partial_r\left(2\alpha-2\alpha_b+2Z\right)+\left(2\alpha-2\alpha_b-Z\right)\partial_r\log\rho\\
&\quad-(V-V_b)\partial_r V_b+O_1(r^{-2-\kappa}).
\end{split}
\end{equation}
\end{proof}

We are now ready to compute the flux integral of \eqref{Bon1} from the asymptotic end, in terms of the difference
of masses.
\begin{lemma}\label{LemmaIinfinity}
The boundary integral at infinity takes the form
\begin{equation}
\lim_{\varsigma_3\to 0}\lim_{\varsigma_2\to 0}\lim_{\varsigma_1\to 0}\mathcal{B}^{\boldsymbol{\varsigma}}_\infty=2\left(\mathrm{mass}_b(M,g)-\mathrm{mass}_b(M,g_o)\right)\,.
\end{equation}
\end{lemma}

\begin{proof} 
Recall from \eqref{Bon1} that 
\begin{equation}
\mathcal{B}^{\boldsymbol{\varsigma}}_\infty=\int_{\partial B_{2/\varsigma_3}\cap\partial\Omega_{\boldsymbol{\varsigma}}}\!\!\!\!\!\!\!\!\!\! X(\nu)\, dA\, ,
\end{equation} 
where $X$ is given by \eqref{defX}. Using the asymptotics in Section \ref{sec7:infinity}, for ALE geometries we have 
\begin{equation}
p^{-1}r X(\nu)=-2\partial_r\left(\alpha-\alpha_o+Z\right)+\left(2\alpha-2\alpha_o-Z\right)\partial_r\log\rho+O(r^{-3-\kappa}),
\end{equation}
with $dA=\frac{r^4}{2p^2}\sin 2\theta d\theta d\varphi$ and $\nu=pr^{-1}\partial_r$. Moreover, for ALF and AF$_{\beta\ell}$ geometries we have
\begin{equation}
\ell X(\nu)=-2\partial_r\left(\alpha-\alpha_o+Z\right)+\left(2\alpha-2\alpha_o-Z\right)\partial_r\log\rho-\left(V-V_o\right)\partial_rV_o+O(r^{-2-\kappa}),
\end{equation}with $dA=\ell^2 r^2\sin\theta d\theta d\varphi$ and $\nu=\ell^{-1}\partial_r$. 
The desired conclusion now follows by comparing with Proposition \ref{massinfinity}.
\end{proof}

\section{Proof of Theorem \ref{main.theorem}} 
\label{sec6} \setcounter{equation}{0}
\setcounter{section}{6}

By the scalar curvature equation \eqref{integralscalarcurvature}, Proposition \ref{Prop:finiteness}, and Theorem \ref{thm.convexity}, we have 
\begin{equation}\label{fpapfj-qa9hj-y09qh}
\lim_{\varsigma_3\to 0}\lim_{\varsigma_2\to 0}\lim_{\varsigma_1\to 0}\left(\mathcal{B}^{\boldsymbol{\varsigma}}_\text{axis}+\mathcal{B}^{\boldsymbol{\varsigma}}_\text{corner}+\mathcal{B}^{\boldsymbol{\varsigma}}_{\infty}\right)\geq \mathcal{I}(\Psi) \geq C\left(\int_{\mathbb{R}^3}
\operatorname{dist}_{\mathbb{H}^2}^{6}(\Psi,\Psi_{o})\,dx
\right)^{1/3}. 
\end{equation}Combining this with Lemmas \ref{lem4.1}, \ref{lem4.2}, and \ref{LemmaIinfinity} produces the following desired inequality
\begin{equation}\label{PMTinequality1}
\begin{split}
\mathrm{mass}_b\,(M,g)- \mathrm{mass}_b\,(M,{g}_o)&\geq 2\pi\sum_{n=1}^{N+1}\int_{\Gamma_n}(\pmb{\vartheta}^n -\pmb{\vartheta}^n_o) dz \,. 
\end{split}
\end{equation}
If equality holds, then the bulk integral terms of \eqref{integralscalarcurvature} yield $R=Z=0$, and $|F^i|=|dA^i|=0$, $i=1,2$. The latter equation implies, as in the proof of Proposition \ref{propA0}, that by choosing appropriate coordinates we may assume $A^i=0$ for $i=1,2$. Moreover, the right-hand side of \eqref{fpapfj-qa9hj-y09qh} shows that $\Psi=\Psi_o$. The scalar curvature formula \eqref{scal2} then shows that
\begin{equation}
\Delta_2\left(\alpha-\alpha_o\right)=0.
\end{equation} 
Furthermore, the asymptotics in Section \ref{sec7:infinity} give boundary conditions at infinity, namely $\alpha-\alpha_o=o(1)$ as $r\to\infty$. Additionally, by \eqref{regphi} we have $\alpha=\alpha_o$ on the semi-infinite rods $\Gamma_1$ and $\Gamma_{N+1}$ since conical singularities are absent at those locations. If $\alpha-\alpha_o$ is positive somewhere, then according to the maximum principle the global maximum must be achieved at a point on a finite axis rod, $p_{\mathrm{max}}\in\Gamma_i$ where $1<i<N+1$, and the maximum value is not achieved at any interior point. Moreover, $p_{\mathrm{max}}$ cannot be at a corner point since \eqref{alphar} shows that $\alpha=\alpha_o$ at such points. Thus, $p_{\mathrm{max}}$ must be in the interior of $\Gamma_i$, and according to \eqref{alphaax} the difference $\alpha-\alpha_o$ is $C^1$ at the interior of axis rods. We may then apply the Hopf lemma to find that $\partial_{\rho}(\alpha-\alpha_o)(p_{\mathrm{max}})<0$. However, this contradicts the asymptotics \eqref{alphaax}, and hence we conclude that $\alpha\leq \alpha_o$ holds globally. A similar argument yields the opposite inequality, so that $\alpha=\alpha_0$ on $\mathbb{R}^3$. Therefore, $(M,g)$ is isometric to $(M,g_o)$.



\section{Asymptotics Near Infinity, Axes, and Corners}
\label{sec7:asymptotics} \setcounter{equation}{0}
\setcounter{section}{7}

In this section, we record the asymptotic behavior in Brill coordinates of the metric and related functions of toric ALE and ALF manifolds, in a neighborhood of the asymptotic region, the axes, and corners. 

\subsection{Asymptotic end}\label{sec7:infinity}
Consider a toric ALE, ALF-$A_{k-1}$, or AF$_{\beta\ell}$ manifold $(M,g)$, with metric expressed in the Brill coordinates of Section \ref{brill}. To determine the asymptotic behavior of the metric components, and the harmonic map component functions, we will use
the explicit expression of model metrics $b_{ALE}$, $b_{ALF}$, and $b_{AF}$ from Section \ref{modelgeom}, along with the decay relations \eqref{decay12}.\medskip


\noindent\emph{ALE asymptotics}. 
The coordinate transformation between polar and cylindrical Brill coordinates in the ALE case is given by
\begin{equation}\label{alerho}
\rho=\frac{r^2}{2p}\sin 2\theta, \qquad z=\frac{r^2}{2p}\cos 2\theta,
\end{equation} 
where $r > 0$, $\theta \in (0,\pi/2)$.
As a simple example, consider the standard Euclidean metric on $\mathbb{R}^4$, expressed in polar Brill (or polar Hopf) coordinates  $(r,\theta, \phi^1, \phi^2)$ by
\begin{equation}
    \delta_4 = d r^2 + r^2 \left(d \theta^2 + \sin^2\theta (d\phi^1)^2 + \cos^2\theta (d\phi^2)^2 \right),
\end{equation} 
where $\phi^1$, $\phi^2$ each generate $2\pi-$periodic rotations. Using \eqref{alerho} with $p=1$ produces the Euclidean metric in cylindrical Brill coordinates
\begin{equation}
    \delta_4 = \frac{d\rho^2 + d z^2}{2\sqrt{\rho^2 + z^2}} + \left(\sqrt{\rho^2 + z^2} - z) (d\phi^1)^2 + (\sqrt{\rho^2 + z^2} + z) (d \phi^2)^2 \right).
\end{equation} 
Note that the conformal factor of the orbit space is not a constant as in the ALF setting. 

In the genreal case of \eqref{m1},
the decay
\begin{equation}
     |\mathring{\nabla}^l(g-b_{ALE})|_{b_{ALE}}=O(r^{-1-\kappa-l}),\quad\quad l=0,1,2,
\end{equation}
where $\kappa>0$, combined with the asymptotics at semi-infinite axes, implies the following fall-off for the metric components
 \begin{equation}
    \alpha=-\log\left(\frac{r}{p}\right)+O_1(r^{-1-\kappa}),\quad \quad \frac{q}{p}A^{1}_a+A^2_a=\frac{1}{\cos\theta}O_1(r^{-3-\kappa}),
\end{equation}
\begin{equation}
    A^1_a= \frac{1}{\sin\theta}O_1(r^{-3-\kappa}),\qquad G_{12}=\frac{q}{p}r^2\cos^2\theta\left(1+\rho^2O_1(r^{-5-\kappa})\right)+\rho^2O_1(r^{-3-\kappa}),
\end{equation}
\begin{equation}
G_{11}=r^2\left(\frac{1}{p^2}\sin^2\theta+\frac{q^2}{p^2}\cos^2\theta\right)\left(1+\rho^2 O_1(r^{-5-\kappa})\right),\qquad G_{22}= r^2\cos^2\theta\left(1+\rho^2 O_1(r^{-5-\kappa})\right).
\end{equation}
Furthermore, using the definition of the hyperbolic Fermi coordinate functions $V$, $W$ in \eqref{Def:VW}, and the function $Z$ from \eqref{def:Z}, we obtain the following asymptotics
\begin{equation}\label{ALE:VW}
 V=\frac{1}{2}\log\left(\frac{\tan^2\theta+q^2}{p^2}\right)+\rho^2O_1(r^{-5-\kappa}),\quad\sinh W=q\cot\theta\left(1+\rho^2 O_1(r^{-5-\kappa})\right)+\rho O_1(r^{-3-\kappa})\,,
\end{equation}
\begin{equation}\label{ALE:ZDvDWDZ}
 |\nabla V|^2=\frac{p^2\tan^2\theta}{r^4\left(\sin^2\theta+q^2\cos^2\theta\right)^2}+\frac{p^2\sin^2\theta}{\left(\sin^2\theta+q^2\cos^2\theta\right)^2}O(r^{-5-\kappa})+\sin^2\theta O(r^{-6-2\kappa}),
\end{equation} 
\begin{equation}\label{ALE:ZDvDWDZ2}
Z=O_1(r^{-1-\kappa}),\qquad|\nabla W|^2=\frac{p^2q^2\left(r^{-4}+O(r^{-5-\kappa})\right)}{\sin^2\theta( \sin^2\theta + q^2 \cos^2\theta)}+O(r^{-6-2\kappa}),\qquad |\nabla Z|= O(r^{-3-\kappa})\,.
\end{equation}

Now consider the corresponding toric gravitational instanton $(M,g_o)$.
By \cite[Theorem, page 314]{bando1989construction}, any ALE Ricci flat 4-manifold with $L^2$-Riemann tensor admits fourth order decay, that is 
\begin{equation}
     |\mathring{\nabla}^l(g_o-b_{ALE})|_{b_{ALE}}=O(r^{-4-l}),\quad\quad l=0,1,2.
 \end{equation}
Therefore we have 
 \begin{equation}
    \alpha_o=-\log\left(\frac{r}{p}\right)+O_1(r^{-4}),\quad  \quad G_{o12}=\frac{q}{p}r^2\cos^2\theta\left(1+\rho^2O_1(r^{-8})\right)+\rho^2O_1(r^{-6}),
\end{equation}
\begin{equation}
G_{o11}=r^2\left(\frac{1}{p^2}\sin^2\theta+\frac{q^2}{p^2}\cos^2\theta\right)\left(1+\rho^2O_1(r^{-8})\right),\qquad G_{o22}= r^2\cos^2\theta\left(1+\rho^2O_1(r^{-8})\right).
\end{equation}
Again using the definitions of $V_o$, $W_o$ from \eqref{Def:VW} and $Z$ from \eqref{def:Z}, it follows that
\begin{equation}\label{ALE:VWhatmonic}
 V_o=\frac{1}{2}\log\left(\frac{\tan^2\theta+q^2}{p^2}\right)+\rho^2O_1(r^{-8}),\qquad \sinh W_o=q\cot\theta\left(1+\rho^2O_1(r^{-8})\right)+\rho O_1(r^{-6}).
\end{equation}
Moreover, we have 
\begin{equation}
\alpha-\alpha_o=O_1(r^{-1-\kappa}),\qquad V-V_o=\rho^2 O_1(r^{-5-\kappa}),\qquad W-W_o=\rho^2 O_1(r^{-5-\kappa}).
\end{equation}
\smallskip

\noindent\emph{ALF-$A_{k-1}$ asymptotics}. 
The coordinate transformation between polar and cylindrical Brill coordinates in the ALF-$A_{k-1}$ case is given by
\begin{equation}\label{alfrho}
\rho=\ell r\sin\theta,\qquad z=\ell r\cos\theta. 
\end{equation}
The asymptotic expansion of the metric 
 \begin{equation}
     |\mathring{\nabla}^l (g-b_{ALF})|_{b_{ALF}}=O(r^{-\frac{1}{2}-\kappa-l}),\quad\quad l=0,1,2,
 \end{equation}
where $\kappa>0$, combined with the asymptotics at semi-infinite axes, implies the following fall-off for the metric components
\begin{equation}
    \alpha=-\log\ell+ O_1(r^{-\frac{1}{2}-\kappa}),\quad  A^1_a= \rho^{-1}O_1(r^{-\frac{1}{2}-\kappa}),\quad A^2_a+k\cos^2\left(\frac{\theta}{2}\right)A^1_a= O_1(r^{-\frac{1}{2}-\kappa}),
\end{equation}
\begin{equation}
G_{12}=\ell^2k \cos^2\left(\frac{\theta}{2}\right)\left(1+\rho^2 O_1(r^{-\frac{5}{2}-\kappa})\right),\qquad G_{22}=\ell^2\left(1 +\rho^2O_1(r^{-\frac{5}{2}-\kappa})\right),
\end{equation}
\begin{equation}
G_{11}=\left(r^2\sin^2\theta+\ell^2 k^2\cos^4\left(\frac{\theta}{2}\right)\right)\left(1+\rho^2 O_1(r^{-\frac{5}{2}-\kappa})\right).
\end{equation}
As above, using the definition of the hyperbolic Fermi coordinate functions $V$, $W$ in \eqref{Def:VW}, and the function $Z$ from \eqref{def:Z}, we obtain the following asymptotics
\begin{equation}\label{asyminfALF1}
    V = \frac{1}{2} \log \left[ \frac{\rho^2}{\ell^4} + \frac{k^2}{4} \left(1 + \frac{z}{\sqrt{\rho^2 + z^2}} \right)^2\right] +\rho^2O_1(r^{-\frac{5}{2}-\kappa}),\qquad |\nabla V|=\tan\left(\frac{\theta}{2}\right)O(r^{-1}), 
\end{equation}
\begin{equation}\label{asyminfALF12}
\sinh W =\frac{\ell^2 k}{\rho}\cos^2\left(\frac{\theta}{2}\right)\left(1+\rho^2O_1(r^{-\frac{5}{2}-\kappa})\right), \qquad |\nabla W| = O(r^{-2}),
\end{equation} 
\begin{equation} 
Z=O_1(r^{-\frac{1}{2}-\kappa}),\text{ }\quad\qquad\text{ }|\nabla Z|=O(r^{-\frac{3}{2}-\kappa}).
\end{equation}
The corresponding quantities associated with the toric gravitational instanton $(M,g_o)$ satisfy analogous asymptotics, as can be shown from \cite[Proposition 2.7]{Kunduri:2026xvc}.
In particular, we have 
\begin{equation}
\alpha-\alpha_o=O_1(r^{-\frac{1}{2}-\kappa}),\qquad V-V_o=\rho^2O_1(r^{-\frac{5}{2}-\kappa}),\qquad W-W_o=\rho^2O_1(r^{-\frac{5}{2}-\kappa}).
\end{equation}
\smallskip

\noindent\emph{AF$_{\beta\ell}$ asymptotics}. 
The coordinate transformation between polar and cylindrical Brill coordinates in the AF$_{\beta\ell}$ case is the same as \eqref{alfrho},
while the asymptotic expansion of the metric is given by
 \begin{equation}
     |\mathring{\nabla}^l (g-b_{AF})|_{b_{AF}}=O(r^{-\frac{1}{2}-\kappa-l}),\quad\quad l=0,1,2,
 \end{equation}
where $\kappa>0$. When combined with the asymptotics at semi-infinite axes, this implies the following fall-off for the metric components
\begin{equation}
    \alpha=-\log\ell+O_1(r^{-\frac{1}{2}-\kappa}),\quad  A^2_a=O_1(r^{-\frac{1}{2}-\kappa}),\quad A^1_a+\beta \ell A^2_a=\rho^{-1}O_1(r^{-\frac{1}{2}-\kappa}), 
\end{equation}
\begin{equation}
    G_{12}=\beta\ell r^2\sin^2\theta\left(1+ \rho^2 O_1(r^{-\frac{5}{2}-\kappa})\right),\qquad G_{11}=r^2\sin^2\theta\left(1+\rho^2 O_1(r^{-\frac{5}{2}-\kappa})\right),
\end{equation} 
\begin{equation}
G_{22}=\left(\ell^2+\beta^2\ell^2r^2\sin^2\theta\right)\left(1+\rho^2O_1(r^{-\frac{5}{2}-\kappa})\right),
\end{equation} 
or 
\begin{equation}
    G_{12}-\beta\ell G_{11}=\rho^2 O_1(r^{-\frac{3}{2}-\kappa}),\qquad G_{11}=\ell^{-2}\rho^2\left(1+\rho^2O_1(r^{-\frac{5}{2}-\kappa})\right),
\end{equation}
\begin{equation}
    G_{22}-2\beta\ell G_{12}+\beta^2\ell^2G_{11}=\ell^2+\rho^2O_1(r^{-\frac{5}{2}-\kappa}).
\end{equation}
Using the definition of the hyperbolic Fermi coordinate functions $V$, $W$ in \eqref{Def:VW}, and the function $Z$ from \eqref{def:Z}, we obtain the following asymptotics
\begin{equation}\label{asyminfAF1}
 V=\log(\ell^{-2}\rho)+\rho^2O_1(r^{-\frac{5}{2}-\kappa}),\qquad  W=\rho O_1(r^{-\frac{3}{2}-\kappa}),\qquad Z=O_1(r^{-\frac{1}{2}-\kappa}),
\end{equation}
\begin{equation}\label{asyminfAF2}
 |\nabla V|=\frac{1}{\rho}+O(r^{-\frac{3}{2}-\kappa}),\qquad |\nabla W|=O(r^{-\frac{3}{2}-\kappa}),\qquad |\nabla Z|=O(r^{-\frac{3}{2}-\kappa}).
\end{equation}
The corresponding quantities associated with the toric gravitational instanton $(M,g_o)$ satisfy analogous asymptotics, as may be derived from \cite[Section 4]{LiSun}.
In particular, we have 
\begin{equation}
\alpha-\alpha_0=O_1(r^{-\frac{1}{2}-\kappa}),\qquad V-V_o=\rho^2O_1(r^{-\frac{5}{2}-\kappa}),\qquad W-W_o=\rho^3O_1(r^{-\frac{7}{2}-\kappa})\,.
\end{equation}\smallskip

\subsection{Axes}\label{sec7.axis} 
To derive the asymptotic behavior of metric components near interior points of an axis rod $\Gamma_n$, it is convenient to distinguish between three types of rod structure $\mathbf{v}_n$. 
\medskip

\noindent\emph{Case I: $\mathbf{v}_n=(1,0)$.} In this situation the model metric expressed in Brill coordinates takes the form 
\begin{equation}
g_\mathrm{cone}=e^{2\alpha_c(z)}\left(d\rho^2+dz^2\right)+c_{11}(z)\rho^2(d\phi^1+c_{12}(z)d\phi^2)^2+c_{22}(z)(d\phi^2)^2\,,
\end{equation}
where $c_{ii}(z)>0$ for $i=1,2$. According to \eqref{ccc1},
upon approach to the interior of the rod we have
\begin{equation}
|\mathring{\nabla}^l (g-g_{\mathrm{cone}})|_{g_{\mathrm{cone}}}=O(\rho^{1+\zeta-l}),\quad\quad l=0,1,
\end{equation}
for some $\zeta> 0$.
It follows that the metric components satisfy 
\begin{equation}\label{alphaax}
\alpha=\alpha_c+O_1(\rho^{1+\zeta}),\qquad G_{11}=c_{11}(z)\rho^2+O_1(\rho^{3+\zeta}),\qquad G_{22}=c_{22}(z)+O_1(\rho^{1+\zeta}),
\end{equation}
\begin{equation}
G_{12}=c_{12}(z)c_{11}(z)\rho^2+O_1(\rho^{2+\zeta}),\qquad A^1_a=O_1(\rho^{\zeta}),\qquad A^2_a=O_1(\rho^{1+\zeta})\,.
\end{equation}
Moreover, using the relation $\Phi=\rho^{-1}e^{-Z}G$ and the fact that $\det\Phi=1$, we have 
\begin{equation}\label{phi88}
  Z=\frac{1}{2}\log\left(c_{11}c_{22}\right)+O_1(\rho^{1+\zeta}),\qquad    \Phi=\begin{pmatrix}
        \rho\sqrt{\frac{c_{11}}{c_{22}}}+O_1(\rho^{2+\zeta}) &  c_{12}\sqrt{\frac{c_{11}}{c_{22}}}\rho+O_1(\rho^{1+\zeta})\\
         c_{12}\sqrt{\frac{c_{11}}{c_{22}}}\rho+O_1(\rho^{1+\zeta}) & \rho^{-1}\sqrt{\frac{c_{22}}{c_{11}}}+O_1(\rho^\zeta)
    \end{pmatrix}\,.
\end{equation} 
Furthermore, combining this with the definitions of $V$ and $W$ in \eqref{Def:VW} produces
\begin{equation}\label{v1app1}
V=\log\left(\rho\sqrt{\frac{c_{11}}{c_{22}}}\right)+O_1(\rho^{1+\zeta}),\qquad  W=\left(c_{12}-\beta\ell\right)\sqrt{\frac{c_{11}}{c_{22}}}\rho+O_1(\rho^{1+\zeta}),
\end{equation} 
\begin{equation}\label{v1app12}
|\nabla V|^2=\rho^{-2}+\left(\partial_z\log\left(\sqrt{\frac{c_{11}}{c_{22}}}\right)\right)^2+O(\rho^{\zeta-1}),\qquad |\nabla W|^2=\frac{c_{11}}{c_{22}}\left(c_{12}-\beta\ell\right)^2+O(\rho^{\zeta}),
\end{equation}
\begin{equation}\label{g-0u-9}
\nabla V \cdot \nabla W
=
\frac{1}{\rho}\sqrt{\frac{c_{11}}{c_{22}}}\,(c_{12}-\beta\ell)
+ O(\rho^{\zeta-1}).
\end{equation}
By \cite[Proposition 4.12]{LiSun}, the harmonic map $\Psi_o =(V_o,W_o)$ admits the same asymptotics as \eqref{v1app1}--\eqref{g-0u-9}. However, note that the $\alpha_c$-component functions associated with $g$ and $g_o$ do not necessarily agree.\medskip

\noindent\emph{Case II: $\mathbf{v}_n=(0,1)$}. In this situation the model metric expressed in Brill coordinates takes the form 
\begin{equation}
g_\mathrm{cone}=e^{2\alpha_c(z)}\left(d\rho^2+dz^2\right)+ + c_{11}(z) (d\phi^1)^2 + c_{22}(z)\rho^2(d\phi^2+c_{12}(z)d\phi^1)^2 \,.
\end{equation} 
where $c_{ii}(z)>0$ for $i=1,2$. As in Case I we may use \eqref{ccc1} to find
\begin{equation}
  Z=\frac{1}{2}\log\left(c_{11}c_{22}\right)+O_1(\rho^{1+\zeta}),\qquad   \Phi=\begin{pmatrix} \rho^{-1}\sqrt{\frac{c_{11}}{c_{22}}}+O_1(\rho^\zeta)
 &  c_{12}\sqrt{\frac{c_{22}}{c_{11}}}\rho+O_1(\rho^{1+\zeta})\\
    c_{12}\sqrt{\frac{c_{22}}{c_{11}}}\rho+O_1(\rho^{1+\zeta}) &         \rho\sqrt{\frac{c_{22}}{c_{11}}}+O_1(\rho^{2+\zeta})
    \end{pmatrix}\,.
\end{equation} 
When $\beta =0$ it follows that
\begin{equation}\label{v1app2}
    V=-\log\left(\rho\sqrt{\frac{c_{22}}{c_{11}}}\right)+O_1(\rho^{1+\zeta})\,, \qquad  W=c_{12}\sqrt{\frac{c_{22}}{c_{11}}}\rho+O_1(\rho^{1+\zeta}),
\end{equation} 
\begin{equation}
    |\nabla V|^2 = \frac{1}{\rho^2}  + O(\rho^{-1 + \zeta}), \quad |\nabla W|^2 = \frac{c_{12}^2 c_{22}}{c_{11}} + O(\rho^\zeta), \quad \nabla V \cdot \nabla W = -\frac{c_{12}}{\rho} \sqrt{\frac{c_{22}}{c_{11}}} + O(\rho^{\zeta-1}),
\end{equation} 
whereas when $\beta \neq 0$ we have
\begin{equation}
    V = -\log |\beta \ell| + O_1(\rho^{1 + \zeta}), \qquad \pm W = \log \rho - 
    \log(2|\beta \ell|)+ O_1(1),
\end{equation} 
\begin{equation}\label{???}
  |\nabla V|^2 = O(\rho^{2\zeta}), \qquad  |\nabla W|^2 = \frac{1}{\rho^2} + O(\rho^{-1}), \quad \nabla V \cdot \nabla W = O(\rho^{\zeta-1}),
\end{equation}
where $\pm =\mathrm{sgn}(\beta)$.
Again, the harmonic map $\Psi_o =(V_o,W_o)$ admits the same asymptotics as \eqref{v1app2}--\eqref{???}.\medskip

\noindent\emph{Case III: $\mathbf{v}_n=(v^1_n,v^2_n)\in(\mathbb{Z}\backslash \{0\})^2$}. 
In this general situation, we may use a transformation matrix $B=(b_{kl})\in SL(2,\mathbb{Z})$ to reduce back to Case I. In particular, using \eqref{phi88} from Case I it holds that
\begin{equation}
    \begin{split}
    \Phi&=B^t\begin{pmatrix}
        \rho\sqrt{\frac{c_{11}}{c_{22}}}+O_1(\rho^{2+\zeta}) &  c_{12}\sqrt{\frac{c_{11}}{c_{22}}}\rho+O(\rho^{1+\zeta})\\
         c_{12}\sqrt{\frac{c_{11}}{c_{22}}}\rho+O(\rho^{1+\zeta}) & \rho^{-1}\sqrt{\frac{c_{22}}{c_{11}}}+O(\rho^\zeta)
    \end{pmatrix}B\, , \quad\qquad  B = \begin{pmatrix} w^2_n & -w^1_n \\ -v^2_n & v^1_n \end{pmatrix},
    \end{split}
\end{equation} 
where $w^i_n \in \mathbb{Z}$ are chosen so that $\det B = 1$. This implies that the individual components of $\Phi$ admit the expansions
\begin{equation}
\Phi_{kl}=\rho^{-1}\sqrt{\frac{c_{22}}{c_{11}}}b_{2k}b_{2l}+\rho\sqrt{\frac{c_{11}}{c_{22}}}B_k^t\begin{pmatrix}
        1 &  c_{12}\\
         c_{12} & 0    \end{pmatrix}B_l+O_1(\rho^{1+\zeta})\,,
\end{equation}
where $B_k$ is the kth column of $B$. 
Note that it may be assumed that $v^1_n+\beta\ell v^2_n \neq 0$, otherwise $\beta\ell \in \mathbb{Q}$ and we can change the torus generators to reduce back to the case when $\beta=0$. We then have
\begin{equation}\label{axiscase2}
V=\log\left(\frac{|b_{21}|}{|b_{22}-\beta\ell b_{21}|}\right)+O_1(\rho^2),\quad \pm W=-\log\rho+\log\left(2|b_{21}(b_{22}-\beta\ell b_{21})|\sqrt{\frac{c_{22}}{c_{11}}}\right)+O_1(\rho^{2}),
\end{equation} 
where $\pm =\sgn(b_{21}(b_{22}-\beta\ell b_{21}))$, and 
\begin{equation}\label{axiscase2234}
    |\nabla V|=O(\rho),\qquad \qquad\quad|\nabla W|=\frac{1}{\rho}+O(1)\,.
\end{equation} 
As before, the harmonic map $\Psi_o=(V_o,W_o)$ admits the same asymptotics as \eqref{axiscase2} and \eqref{axiscase2234}. 
\medskip

\subsection{Corners}\label{sec7:corner}
In a neighborhood of corner point occuring at height $z_n$ on the $z$-axis, the model metric $g_{\mathrm{corner}}$ expressed in Brill coordinates takes the form \eqref{metriccornercone}, and according to \eqref{cccc} we have
\begin{equation}
|\mathring{\nabla}^l (g-g_{\mathrm{corner}})|_{g_{\mathrm{corner}}}=O(r_n^{2-l}),\quad\quad l=0,1,
\end{equation}
where $r_n$ denotes the $g_{\mathrm{corner}}$-distance to the corner point and
the relation between radial and cylindrical Brill coordinates is given by
\begin{equation}\label{cornerrho}
\rho=\frac{r_n^2}{2}\sin 2\theta,\qquad z-z_n=\frac{r_n^2}{2}\cos 2\theta.
\end{equation}
It follows that
\begin{equation}\label{alphar}
\alpha=-\log(r_n)+O_1(r_n^{2}),\qquad G_{11}=c_1^2r_n^2\sin^2\theta\left(1+O_1(r_n^{2})\right),\qquad G_{12}=\rho O_1(r_n^{2})\,,
\end{equation}
\begin{equation}
G_{22}=c_2^2r_n^2\cos^2\theta\left(1+O_1(r_n^{2})\right),\qquad A^1_a=\frac{1}{\sin\theta}O_1(1),\qquad 
A^2_a =\frac{1}{\cos\theta}O_1(1).
\end{equation}
Next, using the definition of the hyperbolic Fermi coordinate functions $V$, $W$ in \eqref{Def:VW}, and the function $Z$ from \eqref{def:Z}, we obtain the following asymptotics
\begin{equation}\label{--}
V=\frac{1}{2}\log\left(\frac{c_1^2\tan\theta}{c_2^2\cot\theta+\beta^2\ell^2c_1^2\tan\theta}\right)+O_1(r_n^{2}),\qquad W=\sinh^{-1}\left(-\beta\ell\frac{c_1^2}{|c_1c_2|}\tan\theta\right)+O_1(r_n^{2})\,,
\end{equation}
\begin{equation}
|\nabla V|^2=\frac{c_2^4\cos^4\theta}{\rho^2\left(c_2^2\cos^2\theta+\beta^2\ell^2c_1^2\sin^2\theta\right)^2}+ O(r_n^{-2}),\qquad Z=\log|c_1c_2|+O_1(r_n^{2}),
\end{equation}
\begin{equation}\label{---}
|\nabla W|^2=\frac{\beta^2\ell^2c_1^2\sec^4\theta}{r_n^4\left(c_2^2+\beta^2\ell^2c_1^2\tan^2\theta\right)}+|\beta| O(r_n^{-2})+O(1), \qquad |\nabla Z|=O(1).
\end{equation}
The harmonic map $\Psi_o=(V_o,W_o)$ admits the same asymptotics as \eqref{--}--\eqref{---}, as may be derived from
\cite[Section 4.2.2]{LiSun}.

\section{Examples} 
\label{sec8} \setcounter{equation}{0}
\setcounter{section}{8}
In this section, we provide computations of the mass and other relevant quantities for some well-known explicitly known families of gravitational instantons in the three asymptotic classes. Moreover, Theorem \ref{main.theorem} will be illustrated in the context of
the Reissner-Nordstr\"{o}m manifold.

\subsection{AF\texorpdfstring{$_{\beta\ell}$}{AF betal} manifolds} 
We consider here two explicit examples of asymptotically flat gravitational instantons, with $\beta\neq 0$ and $\beta=0$, and determine their masses and harmonic map components.

\subsubsection{Kerr instanton}
The two-parameter family of Kerr instantons $(\mathbb{R}^2 \times \mathbb{S}^2, g_K)$ are AF$_{\beta\ell}$, with the smooth Ricci flat metric expressed in radial Brill coordinates as
\begin{equation}
    \begin{aligned}
    g_K &= \frac{\ell^2 f}{\Sigma}(d\phi^2 + \frac{a}{\ell} \sin^2\theta(d\phi^1 + \beta \ell d\phi^2))^2 + \frac{\sin^2\theta}{\Sigma} ((r^2 - a^2) (d\phi^1 + \beta \ell d\phi^2) - a \ell d\phi^2)^2\\
    &\quad + \Sigma \left(\frac{dr^2}{f} + d\theta^2\right),
\end{aligned} 
\end{equation} 
where $(\phi^1,\phi^2)$ are independent $2\pi-$periodic coordinates, $f = r^2 - 2 m r - a^2$, and $\Sigma = r^2 - a^2 \cos^2\theta$. The solution is parametrized by $(m, a)$ where without loss of generality we may assume that $a\geq 0$. The radial coordinate $r \in (r_+, \infty)$ where $r_+:= m + \sqrt{m^2 + a^2}$ is the real, positive root of $f$ and $\theta \in (0,\pi)$. It is convenient to eliminate the parameter $m$ using
\begin{equation}
    m = \frac{r_+^2 - a^2}{2r_+}.
\end{equation} 
Notice that positive-definiteness of the metric requires $r_+ > a$ to ensure $\Sigma > 0$.
The asymptotic geometry is characterized by $(\ell, \beta)$ where
\begin{equation}
    \beta = \frac{a}{r_+^2 - a^2}, \qquad \ell = \frac{2r_+(r_+^2 - a^2)}{r_+^2 + a^2},
\end{equation} 
and the canonical coordinates $(\rho,z)$ are related to $(r,\theta)$ by
\begin{equation}
    \rho = \ell \sqrt{f} \sin\theta, \qquad z = \ell (r - m) \cos\theta,
\end{equation} 
where $\gamma = L^2$ (c.f. \eqref{asympt}).  There are three axis rods, namely
\begin{enumerate}[label=(\roman*)]
    \item a semi-infinite rod $(-\infty, z_1]$, with $z_1 = -\ell(r_+ - m)$ and rod structure $\mathbf{v}_1 = (1,0)$;
    \item a finite rod $[z_1,z_2]$, with $z_2 = \ell(r_+ - m)$ (corresponding to $r=r_+$ and $\theta\in[0, \pi]$) having rod structure $\mathbf{v}_2=(0,1)$; 
    \item a semi-infinite rod $[z_2,\infty)$ with rod structure $\mathbf{v}_3=(1,0)$.
\end{enumerate} 
As $r\to\infty$, we can read off
\begin{equation}
    \alpha_{Kerr} = \frac{1}{2}\log\left(\frac{\Sigma}{\ell^2(f + (m^2+a^2)\sin^2\theta)}\right) = -\log \ell+\frac{m}{r}+O_1(r^{-2}),
\end{equation} 
and 
\begin{equation}
V_{Kerr}=\log \left(\frac{\rho}{\ell^2}\right) + \frac{2m}{r} + O_1(r^{-2}), \qquad W_{Kerr} = -\frac{2m a \sin\theta}{r^2} + O_1(r^{-3}).
\end{equation} 
Therefore 
\begin{equation}
\alpha_{Kerr}-\alpha_b = m/r + O_1(r^{-2}), \quad V_{Kerr}-V_b=2m/r + O_1(r^{-2}),\quad W_{Kerr} - W_b =O_1(r^{-2}).
\end{equation}
Since $\partial_r \log \rho = 1/r + m/r^2 + O(r^{-3})$, it follows from the formula \eqref{massAFALF} that the mass is
\begin{equation}
\begin{split}\label{massKerr}
\mathrm{mass}_b(M,g_K)& =4\pi m\ell =4\pi \frac{(r_+^2 - a^2)^2}{(r_+^2 + a^2)} > 0.
\end{split}
\end{equation} 
The one-parameter family of AF$_0$ Schwarzschild instantons is recovered when $a=0$, and its asymptotic $S^1$ has bounded length given by $2\pi \ell = 8 \pi m$, and thus its mass is $16\pi m^2$. 

\subsubsection{Chen--Teo instanton} 
The two-parameter family \cite{Chen:2011tc} of Chen-Teo gravitational instantons $(CP^2 \setminus S^1, g_{CT})$  are AF$_{\beta\ell}$ with a smooth Ricci flat metric given explicitly in \cite[Appendix B2]{KunduriLucietti}: 
\begin{align}
    g_{CT} & = \frac{F(x,y)}{(x-y) H(x,y)} \left(d \bar\tau + \frac{G(x,y)}{F(x,y)} d\bar\phi\right)^2 + \frac{\kappa H(x,y)}{(x-y)^3} \left( \frac{dx^2}{X(x)} - \frac{dy^2}{Y(y)} - \frac{X(x)Y(y)}{\kappa F(x,y)} d\bar\phi^2 \right),
\end{align} 
where the auxiliary angles $(\bar\tau, \bar\phi)$ are related to the canonical $2\pi-$periodic coordinates $(\phi^1, \phi^2)$ by
\begin{equation}
    \bar\tau = \frac{b_1}{k_1} \phi^2 + \frac{b_2}{k_2} \phi^1, \qquad \bar\phi = \frac{\phi^2}{k_1} + \frac{\phi^1}{k_2},
\end{equation} 
in which the explicit expressions for the constants $(b_i,k_i)$ are given in \cite[Eqs. 174-175]{KunduriLucietti} and the metric functions are given in \cite[Eq. 165]{KunduriLucietti}. The coordinates $(y,x)$ parameterize the interior of a rectangle $x_1 < y < x_2 < x < x_3$ where $x_i$ are the roots of a quartic $P(u)$ with $X(x) = P(x), Y(y) = P(y)$. There is a (twisted) AF end, not covered in this coordinate chart, which arises as $x \to x_2^+, y \to x_2^-$. The remaining functions $F(x,y), G(x,y), H(x,y)$ are bivariate polynomials of degrees 6, 8, and 3 respectively. This is a two-parameter family characterized by an overall scale parameter $\kappa > 0$ and a parameter $\xi \in (1/2, 1/\sqrt{2})$. We may pass to the standard $(r,\theta)$ chart by setting 
\begin{equation} x = x_2 - \frac{x_2 \sqrt{(1 - \nu^2)\kappa} \cos^2 \left(\frac{\theta}{2}\right)}{r}, \qquad y = x_2 + \frac{x_2 \sqrt{(1 - \nu^2) \kappa } \sin^2 \left(\frac{\theta}{2}\right)}{r}.
\end{equation}
In terms of these, the asymptotic moduli are 
\begin{equation}
    \ell = \frac{8 \sqrt{\kappa} \xi^4}{\sqrt{1 - 4 \xi^4}(2 \xi^2 - 2 \xi + 1)^2}, \qquad \beta = \frac{(1- \xi)^2\sqrt{1 - 4 \xi^4}}{2\sqrt{\kappa} \xi^2}.
\end{equation} The canonical variables are $(\rho,z)$ are then 
\begin{equation}\begin{aligned}
    \rho &= \left(\frac{b_2 - b_1}{k_1 k_2}\right) \cdot \frac{\sqrt{-X(x)Y(y)}}{(x-y)^2}, \\ z &= \left(\frac{b_2 - b_1}{k_1 k_2}\right)  \cdot \frac{2(a_0 + a_2 x y + a_4 x^2 y^2) + (x+y)(a_1 + a_3 x y)}{2(x-y)^2},
    \end{aligned} 
\end{equation} 
where $a_i$ are constants (see \cite[pg. 28]{KunduriLucietti}). A computation yields the expansions
\begin{equation}
    \begin{aligned}
    e^{2\alpha} &= \frac{1}{\ell^2}\left[ 1 + \frac{(1 + 2\xi^2)\sqrt{\kappa(1 - 4 \xi^4)}}{(1 - 2 \xi^2)r} + O(r^{-2}) \right] , \\ V  &=\log \left(\frac{\rho}{\ell^2}\right) + \frac{\sqrt{\kappa} ( 1+ 2\xi^2)^2}{\sqrt{1 - 4\xi^4} r} + O(r^{-2}).
\end{aligned}
\end{equation} 
Using $\partial_r \log \rho = \frac{1}{r}  + O(r^{-2})$ and the formula \eqref{massAFALF} the mass of the two-parameter family of Chen-Teo instantons is 
\begin{equation} \label{massCT}
\mathrm{mass}_b(M,g_{CT}) = \frac{2\pi \ell (1 + 2\xi^2)^2 \sqrt{\kappa}}{\sqrt{1 - 4 \xi^4}} > 0.
\end{equation}

\subsubsection{Reissner-Nordstr\"{o}m instanton}
This is a two-parameter family of scalar-flat Einstein-Maxwell instantons $(\mathbb{R}^2 \times S^2, g_{RN})$ with smooth metric
\begin{equation}
    g_{RN}  = \ell^2 U(r) (d\phi^2)^2 + \frac{dr^2}{U(r)} + r^2 \left(d\theta^2 + \sin^2\theta(d\phi^1)^2 \right), \qquad U(r):=1 - \frac{2m}{r} + \frac{c_1}{r^2},
\end{equation} 
where $(\phi^1, \phi^2)$ have $2\pi$-period, $r > r_+:= m + \sqrt{m^2 - c_1}$, $\theta \in (0,\pi)$, and regularity requires 
\begin{equation}
   \ell = \frac{r_+^2}{\sqrt{m^2 - c_1}}.
\end{equation} 
Note that $m$ and $c_1$ must satisfy $m^2 >c_1$. In practice, it is convenient to express the solution in terms of $r_+$ and  $c_1$, namely by using
$m = (2r_+)^{-1}(r_+^2 + c_1)$ we have 
\begin{equation}
U(r) = \frac{(r-r_+)(r - c_1 r_+^{-1})}{r^2}, \qquad \ell = \frac{2r_+^3}{r_+^2 - c_1}.
\end{equation}
Here $r_+$ is defined to be the positive root of $U(r)$, and from previous restrictions $r_+^2 > c_1$ (note that $c_1$ can have either sign).  
Furthermore, the canonical coordinates are obtained by setting 
\begin{equation}
   \rho=\ell\sqrt{r^2-2mr+c_1}\sin\theta, \qquad z=\ell\left(r-m\right)\cos\theta,
\end{equation} 
with
\begin{equation}
    \alpha_{RN} = \frac{1}{2} \log \left[\frac{r^2}{\ell^2((r-m)^2 - (m^2 - c_1) \cos^2\theta)}\right], \qquad V_{RN} = \frac{1}{2} \log \left[\frac{\rho^2}{\ell^4 U(r)^2}\right].
\end{equation}
It is straightforward to read off the rod structure associated to the solution. In particular, let $z_2=-z_1:= \ell \sqrt{m^2 - c_1} >0$, then we find
\begin{enumerate}[label=(\roman*)]
    \item a semi-infinite rod $\Gamma_1=(-\infty, z_1]$ with rod structure $\mathbf{v}_1= (1,0)$;
    \item a finite rod $\Gamma_2=[z_1, z_2]$ with rod structure $\mathbf{v}_2=(0,1)$; 
    \item a semi-infinite rod $\Gamma_3=[z_2, \infty)$ with rod structure $\mathbf{v}_3=(1,0)$.
\end{enumerate}
The mass is given by
\begin{equation}
\begin{split}
\mathrm{mass}_b(M,g)=\frac{1}{4\pi}\lim_{r\to\infty}\int_{\mathcal{S}_r}&\left(-\partial_r\left(2\alpha-2\alpha_b+2Z\right)\right.\\
&\left.\quad+\left(2\alpha-2\alpha_b-Z\right)\partial_r\log\rho-\left(V-V_b\right)\partial_rV_b\right)\ell r^2\sin\theta d\theta d\phi^1 d\phi^2 ,
\end{split}
\end{equation} 
where $\alpha_b=-\log \ell$ and $V_b=\log\rho-2\log \ell$. Then since
\begin{equation}
   \partial_r\log\rho=\frac{1}{r}+\frac{m}{r^2}+O(r^{-3}),\quad \alpha_{RN}=\alpha_b +\frac{m}{r}+O(r^{-2}), \qquad V_{RN} = V_b  + \frac{2m}{r}+O(r^{-2}),
\end{equation} 
it follows that
\begin{equation}
\begin{split}
\mathrm{mass}_b(M,g)&=\frac{1}{4\pi }\lim_{r\to\infty}\int_{\mathcal{S}_r}\left(\frac{2m}{r^2}\right)\ell r^2\sin\theta d\theta d\phi^1 d\phi^2=4\pi m\ell.
\end{split}
\end{equation} 
Notice that $\mathrm{mass}_b(M,g)$ can be negative provided that $m < 0, c_1 < 0$. 

\subsubsection{Checking Theorem \ref{main.theorem} for Reissner-Nordstr\"{o}m} 
Here we provide a simple illustration of the main theorem by using the Schwarzschild instanton as the equilibrium geometry for the Reissner-Nordstr\"{o}m manifold. The Schwarzschild instanton is chosen to have the same rod structure as the given Reissner-Nordstr\"{o}m manifold, and note that both have vanishing cone angles on the two semi-infinite rods $\Gamma_1$ and $\Gamma_3$.

\begin{lemma}\label{lemma8.1}
The logarithmic angle defect of the finite rod $\Gamma_2$ for the Reissner-Nordstr\"{o}m manifold and Schwarzschild instanton takes the form
\begin{equation}
   \pmb{\vartheta}^2_{RN}=\frac{1}{2}\log\left
   (\frac{(-z_1+\sqrt{z_1^2+\ell^2c_1})^4}{\ell^4z_1^2}\right),\qquad 
   \pmb{\vartheta}^2_{S}=\frac{1}{2}\log\left
   (16\ell^{-4}z_1
   ^2\right).
\end{equation}
\end{lemma}

\begin{proof}
The Reissner-Nordstr\"{o}m metric may be expressed as
\begin{equation}
g_{RN}=e^{2\alpha_{RN}}\left(d\rho^2+dz^2\right)+G_{RN11}(d\phi^1)^2+G_{RN22}(d\phi^2)^2,
\end{equation}
where
\begin{equation}
\alpha_{RN}=\frac{1}{2}\log\left(\frac{\left(R_{z_1}+R_{z_2}+2\sqrt{z_2^2+\ell^2c_1}\right)^2}{4\ell^2R_{z_1}R_{z_2}}\right),
\quad\quad R_{z_i}:=\sqrt{\rho^2+(z-z_i)^2},
\end{equation}
\begin{equation}
G_{RN22}=\frac{\ell^2\left(R_{z_1}+R_{z_2}\right)^2-2\ell^2(z_2^2+z_1^2)}{(R_{z_1}+R_{z_2}+2\sqrt{z_2^2+\ell^2c_1})^2}=\rho\Phi_{RN22},
\end{equation}
\begin{equation}
G_{RN11}=\rho^2\frac{\left(R_{z_1}+R_{z_2}+2\sqrt{z_2^2+\ell^2c_1}\right)^2}{\ell^2\left(R_{z_1}+R_{z_2}\right)^2-2\ell^2(z_2^2+z_1^2)}=\rho\Phi_{RN11}.
\end{equation}
By \eqref{regphi} the desired angle defect may be obtained from the quantity
\begin{equation}\label{RNregularity2}
\begin{split}
   \alpha_{RN}+\frac{1}{2}\log\rho-\frac{1}{2}\log(\Phi_{RN22})=\frac{1}{2}\log\left(\frac{\rho^2(R_{z_1}+R_{z_2}+2\sqrt{z_2^2+\ell^2c_1})^4}{4\ell^4R_{z_1}R_{z_2}\left(\left(R_{z_1}+R_{z_2}\right)^2-2(z_2^2+z_1^2)\right)}\right).
\end{split} 
\end{equation}
Observe that near the interior of $\Gamma_2$ we have
\begin{equation}
    R_{z_1}=\sqrt{\rho^2+(z-z_1)^2}=(z-z_1)+\frac{1}{2(z-z_1)}\rho^2+O(\rho^4),
\end{equation}
\begin{equation}
    R_{z_2}=\sqrt{\rho^2+(z-z_2)^2}=-(z-z_2)-\frac{1}{2(z-z_2)}\rho^2+O(\rho^4),
\end{equation} 
and hence
\begin{equation}
    \left(R_{z_1}+R_{z_2}+2\sqrt{z_2^2+L^2c_1}\right)^4=\left(z_2-z_1+2\sqrt{z_2^2+L^2c_1}\right)^4+O(\rho^2),
\end{equation}
\begin{equation}
    \begin{split}
        (R_{z_1}+R_{z_2})^2-2(z_2^2+z_1^2)=\frac{(z_2-z_1)^2}{(z_2-z)(z-z_1)}\rho^2+O(\rho^4),\qquad \text{since}\quad z_2=-z_1,
    \end{split}
\end{equation}
\begin{equation}
    R_{z_1}R_{z_2}=-(z-z_1)(z-z_2)+O(\rho^2).
\end{equation}
It follows that
\begin{equation}\label{RNregularity3}
\begin{split}
\pmb{\vartheta}^2_{RN}=\frac{1}{2}\log\left
   (\frac{(z_2-z_1+2\sqrt{z_2^2+\ell^2c_1})^4}{4\ell^4(z_2-z_1)^2}\right)
   =\frac{1}{2}\log\left
   (\frac{(-z_1+\sqrt{z_1^2+\ell^2c_1})^4}{\ell^4z_1^2}\right).
\end{split} 
\end{equation}
Moreover, setting $c_1=0$ yields the corresponding formula for the Schwarzschild instanton.
\end{proof}

To compare the Reissner-Nordstr\"om and Schwarzschild families of instanton, and verify Theorem \ref{main.theorem}, we should have both with the same rod length for $\Gamma_2$. To achieve this, the mass parameter of Schwarzschild instanton must be chosen to be $\sqrt{m^2-c_1}$, where $m$ and $c_1$ are Reissner-Nordstr\"{o}m parameters. In this setting define the following quantity
\begin{equation}
\begin{split}
\mathcal{P}(m,c_1):=\mathrm{mass}_b\,(M,g_{RN})- \mathrm{mass}_b\,(M,{g}_{S})-2\pi\sum_{n=1}^{N+1}\int_{\Gamma_n}(\pmb{\vartheta}^n -\pmb{\vartheta}^n_S) dz.
\end{split}
\end{equation}

\begin{lemma}
For all $m, c_1\in\mathbb{R}$ with $m^2 > c_1$ it holds that $\mathcal{P}(m,c_1)\geq 0$ . Moreover, equality holds if and only if $c_1=0$, in which case the Reissner-Nordstr\"{o}m manifold reduces to the Schwarzschild instanton.
\end{lemma}

\begin{proof}
Recall that $\mathrm{mass}_b\,(M,g_{RN})=4\pi m\ell$ and $\mathrm{mass}_b\,(M,g_S)=4\pi\ell\sqrt{m^2-c_1}$, and $z_2=-z_1=\ell\sqrt{m^2-c_1}$. Acording to Lemma \ref{lemma8.1} it follows that
\begin{equation}\label{RNSinequality1}
\begin{split}
(4\pi\ell)^{-1}\mathcal{P}(m,c_1)
=m-\sqrt{m^2-c_1}+2\sqrt{m^2-c_1}\log\left(\frac{2\sqrt{m^2-c_1}}{\sqrt{m^2-c_1}+m}\right).
\end{split}
\end{equation}
Define 
\begin{equation}
x\equiv\log\left(\frac{2\sqrt{m^2-c_1}}{\sqrt{m^2-c_1}+m} \right),
\end{equation}
Then we have 
\begin{equation}
 \begin{split}
        \mathcal{P}(m,c_1)&=8\pi \ell e^{-x}\sqrt{m^2-c_1}\left(1-e^x+xe^x\right).
    \end{split}
\end{equation}
Clearly, $\mathcal{P}\geq 0$ since $1-e^x+xe^x$ is decreasing for $x<0$ and increasing for $x>0$ and zero at $x=0$. Moreover, $x=0$ is equivalent to $c_1=0$.
\end{proof}

\subsection{ALF manifolds}
We present here three explicit families of ALF-$A_{k-1}$ geometries with $k=1$ so that the boundary at infinity is $S^3$; the first two are Ricci flat while the third is scalar flat. The mass of each will be computed along with other relevant quantities associated with the main theorem.

\subsubsection{Taub-NUT instanton}
The (Ricci-flat) Taub-NUT space $(\mathbb{R}^4, g_{TN})$ is a complete, ALF-$A_0$ gravitational instanton. In local coordaintes the metric is given by
\begin{equation}\label{TaubNUTmet}
\begin{aligned}
g_{TN} &= H^{-1}\ell^2 \left(d\phi^2 + \cos^2\left(\frac{\theta}{2}\right) d \phi^1 \right)^2 + H (d r^2 + r^2 d \theta^2 + r^2 \sin^2\theta (d\phi^1)^2) , \\
 H & = 1 + \frac{\ell}{2r} ,
 \end{aligned}
\end{equation} 
where $(\phi^1, \phi^2)$ are independently $2\pi-$periodic angles, $r > 0$, and $\theta \in (0,\pi)$. This is a one-parameter family of metrics parametrized by the radius $\ell$ of $S^1$ at infinity.  The associated $2\pi-$periodic generators are $\partial_{\phi^1}$ and $\partial_{\phi^2}$. We then may select
\begin{equation}
\rho = \ell r \sin \theta, \qquad z = \ell r \cos\theta,
\end{equation} 
from which one finds
\begin{equation}
    \alpha_{TN} = \frac{1}{2} \log \left( \frac{H}{\ell^2}\right).
\end{equation}
The rod structure of Taub-NUT consists of two semi-infinite rods $\Gamma_1=(-\infty,0]$ and $\Gamma_2=[0,\infty)$, with rod structures $\mathbf{v}_1=(1,0)$ and $\mathbf{v}_2 =(1,-1)$.
The asymptotic boundary is topologically $L(1,1) = S^3$ with one direction (the $S^1$ fibre) having bounded size, while the $S^2$ base grows to infinite size. Notice that $r=0$ is a corner point where both torus generators degenerate. As is well known, although the local metric has a coordinate singularity at $r =0$, this point may be included so that $g_{TN}$ extends to a smooth metric on $\mathbb{R}^4$. 
The relevant harmonic map is given by
\begin{equation}
\Phi = \frac{1}{\ell r \sin\theta} \begin{pmatrix} 
\frac{\ell^2}{H} \cos^4\left(\frac{\theta}{2}\right) + H r^2 \sin^2\theta & \frac{\ell^2}{H} \cos^2\left(\frac{\theta}{2}\right) \\  \frac{\ell^2}{H} \cos^2\left(\frac{\theta}{2}\right) & 
\frac{\ell^2}{H}
\end{pmatrix},
\end{equation} 
and
\begin{equation}
    V - V_b = \frac{\ell}{2r} + O(r^{-2}), \qquad \alpha - \alpha_b = \frac{\ell}{4r} + O(r^{-2}).
\end{equation} 
Using the formula \eqref{massAFALF}, it follows that
\begin{equation}
    \mathrm{mass}_b(M,g_{TN}) = \pi \ell^2.
\end{equation}

\subsubsection{Taub-Bolt instanton}
The one-parameter family of Euclidean Taub-Bolt gravitational instantons are Ricci-flat and ALF, with $S^3$ asymptotic boundary. In local coordinates the metric takes the form
\begin{equation}
\begin{aligned}
 g_{TB} &=  U(r) \ell^2\left(d \phi^2 + \cos^2\left(\frac{\theta}{2}\right) d \phi^1\right)^2 +  \frac{d r^2}{U(r)} + \left(r^2 - \frac{\ell^2}{16} \right) (d\theta^2 + \sin^2\theta (d \phi^1)^2), \\
 U(r) & = \frac{(2r - \ell)(8r  -\ell)}{16r^2 -\ell^2},
\end{aligned}
\end{equation} 
 where $r \in (\ell/2, \infty), \theta \in (0,\pi)$, and $(\phi^1,\phi^2)$ are independently $2\pi-$periodic angles. The metric is parameterized by $\ell>0$ which characterizes the radius of the $S^1$ in the asymptotic region.  The rod data set consists of:
\begin{enumerate}[label=(\roman*)] 
\item a semi-infinite rod characterized by $r > \ell/2$ and $\theta = \pi$, with rod structure $\mathbf{v}_1=(1,0)$, 
\item a finite rod characterized by $r = \ell/2$ and $\theta \in (0,\pi)$, with rod structure $\mathbf{v}_2=(0,1)$, 
\item a semi-infinite rod characterized by $r > \ell/2$ and $\theta =0$, with rod structure $\mathbf{v}_3=(1,-1)$. 
\end{enumerate}
The topology of Taub-Bolt is thereforfe $M=\mathbb{CP}^2 \setminus \{\mathrm{point}\}$.
Moreover, it has the same asymptotic behavior as the Taub-NUT instanton discussed above. Hence, it may be viewed as Taub-NUT with an additional finite bolt rod corresponding to a two-sphere. 
Canonical coordinates are obtained by defining
\begin{equation}
    \rho = \frac{\ell}{4} \sqrt{(2r - \ell)(8r - \ell)} \sin \theta, \qquad z = \ell\left( r- \frac{5\ell}{16} \right) \cos\theta.
\end{equation}
It follows that
\begin{equation}
    \alpha_{TB} = \frac{1}{2} \log\left( \frac{16 (16 r^2 - \ell^2)}{\ell^2\left[ (16 r - 5\ell)^2 - 9\ell^2 \cos^2\theta \right]}\right),
\end{equation} 
and the associated harmonic map matrix is given by
\begin{equation}
    \Phi = \rho^{-1} \begin{pmatrix} U(r) \ell^2 \cos^4 \left(\frac{\theta}{2}\right) + \left(r^2 - \frac{\ell^2}{16}\right) \sin^2\theta & U(r) \ell^2 \cos^2 \left(\frac{\theta}{2}\right) \\
    U(r) \ell^2 \cos^2 \left(\frac{\theta}{2}\right) & U(r) \ell^2 \end{pmatrix}.
\end{equation} 
To compute the mass, note that the appropriate asymptotic model space is Taub-NUT with the same $\ell$. We find that
\begin{equation}
    \alpha_{TB} - \alpha_b = \frac{5\ell}{16 r} + O(r^{-2}), \qquad V_{TB} - V_b = \frac{5\ell}{8r} + O(r^{-2}).
\end{equation}  
Using formula \eqref{massAFALF} then yields
\begin{equation}
    \mathrm{mass}_b(M,g_{TB}) = \frac{5 \pi \ell^2}{4}.
\end{equation}

\subsubsection{Charged Taub-Bolt instanton}\label{CTaub-bolt} 
The following two-paramter familiy of complete, ALF scalar-flat metrics can be obtained by a suitable analytic continuation of a local family of Lorentzian metrics that satisfy the Einstein-Maxwell equations. It can be thought of as a one-parameter `charged' generalization of the Ricci-flat Taub-Bolt solution in the same way Reisner-Nordstr\"{o}m contains the Schwarzschild instanton. In the standard coordinate system, the metric takes the form
\begin{equation}
    g_{CTB} = \frac{\ell^2 F(r)}{r^2 - \frac{\ell^2}{16}} \left(d \phi^2 + \cos^2 \left(\frac{\theta}{2}\right) d \phi^1 \right)^2 + \left(r^2 - \frac{\ell^2}{16}\right) \left( \frac{d r^2}{F(r)} + d \theta^2 + \sin^2\theta (d \phi^1)^2\right),
\end{equation} 
where
\begin{equation}
    F(r)= (r - r_+)\left( r - 
    \left(r_+ + \frac{\ell}{8} - \frac{2r_+^2}{\ell}\right)\right).
\end{equation} 
The solution is parameterized by the positive parameters $(r_+, \ell)$ with $r_+ > \ell/4$, where the coordinate ranges are $r > r_+$, $\theta \in (0,\pi)$, and $(\phi^1, \phi^2)$ are independently $2\pi-$periodic angles. The apparent singularity of the metric as $r \to r_+$ can be smoothly resolved by adding in a 2-sphere bolt at $r = r_+$.  Observe that $r_+$ is the largest root of $F(r)$ because $r_+ > r_+ + \ell/8 -2r_+^2/\ell$. It is straighgtforward to verify that the asymptotic geometry as $r \to \infty$ is ALF with asymptotic boundary $S^3$. 

We note that the Taub-NUT and Taub-Bolt metrics can be recovered by setting $r_+ = \ell/4$ and $r_+ = \ell/2$ respectively (in the former case, the radial coordinate $r$ must be shifted in order to recover the explicit metric $g_{TN}$ in \eqref{TaubNUTmet}). Moreover, observe that the function $\rho = \sqrt{\det G}$ is harmonic on the 2-dimensional orbit space (in this case, parameterized by $(r,\theta)$), and from this one may find the harmonic conjugate $z$ to produce canonical coordiantes 
\begin{equation}
    \rho = \sqrt{F} \ell \sin\theta, \qquad z = \left(r - \frac{\ell}{16} - r_+ + \frac{r_+^2}{\ell}\right) \ell \cos\theta. 
\end{equation}
There are three rods with rod structures $(1,0)$, $(0,1)$, and $(1,-1)$. Moreover, a computation shows that
\begin{equation}
\alpha_{CTB}=\frac{1}{2} \log \left[ \left(r^2 - \frac{\ell^2}{16}\right)\frac{256}{P^2-(\ell^2\cos\theta-16r_+^2 \cos
\theta )^2} \right]
\end{equation}
where $P=\ell^2-16r\ell+16r_+\ell-16r_+^2$, and
\begin{equation}
    V_{CTB}=\frac{1}{2}\log\left(\frac{\left(r^2 - \frac{\ell^2}{16}\right)^2\sin^2\theta+\ell^2F(r)\cos^4\frac{\theta}{2}}{\ell^2F(r)}\right),
\end{equation}  
as well as
\begin{equation}
    \alpha_{CTB}- \alpha_b = \frac{c}{r} + O((r^{-2}), \qquad V_{CTB} - V_b = \frac{2c}{r} + O((r^{-2}), \qquad c = r_+ + \frac{\ell}{16} - \frac{r_+^2}{\ell}.
\end{equation} 
Using the formula \eqref{massAFALF}, it follows that
\begin{equation}\label{massCTB}
\mathrm{mass}_b(M,g_{CTB})= 4\pi \ell c.
\end{equation} 
We point out that the mass is negative whenever $r_+>\frac{2+\sqrt{5}}{4}\,\ell$. Thus, this family provides a continuous family of smooth, complete, scalar-flat ALF manifolds with negative mass.

\subsection{ALE manifolds}
\subsubsection{Eguchi-Hanson instanton} 
This is a one-parameter family of hyperk\"ahler metrics on $M= T^* S^2$ with metric given by 
\begin{equation} \begin{aligned} 
g_{EH} &= \frac{dr^2}{f(r)} + r^2 \left(d\theta^2 + \frac{f(r)}{4}\left( d\phi^1 + 2 \cos^2\theta d\phi^2\right)^2 + \frac{\sin^2 2\theta}{4}  (d\phi^2)^2 \right) \\
& = \frac{dr^2}{f(r)} + r^2 \left( d\theta^2  + \frac{\sin^2 \theta (d\phi^1)^2}{4} + \cos^2\theta \left(d\phi^2 + \frac{d\phi^1}{2}\right)^2 - \frac{a^4}{4r^4} \left( d\phi^1 + 2 \cos^2\theta d\phi^2\right)^2 \right),
\end{aligned}
\end{equation} 
where $f(r) = 1 - a^4/r^4$ and $a > 0$; here $r>a$, $\theta\in(0,\frac{\pi}{2})$, and $(\phi^1,\phi^2)$ are independently $2\pi$-periodic. The second form of the metric exhibits clearly that the Eguchi-Hanson instanton is ALE with boundary $L(2,1)$ at infinity.  The harmonic map is 
\begin{equation}
\Phi =   \rho^{-1} \begin{pmatrix} \frac{r^2 f(r)}{4} & \frac{r^2 f(r)}{2} \cos^2 \theta \\ \frac{r^2 f(r)}{2} \cos^2 \theta &r^2 \cos^2 \theta \left(f(r)\cos^2\theta + \sin^2\theta\right) \end{pmatrix}
\end{equation} 
where the canonical coordinates are
\begin{equation}
\rho = \frac{1}{4} \sqrt{r^4 - a^4} \sin 2\theta, \qquad z = \frac{r^2}{4} \cos 2\theta.
\end{equation} 
It follows that
\begin{equation}
\alpha_{EH} = \frac{1}{2} \log \left(\frac{4r^2}{r^4 - a^4 \cos^2 2\theta}\right),
\end{equation} 
and after a computation
\begin{equation}
     V_{EH} = \frac{1}{2} \log \left(\frac{f(r)}{4 \cos^2\theta (f(r) \cos^2\theta + \sin^2\theta)} \right), \qquad W_{EH} = \text{arcsinh} \left(\frac{r^2 f(r)}{2\rho} \cos^2 \theta\right).
\end{equation} 
The rod data set consists of:
\begin{enumerate}[label=(\roman*)]
    \item a semi-infinite rod $\Gamma_1=(-\infty, -\frac{a^2}{4}]$ with rod structure $\mathbf{v}_1= (0,1)$;
    \item a finite rod $\Gamma_2=[-\frac{a^2}{4}, \frac{a^2}{4}]$ with rod structure $\mathbf{v}_2=(1,0)$; 
    \item a semi-infinite rod $\Gamma_3=[\frac{a^2}{4}, \infty)$ with rod structure $\mathbf{v}_3=(2,-1)$.
\end{enumerate}
To compute the mass observe that
\begin{equation}
\begin{aligned}
     \alpha_{EH} - \alpha_b &= O(r^{-4}), \\
V_{EH} &= -\log\left(2\cos\theta\right) - \frac{a^4 \sin^2\theta}{2r^4} + O(r^{-8}), \\
W_{EH} & = \log \left(\cot\theta + \csc\theta\right) - \frac{a^4 \cos\theta}{2 r^4} + O(r^{-8}),
\end{aligned}
 \end{equation} and hence from \eqref{massAEALE} we confirm that the Eguchi-Hanson instanton has vanishing mass. 
 
\appendix

\section{Ricci and Scalar Curvature Computations}
\label{Ricc}

Consider a Riemannian manifold $(M,g)$ of dimension $D\geq 4$ whose metric takes the following Brill form \eqref{m1} in local coordinates 
\begin{equation}\label{global}
g = \hat{g}_{ab} d x^a d x^b + G_{ij} (d \phi^i + A^i_a d x^a)(d \phi^j + A^j_b d x^b),
\end{equation} 
where $\partial_{\phi^i}$, $i=1,\ldots, D-2$ are Killing vector fields generating the $T^{D-2}$ isometry group, and $(x^1,x^2)$ are coordinates on the space transverse to the torus action.
The principal orbits of the torus action are $(D-2)$-dimensional, and \eqref{global} describes a class of cohomogeneity-two metrics. All functions appearing in the metric depend only on $(x^1,x^2)$, and the inverse metric coefficients are given by 
\begin{equation}
\begin{split}
g^{ab}=\hat{g}^{ab},\quad\quad g^{ij}=G^{ij}+\hat{g}^{ab}A^i_aA^j_b,\quad\quad g^{ia}=-\hat{g}^{ab}A^i_b .
\end{split}
\end{equation}
The components of the Ricci tensor are (see \cite[Appendix A]{Harmark} or \cite[(4.6)]{Lott}):
\begin{equation}
    \begin{split}
        R_{ij}&=-\frac{1}{2}\hat{\nabla}_a\hat{\nabla}^aG_{ij}-\frac{1}{4}\partial_a\left(\log\det G+\log\det \hat{g}\right)\partial^aG_{ij},\\
        &+\frac{1}{2}\partial^aG_{ik}G^{kl}\partial_bG_{lj}+\frac{1}{4}G_{ik}G_{jl}\hat{g}^{ac}\hat{g}^{bd}F_{ab}^kF_{cd}^l\\
        R_{ia}&=R_{ij}A^j_a+\frac{1}{2\sqrt{\det G\det\hat{g}}}\hat{g}_{ab}\partial_c\left(\sqrt{\det G\det\hat{g}}G_{ij}\hat{g}^{bd}\hat{g}^{ce}F^j_{de}\right),\\
        R_{ab}&=-R_{ij}A^i_aA^j_b+R_{ia}A^i_b+R_{ib}A^i_{a}-\frac{1}{2}\hat{g}^{cd}G_{ij}F^i_{ac}F^{j}_{bd}+\hat{R}_{ab}\\
        &-\frac{1}{2}\hat{\nabla}_a\hat{\nabla}_{b}\log\det G-\frac{1}{4}\mathrm{Tr}\left(G^{-1}\partial_aGG^{-1}\partial_bG\right),
    \end{split}
\end{equation}
where $F_{ab}^i=\partial_aA_b^i-\partial_bA_a^i$ and $\hat{\nabla}$ is the Levi-Civita connection with respect to $\hat{g}$. The scalar curvature then becomes
\begin{equation} \begin{aligned}
R 
&=\hat{R} - \hat{g}^{ab} G^{ij} \hat\nabla_a \hat\nabla_b G_{ij}  + \frac{3}{4} \hat{g}^{ab} G^{ij} \hat\nabla_a G_{jk} G^{kl} \hat \nabla_b G_{li} - \frac{1}{4} \hat{g}^{ab} G^{ij} \hat\nabla_a G_{ij} G^{kl} \hat\nabla_b G_{kl} \\ 
&\quad -\frac{1}{4} \hat{g}^{ac} \hat{g}^{bd} G_{ij} F^i_{ab} F^j_{cd}.
\end{aligned}
\end{equation} 

For simplicity, we will assume in what follows that $D=4$ and that $(M,g)$ satisfies the hypotheses of Theorem \ref{main.theorem}. So far we have not chosen a specific coordinate system on the two-dimensional orbit space. As shown in \cite{JaraczKhuri}, there is a natural set of coordinates $(x^1,x^2) = (\rho,z)$ so that the orbit space $M/T^{2}$ is parameterized by the half-plane $\{(\rho,z)\mid \rho\geq 0, z\in\mathbb{R}\}$.
Moreover, the $z$-axis parameterizing the boundary of the orbit space consists of axes $\Gamma_n$ and corner points $z_n$, where $G$ has rank 1 and 0 respectively, and in these coordinates the orbit space metric is conformally flat 
\begin{equation}\label{conformal}
\hat{g}=e^{2\alpha}\left(d\rho^2+d z^2\right),
\end{equation}
while the scalar curvature becomes
\begin{equation}\label{scalar2} \begin{aligned}
 e^{2\alpha} R &=-2\Delta_2\alpha - \left(\frac{\Delta_2\det G}{\det G}+\frac{1}{4}\mathrm{Tr}\left(G^{-1}\nabla G\right)^2-\frac{3}{4}|\nabla\log\det G|^2\right)\\
&\quad -\frac{1}{4}e^{-2\alpha} \delta_2^{ac} \delta_2^{bd} G_{ij} F^i_{ab} F^j_{cd},
\end{aligned}
\end{equation} 
where the gradient $\nabla$ and Laplacian $\Delta_2$ are with respect to the flat metric $\delta_2 = d \rho^2 + d z^2$. 
Furthermore, the condition $\text{Ric}(g)=0$ applied to \eqref{global} is equivalent to the following system on the orbit space
\begin{equation}\label{ricciflat}
    \begin{split}
      &\partial_a\left(\sqrt{\det G\det\hat{g}}\,\hat{g}^{ab}G^{ik}\partial_bG_{kj}\right)=\frac{1}{2}\sqrt{\det G\det\hat{g}}\,F^i_{ab}G_{jk}\hat{g}^{ac}\hat{g}^{bd}F^k_{cd},\\
      &\partial_a\left(\sqrt{\det G\det\hat{g}}\,G_{ij}\hat{g}^{ac}\hat{g}^{bd}F_{cd}^j\right)=0,\\
      &\hat{R}_{ab}=\frac{1}{4}\mathrm{Tr}\left(G^{-1}\partial_aG G^{-1}\partial_b G\right)+\frac{1}{2}\hat{\nabla}_a\hat{\nabla_b}\log\det G+\frac{1}{2}\hat{g}^{cd}G_{ij}F^i_{ac}F^{j}_{bd} .
    \end{split}
\end{equation} 

\begin{prop}\label{propA0}
If $(M,g)$ satisfies the hypotheses of Theorem \ref{main.theorem} and is Ricci-flat, then the two-plane distribution orthogonal to the torus generators is integrable. In particular, there exist adapted coordinates in which $A_a^i=0$.
\end{prop}

\begin{proof}
Let $\eta_i = \partial_{\phi^i}$, $i=1,2$ denote the set of mutually commuting Killing fields generating the torus symmetry, so that $\mathcal{L}_{\eta_i} g =0$ and $[\eta_i, \eta_j] =0$.  Define the smooth functions
\begin{equation}
\omega_i:= \star (\eta_1 \wedge \eta_2 \wedge \wedge \td \eta_i),
\end{equation} 
where for convenience we have used the same symbol $\eta_i$ to denote the metric dual one-forms $g(\eta_i, \cdot)$.
Using the Killing field identity $d\star d\eta_i = -2\star \text{Ric}(\eta_i)$ together with Cartan's magic formula
yields $d\omega_i =0$. Since $M$ is connected, and the toric asymptotics imply that each twist function converges to its vanishing model counterpart at infinity, we find that $\omega_i \equiv 0$, $i=1,2$. The vanishing of the twist functions implies, by Frobenius' theorem, that the two-plane distribution orthogonal to the torus orbits is integrable. Hence, on the regular set, one may choose local adapted coordinates (without changing notation) such that the vectors $\partial_{x^a}$ are tangent to the orthogonal leaves and the vectors $\partial_{\phi^i}$ generate the torus action. It follows that
$g(\partial_{x^a},\partial_{\phi^i})=0$, and hence $A^i_a=0$ in these coordinates.
\end{proof}

In the Ricci-flat case, this proposition implies that the second equation of \eqref{ricciflat} is automatically satisfied, and that by taking the trace of the first equation, $\sqrt{\det G}$ is harmonic with respect to $\hat{g}$. 
We may then identify the coordinate $\rho$ with $\sqrt{\det G}$. 

\begin{prop}\label{detG=rho}
If $(M,g)$ satisfies the hypotheses of Theorem \ref{main.theorem} and is Ricci-flat, then the Brill coordinate $\rho$ agrees with $\sqrt{\det G}$.
\end{prop}

\begin{proof} 
Both $\rho$ and $\sqrt{\det G}$ are harmonic with respect to $\hat{g}$. Moreover, considered as functions on the $\rho z$-half plane, they both vanish on the $z$-axis and satisfy $\rho-\sqrt{\det G}=o(1)$ in the asymptotic end. It then follows from the maximum principle that $\rho=\sqrt{\det G}$. 
\end{proof} 

Using the identification $\rho^2 = \det G$, the Ricci-flat equations \eqref{ricciflat} may be rewritten in Brill coordinates as
\begin{align}
      &\partial_a\left(\rho G^{ik}\partial^a G_{kj}\right)=0 ,\label{ricciflat2a} \\
      &\hat{R}_{ab}=\frac{1}{4}\mathrm{Tr}\left(G^{-1}\partial_aG G^{-1}\partial_b G\right)+ \hat{\nabla}_a\hat{\nabla}_b\log \rho, \label{ricciflat2b}
\end{align}
where $\partial^a=\delta_2^{ab}\partial_b$. These equations are triangularly decoupled: equation \eqref{ricciflat2a} is independent of the conformal factor $\alpha$, while \eqref{ricciflat2b} determines $\alpha$ once $G$ is known. As described below, equation \eqref{ricciflat2a} is an axisymmetric harmonic-map equation and implies the integrability condition for the first-order equations determining $\alpha$.

\begin{prop} 
Let $G$ be a solution of \eqref{ricciflat2a}. Then the coefficient $\alpha$ of \eqref{conformal}, and hence the orbit space metric $\hat{g}$, is determined up to a constant.
\end{prop} 

\begin{proof}  
A computation shows that the Ricci tensor of \eqref{conformal} is
\begin{equation}
\hat{R}_{ab} = -  (\delta_2)_{ab}\Delta_2 \alpha.
\end{equation}
Therefore \eqref{ricciflat2b} reduces to
\begin{align}
\begin{split}
- \hat{\Delta}_2 \alpha & = -\frac{1}{\rho^2}  - \frac{\partial_\rho \alpha}{\rho} -\frac{1}{4} \partial_\rho G^{ij} \partial_\rho G_{ij},  \\
0 & = -\frac{\partial_z \alpha}{\rho}  -\frac{1}{4} \partial_\rho G^{ij} \partial_z G_{ij}, \\
-\hat{\Delta}_2 \alpha & = \frac{\partial_\rho \alpha}{\rho}  -\frac{1}{4} \partial_z G^{ij} \partial_z G_{ij},
\end{split}
\end{align} 
which gives the conditions
\begin{equation}
\partial_z \alpha = -\frac{\rho}{4} \partial_\rho G^{ij} \partial_z G_{ij} , \qquad \partial_\rho \alpha = - \frac{1}{2\rho} + \frac{\rho}{8} \partial_z G^{ij} \partial_z G_{ij} - \frac{\rho}{8} \partial_\rho G^{ij} \partial_\rho G_{ij}.
\end{equation}
These can be written more succinctly as
\begin{equation}
\partial_z \alpha = \frac{\rho}{4} \text{Tr} ( J_z J_\rho), \qquad \partial_\rho \alpha = -\frac{1}{2\rho}+ \frac{\rho}{8} \text{Tr}( J_\rho^2- J_z^2),
\end{equation}
where $J:= G^{-1} d G$ is a matrix valued one form. Since the integrability condition for these first order equations is $\mathrm{div}_2 (\rho J)=0$, which is equivalent to \eqref{ricciflat2a}, and the orbit space is simply connected, $\alpha$ is determined up to a constant by $G$.
\end{proof}

Define the matrix $\Phi:= (\det G)^{\frac{1}{2-D}} G = \rho^{\frac{2}{2-D}} G$, and observe that $\det \Phi =1$. We compute
\begin{equation}
    \label{Hmap}
        0=\partial_a\left(\rho G^{-1}\partial^a G\right)  
        =\partial_a\left(\rho\Phi^{-1}\partial^a\Phi\right)+\tfrac{2}{D-2} I_{D-2} \Delta_2 \rho 
        =\partial_a\left(\rho\Phi^{-1}\partial^a\Phi\right),
\end{equation} 
where $I_{D-2}$ is the identity matrix. By introducing an auxiliary $2\pi$-periodic angular coordinate $\varphi$, so that $(\rho,z,\varphi)$ form standard cylindrical coordinates on $\mathbb{R}^3$ with metric 
\begin{equation}
    \delta_3 = d \rho^2 + d z^2 + \rho^2 d \varphi^2,
\end{equation} 
we obtain from \eqref{Hmap} the equation
\begin{equation}\label{Hmap1}
\text{div}_{\delta_3}\left(\Phi^{-1}\nabla \Phi\right)=0.
\end{equation} 
Thus, once a solution $\Phi$ is determined, we can obtain the matrix $G$ and by the above arguments,  $\alpha$ is determined up to an integration constant. Since $\Phi$ is a unimodular symmetric positive definite matrix of dimension $D-2$, it follows that in this setting of a cohomogeneity-two torus action with a half-plane orbit space, the Ricci-flat equations are equivalent to a harmonic map system \eqref{Hmap1} given by $\Phi: \mathbb{R}^3 \setminus\Gamma \to SL(D-2,\mathbb{R}) /SO(D-2)$. Moreover, since $\det G=0$ on the boundary of the orbit space $\Gamma$, the harmonic map $\Phi$ necessarily has singular behavior on the axis that encodes the rod structure of $M$.

\bibliography{ALEALF}

\end{document}